\documentclass[a4paper,11pt]{article}
\usepackage{graphics,color}
\usepackage{amsmath}
\usepackage{amssymb}
\usepackage{mathrsfs}
\usepackage{cite}
\usepackage{verbatim}
\usepackage{float}
\usepackage{graphicx}
\usepackage{amsthm}
\usepackage{textcomp}
\usepackage{subfig}
\usepackage{enumerate}
\usepackage{hyperref}

\newcommand{\blue}{\color{blue}}

\numberwithin{equation}{section}
\mathchardef\emptyset="001F

\newtheorem{Theorem}{Theorem}[section]
\newtheorem{Definition}[Theorem]{Definition}
\newtheorem{Proposition}[Theorem]{Proposition}
\newtheorem{Corollary}[Theorem]{Corollary}
\newtheorem{Lemma}[Theorem]{Lemma}

\newcommand{\nada}[1]{}

\newcommand{\ab}{[a,b]}
\newcommand{\affinesurface}{X^{{\rm aff}}}

\newcommand{\affinesurfaceindice}{{\affinesurface}_{\!\!\!\!\!\!\!(\indiceaffine)}}
\newcommand{\affinesurfacelastcomp}{\widehat X}

\newcommand{\concom}{\Omega}
\newcommand{\currvk}{{\mathcal V}_k}
\newcommand{\currvkRi}{{\mathcal V}_{k,i}}
\newcommand{\currvkRj}{{\mathcal V}_{k,j}}

\newcommand{\dom}{{\rm Dom}}
\newcommand{\eps}{\varepsilon}
\newcommand{\globalminimalmasscurrent}{\mathcal M}
\newcommand{\grad}{\nabla}

\newcommand{\graphofmapvk}{V_k}
\newcommand{\graphofmapu}{U}
\newcommand{\homogeneousmap}{\varphi}
\newcommand{\indice}{j}
\newcommand{\indiceaffine}{\ell}
\newcommand{\indiceP}{k}
\newcommand{\intervallounitario}{I}
\newcommand{\intervalloparametri}{I}
\newcommand{\jump}{S}
\newcommand{\jumpset}{\Sigma}

\newcommand{\mappa}{u}
\newcommand{\mappaaffine}{\affinesurface}

\newcommand{\minimalmasscurrent}{M}
\newcommand{\mres}{\mathbin{\vrule height 1.6ex depth 0pt width
0.13ex\vrule height 0.13ex depth 0pt width 1.3ex}}
\newcommand{\N}{\numberset{N}} 
\newcommand{\newpar}{\alpha}
\newcommand{\numberset}{\mathbb}
\newcommand{\Om}{\Omega} 
\newcommand{\R}{\numberset{R}} 

\newcommand{\rettangolo}{R} 
\newcommand{\repa}{h} 
\newcommand{\rilP}{\overline P} 
 
\newcommand{\secondavariabile}{\sigma}
\newcommand{\secondavariabileinR}{\sigma}
\newcommand{\segmentoingrassato}{[a,b] \times [-\eps,\eps]}

\newcommand{\Suno}{{\mathbb S}^1}

\newcommand{\unitcircle}{\Suno}

\theoremstyle{definition}
\newtheorem{Remark}[Theorem]{Remark}
\newtheorem{Example}[Theorem]{Example}

 \title{Relaxed area of graphs of piecewise Lipschitz maps 
in the strict $BV$-convergence 
}
\author{
Giovanni Bellettini\footnote{
Dipartimento di Ingegneria dell'Informazione e Scienze Matematiche, Universit\`a di Siena, 53100 Siena, Italy,
and International Centre for Theoretical Physics ICTP,
Mathematics Section, 34151 Trieste, Italy.
E-mail: bellettini@diism.unisi.it
                      }\and
Simone Carano\footnote{
Area of Mathematical Analysis, Modelling, and Applications,
             Scuola Internazionale Superiore di Studi Avanzati ``SISSA'',
Via Bonomea, 265 - 34136 Trieste, Italy. E-mail: scarano@sissa.it
                         }
                          \and
Riccardo Scala\footnote{
Dipartimento di Ingegneria dell'Informazione e Scienze Matematiche, Universit\`a di Siena, 53100 Siena, Italy.
E-mail: riccardo.scala@unisi.it
}
}

\begin{document}
\maketitle

\begin{abstract}
We compute the relaxed Cartesian
area in the strict $BV$-convergence
on a class of piecewise Lipschitz maps from the plane to the plane, 
having jump 
made of several curves allowed to meet 
at a finite number of 
junction points. We show that the domain of this
relaxed  area 
is strictly contained in the domain of the classical $L^1$-relaxed 
area.
\end{abstract}


\vspace{2mm}

\noindent {\bf AMS (MOS) 2022 subject clas\-si\-fi\-ca\-tion:} 49J45, 49Q05, 49Q15, 49Q20, 28A75.

\section{Introduction}\label{sec:introduction}
Let $\Omega\subset\R^2$ be an open bounded set. 
Given $v\in C^1(\Om;\R^2)$, 
the area functional  is defined as 
\begin{align}\label{area_functional}
	\mathcal A(v,\Om):=\int_\Om\sqrt{1+|\nabla v|^2+|Jv|^2}~dx
=\int_\Om|\mathcal M(\nabla v)|~dx,
\end{align}
 where $\mathcal M(\nabla v)=(1,\nabla v_1,\grad v_2,Jv)$ and $Jv=\frac{\partial v_1}{\partial x_1}\frac{\partial v_2}{\partial x_2}-\frac{\partial v_2}{\partial x_1}\frac{\partial v_1}{\partial x_2}$
is the Jacobian determinant of $v$. 
The value $\mathcal A(v,\Om)$ 
is  the $2$-dimensional Hausdorff measure of the graph 
$$G_v:=\{(x,y)\in \Om\times\R^2:y=v(x)\}
$$
of $v$. In order to extend the area functional to a more general class 
of maps one is led to consider the relaxation of \eqref{area_functional}: Namely, for all $u\in L^1(\Om;\R^2)$ one 
chooses a convergence, for instance 
the $L^1$-convergence, and sets
\begin{align}\label{area_functional_rel}
	\overline{\mathcal A}_{L^1}(u,\Om):=\inf\left\{\liminf_{k\rightarrow +\infty}\mathcal A(u_k,\Om),\;u_k\in C^1(\Om;\R^2),\;u_k\rightarrow u\text{ in }L^1(\Om;\R^2)\right\}.
\end{align}
In contrast with the case of real valued maps, for which the 
$L^1$-relaxed area 
is well-understood, in higher dimension, including the case of $\R^2$-valued 
maps considered here,
the analysis of $\overline{\mathcal A}_{L^1}$ has been shown to be very challenging and a lot of questions 
remain open. For instance, it is known that the domain 
$\dom(\overline{\mathcal A}_{L^1}(\cdot, \Om))$ of $\overline{\mathcal A}_{L^1}(\cdot, \Om)$ 
is strictly included in $BV(\Om;\R^2)$, 
but its complete description 
is, so far, not available. 
The main difficulty to treat $\overline{\mathcal A}_{L^1}$ is due to its non-local behaviour: Indeed, for general maps $u$ with the only exception of very trivial cases, 
the set function 
$E \subseteq \Om \mapsto \overline{\mathcal A}_{L^1}(u,E)$
is not subadditive, and this excludes to represent
\eqref{area_functional_rel} in integral form.
As a consequence,
the explicit value of $\overline{\mathcal A}_{L^1}(u,\Om)$ 
is, at the moment,
known only for very specific non-smooth maps
 $u$ enjoying a high degree of symmetry
\cite{BP,S,BES}.

A useful simplification in the relaxation analysis of $\mathcal A$ is to 
consider some variants of  \eqref{area_functional_rel}, for example 
modifying the convergence of $v_k$ to $u$ (see \cite{BePaTe:15,BePaTe:16,DP,GMS3}). 
Even if the $L^1$-convergence seems to be 
natural also with respect to the application to the non-parametric Plateau problem, one can replace the $L^1$-topology with different ones. In some recent works \cite{Mc1,BCS}, instead of relaxing with respect to the $L^1$-topology, the authors have considered relaxation with respect to the strict convergence in $BV(\Om;\R^2)$ (shortly $BV$-relaxed area). 
Namely, one defines
\begin{align}\label{area_functional_rel2}
	\overline{\mathcal A}_{BV}(u,\Om):=\inf\left\{\liminf_{k\rightarrow +\infty}\mathcal{A}(u_k,\Om),\;u_k\in C^1(\Om;\R^2),\;u_k\rightarrow u\text{ strictly }BV(\Om;\R^2)\right\}.
\end{align}
Although the analysis of $\overline{\mathcal A}_{BV}$ seems quite more treatable, a complete picture and description of its behaviour is still missing. It is straightforward that for any $u\in BV(\Om;\R^2)$
$$\overline{\mathcal A}_{BV}(u, \Om)\geq \overline{\mathcal A}_{L^1}(u, \Om),$$
so $\dom(\overline{\mathcal A}_{BV}(\cdot, \Om))\subset \dom(\overline{\mathcal A}_{L^1}(\cdot, \Om))$, and the inclusion is strict as Example \ref{tripunto infinito} below shows.

Strictly related to the area functional is the Jacobian total variation functional, namely
$$TVJ(v, \Om):=\int_\Om|Jv|~dx,$$
valid for all $v\in C^1(\Om;\R^2)$.
Also in this case, to extend $TVJ$ to a larger class of functions, a relaxation procedure is in order. However, 
the choice of the $L^1$-convergence is in some cases not interesting: for instance, if $u\in W^{1,1}(\Omega;\Suno)$, with $\Omega$ simply connected, the corresponding relaxed functional trivializes and becomes constantly null (see \cite[Cor. 5]{BMP}). 
On the other hand, the notion of strict convergence in $BV$ gives rise to a nontrivial relaxed functional which shows to play a crucial 
role in the analysis of $\overline{\mathcal A}_{BV}$. 
Specifically, for $u\in BV(\Omega;\R^2)$ we consider 
\begin{align}\label{TVJ_functional_rel}
	\overline{TVJ}_{BV}(u, \Om):=\inf\left\{\liminf_{k\rightarrow +\infty}TVJ(v_k, \Om),\;v_k\in C^1(\Om;\R^2),\;v_k\rightarrow u\text{ strictly }BV(\Om;\R^2)\right\}.
\end{align}

In the present paper we compute the value of $\overline{\mathcal A}_{BV}(u, \Om)$ for some particular 
piecewise Lipschitz 
maps $u$ which are allowed to jump on curves in turn meeting at junction points.
We refer to Definition \ref{def_MLM} for the details on these maps, and we 
summarize here their features: 
 Let $\Om\subset\R^2$ be a bounded open set of class $C^1$ and $\{\Om_k\}_{k=1,\ldots,N}$ a finite partition of $\Om$ made of 
Lipschitz sets. Suppose that
$\jumpset:=\cup_{k=1}^N
\partial\Om_k$ is the support of a finite family of $C^2$-curves $\alpha_\indiceaffine:
\overline 
\intervalloparametri_\indiceaffine
\rightarrow\overline \Om$, $\indiceaffine
=1,\dots,n$, $\intervalloparametri_\indiceaffine = (a_\indiceaffine, b_\indiceaffine)$. 
 We suppose that the curves $\alpha_\indiceaffine$, 
arc-length parametrized on $\overline \intervalloparametri_\indiceaffine$,
are injective on $\intervalloparametri_\indiceaffine$, 
$\alpha_\indiceaffine(\intervalloparametri_\indiceaffine)
\subset\Om$, and of class $C^2$ 
up to $a_\indiceaffine$ and $b_\indiceaffine$ 
(namely $\dot\alpha_\indiceaffine$ and $\ddot\alpha_\indiceaffine$ 
are continuous on $\overline \intervalloparametri_\indiceaffine$). 
Furthermore, 
we assume that $\alpha_\indiceaffine
(\intervalloparametri_\indiceaffine)$ and $\alpha_\indiceaffine(\intervalloparametri_h)$, 
for $\indiceaffine\neq h$, 
may intersect only at the endpoints. 
Endpoints of $\alpha_l$ are allowed
to belong to $\partial\Om$, 
and we assume such endpoints to be distinct for different curves.
 
 A map $u\in BV(\Om;\R^2)$ 
is called  piecewise Lipschitz if its
restriction 
 to any $\Om_k$ is Lipschitz. Notice that if $p_i$ is a junction point 
and $\concom^i_k$ ($k=1,\dots, N_i$) 
are the connected components of 
$\Om\setminus \jumpset$ having $p_i$ as boundary point, then there exists the limit 
$\beta^i_k:=\lim_{\substack{x\rightarrow p_i\\x\in \concom^i_k}}u(x)$.
 For the sake of simplicity, we assume that the 
enumeration $k=1,\dots, N_i$ respects the counterclockwise 
order of $\concom^i_k$'s around $p_i$.
 
 To introduce our main result, we have to consider also a planar Plateau problem for Lipschitz curves: 
Given a Lipschitz curve $\homogeneousmap:
\unitcircle = \partial B_1
\rightarrow \R^2$ we consider the quantity
 \begin{align}\label{Plateau}
 	P(\homogeneousmap):=\inf\left\{\int_{B_1}|Jv|~dx:
 \,v\in \mathrm{Lip}(B_1;\R^2): v_{|\partial B_1}=
 	\homogeneousmap
 	\right\}.
 \end{align}
 For all $i=1,\dots,m$ 
we denote by $\widetilde\gamma^i$ a Lipschitz curve which 
parametrizes on $\Suno$ the polygon in $\R^2$ with 
vertices $\beta^i_1,\beta^i_2,\dots,\beta^i_{N_i}$, in the order (see Fig. \ref{curve_gamma}). 
Notice carefully that this 
curve may self-intersects. Also,
$P(\varphi)$ is invariant under reparametrizations of 
$\varphi$ (Proposition \ref{independence_of_P}). 
Finally, set $\intervallounitario=[0,1]$.  
The main result of the paper is the following
\begin{Theorem}[\textbf{Relaxation for piecewise Lipschitz
maps}]\label{teo:Lip_npoint}
	Let $u:\Om\rightarrow\R^2$ be a piecewise Lipschitz 
map. Then
	\begin{equation}\label{eq:ril_area_gen}
		\overline{\mathcal{A}}_{BV}(u, \Om)=\int_{\Om\setminus \jumpset}
|\mathcal M(\nabla u)|~dx+\
\sum_{
\indiceaffine
=1}^n\int_{[a_\indiceaffine,b_\indiceaffine]\times \intervallounitario}|\partial_t 
\affinesurfaceindice
\wedge
\partial_s 
\affinesurfaceindice
|dtds+\sum_{i=1}^m P(\widetilde\gamma^i),
	\end{equation}
where, for any $\indiceaffine =1,\dots, n$,
\begin{equation}\label{eq:affine_surface_l}
	\affinesurfaceindice
	(t,s):=(t,su_\indiceaffine^+(t)+(1-s)u_\indiceaffine^-(t))
\qquad \forall (t,s)\in[a_\indiceaffine,b_\indiceaffine]\times \intervallounitario,
\end{equation}
and $u^\pm_\indiceaffine$ are the traces of $u$ on the 
support $\alpha_\indiceaffine(I_\indiceaffine)$ 
of $\alpha_\indiceaffine$.
\end{Theorem}

One of the 
main features of expression \eqref{eq:ril_area_gen}
is the presence of {\it two} singular contributions:
 a $0$-dimensional term due to the concentration of the Jacobian determinants of a 
recovery sequence $(v_k)$ for $\overline{\mathcal{A}}_{BV}(u, 
\Om)$ around the junction points (namely, 
the term involving the minimum of the Plateau problems $P(\widetilde\gamma^i)$), and a $1$-dimensional term, which essentially takes into account the concentration of the gradients and of the Jacobian determinants of $v_k$ 
along the jump
 set $\jumpset$. So, we can interpret $\eqref{eq:ril_area_gen}$ as a non-trivial generalization of 
\cite[Theorem 1.3]{BCS}, valid 
for the triple point map $u_T$ (see also Theorem \ref{area n-ple point}), 
and of 
\cite[Theorem 1.1]{BCS}, valid for $0$-homogeneous maps of 
the form $\varphi\left({x}/{|x|}\right)$ with $\varphi:\Suno\rightarrow\Suno$ Lipschitz. 
Indeed, in the first case the $1$-dimensional term was simply 
the total variation of $u_T$ (consisting of the area
of three vertical walls over $\Sigma$) 
and the $0$-dimensional one was the area of the 
target triangle, which is a 
trivial minimum of \eqref{Plateau}, while in the second 
case we had no $1$-dimensional contribution and the $0$-dimensional 
one was the solution of \eqref{Plateau} with this 
special $\varphi$, that reduces to $P(\varphi)=\pi|\mathrm{deg}(\varphi)|$. 
In other words, the relaxed area of a more general map $u$ 
as in \eqref{teo:Lip_npoint} is still a measure (if we regard it as a 
function of $\Omega$), which has the same dimensional structure, but with a 
more involved and rich expression. 

We observe that for this special kind 
of maps it always holds $\overline {\mathcal A}_{BV}(u,\Om)<+\infty$, 
because the contributions of the Plateau problem 
$P(\widetilde\gamma^i)$ is always 
finite, since one can construct a Lipschitz competitor for \eqref{Plateau}.  
On the other hand, the presence of a 
finite number of junction points is crucial, because, as Example \ref{tripunto infinito} shows, we can build a 
piecewise constant map whose $BV$-relaxed area is infinite. It is here remarkable that the same map can be seen to have 
finite $L^1$-relaxed area (compare with \eqref{eq:L1_ril_fin}). This 
in particular shows the proper inclusion
 $${\rm Dom}(\overline{\mathcal A}_{BV}(\cdot, \Om))
\subsetneq {\rm Dom}(\overline{\mathcal A}_{L^1}(\cdot, \Om)).$$

We divide the proof of Theorem \ref{teo:Lip_npoint} in 
several steps, and in particular we first focus on the 
relaxation on piecewise Lipschitz 
maps $u$ without junction points. In this 
case we show in Corollary \ref{cor_multicurves}  (consequence of Theorems \ref{teo:main_thm_curve} and \ref{teo_main_curve2}) that the 
relaxation provides as singular contribution the integral over the 
jump set $\jump_u$ of $u$ of the area spanned by 
the affine map $\affinesurface$.  The main issue is the proof of Proposition \ref{prop:lower_bound}, 
the lower bound for the relaxed area 
of maps jumping on the central orizontal segment of the rectangle $R=[a,b]\times[-1,1]$.
 Here, we need to use 
some tools from the theory of integer multiplicity currents, in particular
slicing arguments and the isoperimetric inequality, in order to show that over 
the jump segment the graph of the elements of an 
approximating smooth sequence 
$(v_k)$ have area bounded below by the area of 
$\affinesurface$.
 The properties of the strict convergence 
(Lemmas \ref{lem:inheritance} and 
\ref{strict implies uniform 2}) enter at the level of vertical slices of the graph of $v_k$ 
in a neighbourhood of the jump segment, but these results only are not enough to pass to the limit in the area of the graph of $v_k$.
 For this purpose, the idea is to make a decomposition of the graph of $v_k$ and of the surface $\affinesurface$ in several tiny strips, and notice that, when the number of these strips is very high, the boundaries of these two little pieces of surfaces are uniformly 
close together, as a consequence of the strict convergence and, at the same time, the strips which decompose $\affinesurface$ are very close to a minimal mass current having the same boundary.

In \cite{BePaTe:16}, the authors 
compute the relaxed area 
$\overline{\mathcal{A}}_{L^\infty}(u, \Om)$ with respect to the
local uniform convergence out of the jump, 
for $u$ as in Proposition \ref{prop:lower_bound}. They obtain, as singular 
contribution, the area of the minimal 
semicartesian\footnote{A map having the identity
as the first component.} surface 
spanning the graphs of the two traces. 
In particular, since $\affinesurface$ is semicartesian and spans graph$(u^\pm)$ as well (see \cite[Definition 2.4]{BePaTe:16}), we have
$
\overline{\mathcal{A}}_{L^\infty}(u, R)\leq\overline{\mathcal{A}}_{BV}(u, R).
$
In general, this inequality holds strictly, 
even if graph$(u^\pm)$ are coplanar. We can find an example in \cite[Remark 8.5]{BePaTe:16}, where one can notice that in order to minimize the area of the spanning surface, the 
approximating sequence needs
not keep the 
total variation of the limit map, which instead is forced to be preserved under strict convergence.
Moreover, it is important to notice that $\overline{\mathcal{A}}_{L^\infty}(u, \,\cdot)$ is not subadditive (see \cite[Thm. 8.1]{BePaTe:16}), while $\overline{\mathcal{A}}_{BV}(u, \cdot)$ is clearly a measure. 

In a second step we instead consider the case of maps $u$ which are 
piecewise constant but whose jump might have junction points. 
Specifically, in Theorem \ref{area n-ple point} we see that 
the relaxation on  a $n$-uple point map (i.e., whose
jump  consists of $n$ radii of the same ball $B_r(0)$) 
provides as singular contribution, besides the total variation of $u$, 
the number  $P(\widetilde\gamma)$, where $\widetilde\gamma$ is
the piecewise affine curve which parametrizes the perimeter of the polygon whose vertices are the values of $u$ around $0$.

Finally, in Section \ref{sec:Lip_npoint}, we 
use Corollary \ref{cor_multicurves} and Theorem \ref{area n-ple point} to complete the proof of Theorem \ref{teo:Lip_npoint}. 

We point out that, to our best knowledge, 
 it is not yet known, in general, whether
the $BV$-relaxed area is subadditive if 
considered as a set function and, further, 
if it gives rise to a measure. We expect such a 
subadditivity
for $BV$-maps $u$ from the plane to the plane, being motivated by relevant examples with explicit computations, and also because of the presence of a unique cartesian current with minimal completely vertical lifting associated to $u$ (as recently shown in \cite{Mc2}). Unfortunately, this uniqueness result fails in higher codimension, 
where in addition 
we have less explicit examples.

\section{Preliminaries}\label{sec:preliminaries}
We start by collecting some tools
needed in the proof of  the main theorems. 
For an integer $M \geq 2$,
set $\mathbb S^{M-1} := \{x \in \R^M : \vert x\vert =1\}$.
In what follows, $\Om \subset \R^2$ is a bounded open set.

\subsection{Some consequences of the strict convergence}

\begin{Theorem}[\textbf{Reshetnyak}]\label{Reshetnyak}
	Let $\mu_h,\mu$ 
be (finite) Radon measures in $\Omega$, taking
values in $\R^M$. Suppose that $\mu_h\stackrel{*}{\rightharpoonup}\mu$ and $|\mu_h|(\Omega)\rightarrow|\mu|(\Omega)$. Then 
	$$
	\lim_{h\rightarrow+\infty}\int_{\Omega}f\left(x,\frac{\mu_h}{|\mu_h|}(x)\right)d|\mu_h|(x)=\int_{\Omega}f\left(x,\frac{\mu}{|\mu|}(x)\right)d|\mu|(x)
	$$
	for any continuous bounded function 
$f:\Omega\times\mathbb S^{M-1}\rightarrow\R$.
\end{Theorem}
\begin{proof}
See for instance
\cite[Theorem 2.39]{AFP}.
\end{proof}
For any $u\in BV(\Omega;\R^2)$, we recall that
the distributional derivative $Du$ is 
a Radon measure valued in $\R^{2\times2}$. 
$SBV(\Omega)$ stands for the space of special functions of 
bounded variation on $\Om$ \cite{AFP}. 
The symbol $|Du|(\Omega)$ stands for the total variation of $Du$ 
(see \cite[Definition 3.4, pag. 119]{AFP})  
with $\vert \cdot \vert$ the Frobenius norm. We denote by $\jump_u$ 
the jump set of $u$.
 
\begin{Definition}[\textbf{Strict convergence}]
	Let $\mappa\in BV(\Omega;\R^2)$ and $(\mappa_k)
\subset BV(\Omega;\R^2)$. 
We say that $(\mappa_k)$ converges
 to $\mappa$ strictly $BV$, 
if
	$$
	\mappa_k\xrightarrow{L^1}\mappa\qquad\mbox{and}\qquad|D\mappa_k|
(\Omega)\rightarrow|D\mappa|(\Omega).
	$$
\end{Definition}
Let $R = \ab \times [-1,1]$. For $(t,
	\secondavariabile)
	\in R$, set 
	$$
\rettangolo_t^{x_1}:=\{(x_1,x_2)\in\rettangolo:x_1=t\},\qquad \rettangolo_\secondavariabile^{x_2}
	:=\{(x_1,x_2)\in\rettangolo:x_2=\secondavariabile\}.
$$ 
If $u\in BV(R;\R^2)$, by Lebesgue differentiation theorem and Fubini theorem, for almost every $t\in \ab$, the 
restriction $u\mres\rettangolo_t^{x_1}$ 
of $u$ on the vertical segment $\rettangolo_t^{x_1}$ coincides with the trace of $u$ on $\mathcal{H}^1$-almost every point of $\rettangolo_t^{x_1}$. 
So, for almost every $t\in \ab$, the map 
$u\mres\rettangolo_t^{x_1}$ 
is 
well defined because it is independent of 
the representative of $u$. The same argument 
holds in $\rettangolo_\secondavariabile^{x_2}$ for almost every $\sigma\in[-1,1]$.
\begin{Lemma}[\textbf{Inheritance of strict convergence
to slices}]\label{lem:inheritance}
	Let $u\in BV(\rettangolo;\R^2)$. 
	Suppose that $(v_k)\subset C^1(\rettangolo;\R^2)$ is a sequence 
	converging to $u$ strictly $BV(\rettangolo;\R^2)$. 
Then for almost every $(t,\secondavariabile)\in\rettangolo$, there exists a subsequence $(v_{k_h})
	\subset (v_k)$, depending on $t$ and $\secondavariabile$, such that
	\begin{align}\label{thesis}
		&v_{k_h}\mres \rettangolo_t^{x_1}\rightarrow u\mres \rettangolo_t^{x_1}\quad\mbox{strictly }BV(\rettangolo_t^{x_1};\R^2),\\
		&v_{k_h}\mres \rettangolo_\secondavariabile^{x_2}\rightarrow u\mres \rettangolo_\secondavariabile^{x_2}\quad\mbox{strictly }BV(\rettangolo_\secondavariabile^{x_2};\R^2).
	\end{align}
\end{Lemma}
\begin{proof}
	For almost every $t\in[a,b]$, in view of the definition of $R_t^{x_1}$, we can define the total variation of $u\mres\rettangolo^{x_1}_t$ as
	\begin{equation}\label{eq:derivata_restrizione}
		|D(u\mres \rettangolo^{x_1}_t)|(\rettangolo^{x_1}_t)=
\sup\left\{-\int_{-1}^1u(t,x_2)\cdot g'(x_2)dx_2;\, g\in C^1_c((-1,1);\overline B_1(0))\right\},
	\end{equation}
where $\overline B_1(0) = \{(\xi,\eta) \in \R^2 : \xi^2 + \eta^2\leq 1\}$.
	Let us show that
	\begin{align}\label{claim}
		|D_2u|(R)=\int_a^b|D(u\mres \rettangolo^{x_1}_t)|(\rettangolo^{x_1}_t)dt,
	\end{align}
       where $D_2u:=Du\, e_2$ is a Radon measure on $R$ valued in $\R^2$ with finite total variation.
	Since, for almost every $t\in[a,b]$, $v_k\mres
	\rettangolo_t^{x_1}\rightarrow u\mres \rettangolo_t^{x_1}$ in $L^1(\rettangolo^{x_1}_t;\R^2)$, we have,
	using \eqref{eq:derivata_restrizione},
	\begin{align}\label{lsc}
		|D(u\mres \rettangolo^{x_1}_t)|
(\rettangolo^{x_1}_t)\leq\liminf_{k\rightarrow +\infty}\int_{\rettangolo^{x_1}_t}|\partial_{2}v_k(t,x_2)|dx_2.
	\end{align}
	Then, using Fatou lemma and Fubini theorem,
	\begin{align}
\label{ineq}
		\int_a^b|D(u\mres \rettangolo^{x_1}_t)|(\rettangolo^{x_1}_t)dt&\leq\int_a^b\liminf_{k\rightarrow +\infty}\int_{\rettangolo^{x_1}_t}|\partial_{2}v_k(t,x_2)|dx_2dt
\\
&\leq\liminf_{k\rightarrow +\infty}\int_R|\partial_{2}v_k(t,x_2)|dtdx_2\nonumber
=|D_2u|(R),\nonumber
	\end{align}
	where in the last equality we used Theorem \ref{Reshetnyak} with $f(x,\nu)=\sqrt{\nu_3^2+\nu_4^2}$, for every $x\in\rettangolo,\;\nu\in\mathbb S^3\subset \R^4=\R^2\times \R^2$, with
	$$\nu=\left(\begin{matrix}
		\nu_1&\nu_3
\\
		\nu_2&\nu_4
	\end{matrix}\right).$$	
	The converse inequality in \eqref{claim} is 
standard\footnote{We recall that
		\begin{equation*}
|D_2u|
(R)=\sup\left\{-\int_{R}u\cdot\partial_{x_2}g ~dx:\, 
			g\in C^1_c(R;\overline B_1(0))\right\}.
		\end{equation*}
Now, for $g\in C^1_c(R;\overline B_1(0))$,
		$\int_{R}u\cdot\partial_yg ~dx
=\int_a^b\left(\int_{-1}^1 u(t,x_2) \cdot\partial_{x_2}g(t,x_2)dx_2\right)dt\leq
\int_a^b|D(u\mres \rettangolo^{x_1}_t)|
(\rettangolo^{x_1}_t)dt$, so
		$|D_2u|(R)\leq\int_a^b|D(u\mres \rettangolo^{x_1}_t)|
(\rettangolo^{x_1}_t)dt$.
	}.
	So, \eqref{claim} is proved and \eqref{ineq} holds
as an equality, 
	which implies that also \eqref{lsc} holds as an equality, namely
	$$
	|D(u\mres \rettangolo^{x_1}_t)|
(\rettangolo^{x_1}_t)=\liminf_{k\rightarrow +\infty}\int_{\rettangolo^{x_1}_t}|\partial_{2}v_k(t,x_2)|dx_2.
	$$
	Extracting a subsequence $(v_{k_h})\subset (v_k)$ 
	depending on $t$, we get 
	$$
	v_{k_h}\mres R_t^{x_1}\rightarrow u\mres R_t^{x_1}\quad\mbox{strictly }BV(R_t^{x_1};\R^2).
	$$
Finally, repeating the same argument for $v_{k_h}$ on the 
horizontal slices $\{R^{x_2}_\secondavariabile\}$, we get \eqref{thesis} for a 
(not relabeled) sub-subsequence.
\end{proof}

Now, let $B_l$ be the disk of $\R^2$ centered at the origin of 
radius $l>0$. We want to prove the analogue
of Lemma \ref{lem:inheritance} in $B_l$, 
by slicing with concentric circumferences. If $u\in BV(B_l;\R^2)$, 
as in the previous case, for almost every $r \in (0,l)$ 
the restriction $u\mres\partial B_r$ is well-defined 
and independent of the representative of $u$.
In particular,  for almost every $r \in (0,l)$, 
we can define the total variation of $u\mres\partial B_r$ as 
\begin{equation}\label{var on circ}
|D(u\,\mres\partial B_r)|(\partial B_r):=
\sup\left\{ \!-\!\! \int_0^{2\pi}\!\!\bar{u}(r,\theta)\cdot f'(\theta)d\theta; f\in C^1([0,2\pi];\overline B_1(0)),f(0)\!=\!f(2\pi),f'(0)\!=\!f'(2\pi)\!\right\}
\end{equation}
which turns out to be finite (see Lemma \ref{lem:inheritance circ}), 
giving that $u\mres\partial B_r\in BV(\partial B_r;\R^2)$, for almost 
every $r \in (0,l)$. 
Here 
$$
\bar u(r,\theta):=u(r\cos\theta,r\sin\theta), \qquad 
r\in(0,l], \ 
\theta\in[0,2\pi).
$$

We want to relate this quantity with the notion of tangential total variation.
\begin{Definition}\label{def:annulus}
For $x=(x_1,x_2)\in\R^2\setminus\{(0,0)\}$, set 
$\tau(x)=\frac{1}{|x|}(-x_2,x_1)$. Let $0< l < L$ and 
$A_{L,l}:=B_L(0)\setminus \overline{B_l(0)}$ be an 
annulus around $0$. We define the tangential total variation of $u\in BV(A_{L,l};\R^2)$ 
as the total variation of the Radon measure $D_\tau u:=Du\tau$, namely
\begin{equation}\label{tangential var}
|D_\tau u|(A_{L,l})=|Du\tau|(A_{L,l})=\sup\Big\{-\int_{A_{L,l}}u\cdot(\nabla g\tau) ~dx: g\in C^1_c(A_{L,l};\overline B_1(0))
\Big\}.
\end{equation}
\end{Definition}
\noindent
The last equality in \eqref{tangential var} is justified since $\tau\in C^\infty(A_{L,l};\R^2)$ satisfies $\text{\rm div}\tau=0$ everywhere, so for any $g=(g^1,g^2)\in C^1_c(A_{L,l};\R^2)$ we have
\begin{equation}\nonumber
\begin{aligned}
& -\int_{A_{L,l}}u\cdot(\nabla g\tau) ~dx
=-\int_{A_{L,l}}u^1\nabla g^1\cdot\tau ~dx-\int_{A_{L,l}}u^2\nabla g^2\cdot\tau ~dx
\\
&=-\int_{A_{L,l}}u^1\mathrm{div}( g^1\tau) ~dx-\int_{A_{L,l}}u^2\mathrm{div}( g^2\tau )~dx
\\
&=\int_{A_{L,l}}g^1\tau\cdot dDu^1+\int_{A_{L,l}}g^2\tau\cdot dDu^2
=\int_{A_{L,l}}g\cdot (dDu)\tau=\langle Du\tau,g\rangle.
\end{aligned}
\end{equation}
This computation shows that $|D_\tau u|(A_{L,l})\leq|Du|(A_{L,l})$, since $|\tau|\leq1$, and also that \eqref{tangential var} is compatible with the case $u\in W^{1,1}(A_{L,l};\R^2)$, 
where simply $|D_\tau u|(A_{L,l})=\int_{A_{L,l}}|\nabla u\tau|~dx$. 
Moreover, $Du=\frac{Du}{|Du|}|Du|$ by polar decomposition, 
so that
$$
\langle Du\tau,g\rangle=\int_{A_{L,l}}g\cdot (dDu)\tau=\int_{A_{L,l}}g\cdot \left(\frac{Du}{|Du|}d|Du|\right)\tau=\int_{A_{L,l}}g\cdot \left(\frac{Du}{|Du|}\tau\right)d|Du|\qquad\forall g\in C_c^1(B_l;\R^2),
$$
giving that 
\begin{align}\label{polar dec}
D_\tau u=Du\tau=\frac{Du}{|Du|}\tau|Du|.
\end{align}
\begin{Lemma}[\textbf{Inheritance of strict convergence to circumferences}]\label{lem:inheritance circ}
Let $u\in BV(B_R;\R^2)$ and
$(v_k)\subset C^1(B_R;\R^2)$ be a sequence 
converging to $u$ strictly $BV(B_R;\R^2)$. 
Then, for almost every 
$r\in(0,R)$, 
there exists a subsequence $(v_{k_h})
\subset (v_k)$, depending on $r$, such that
\begin{equation}
v_{k_h}\mres \partial B_r\rightarrow u\mres \partial B_r\quad\mbox{strictly }BV(\partial B_r;\R^2)
\qquad {\rm as}~ h \to +\infty.
\end{equation}
\end{Lemma}
\begin{proof}
 For almost every $r\in(0,R)$, by Fatou lemma and Fubini theorem, the restriction $v_k\mres\partial B_r$ has equi-bounded variation w.r.t. $k$. Moreover, we may assume that $(v_k)$ converges to $u$ almost everywhere in $B_R$, so that, 
for almost every $r\in(0,R)$,
\begin{align}\label{a.e. convergence}
v_k\mres\partial B_r\rightarrow u\mres\partial B_r\quad\mathscr{H}^1\mbox{-a.e. in }\partial B_r.
\end{align}
Now, let $r\in(0,R)$ be such that $v_k\mres\partial B_r$ has equi-bounded variation and \eqref{a.e. convergence} holds. Then, there exists a subsequence $(v_{k_h})\subset(v_k)$ depending on $r$ such that
$$
v_{k_h}\mres\partial B_r\stackrel{*}{\rightharpoonup}u\mres\partial B_r\quad\mbox{w*}-BV(\partial B_r;\R^2).
$$
By lower semicontinuity of the variation, we infer that for 
almost every $r\in(0,R)$, $u\,\mres\partial B_r\in BV(\partial B_r;\R^2)$ and
\begin{align}\label{lsc on circ}
|D(u\mres\partial B_r)|(\partial B_r)\leq\liminf_{h\rightarrow+\infty}\int_{\partial B_r}\left|\nabla v_{k_h}\tau\right|d\mathcal{H}^1.
\end{align}
Let $0<l<L\leq R$ be such that $v_k\rightarrow u$ strictly $BV(A_{L,l},\R^2)$ where,
as in Definition \ref{def:annulus},
$A_{L,l}:=B_L(0)\setminus \overline{B_l(0)}$ 
(notice that this holds for a.e. $l$ and $L$); by integration, we get
\begin{equation}\label{chain ineq}
\begin{aligned}
& \int_l^L|D(u\mres\partial B_r)|(\partial B_r)~dr
\leq\int_l^L\left(\liminf_{h\rightarrow+\infty}\int_{\partial B_r}\left|\nabla v_{k_h}\tau\right|d\mathcal{H}^1\right)~dr
\\
\leq& \liminf_{h\rightarrow+\infty}\int_l^L\int_{\partial B_r}\left|\nabla v_{k_h}\tau\right|d\mathcal{H}^1dr
=\liminf_{h\rightarrow+\infty}\int_{A_{L,l}}\left|\nabla v_{k_h}\tau\right|~dx.
\end{aligned}
\end{equation}
Thanks to Theorem \ref{Reshetnyak}, with the choices $M=4$,
	$\mathbb S^3\subset\R^4=\R^{2\times2}$, 
$f\in C_b(A_{L,l}\times \mathbb S^3)$, 
$$f(x,\nu):=\sqrt{|\nu_{{\rm hor}}\cdot \tau(x)|^2+|\nu_{{\rm vert}}\cdot  \tau(x)|^2},$$ 
where $\nu\in\mathbb S^3$ and $\nu_{{\rm hor}}:=
(\nu_1,\nu_3), \nu_{{\rm vert}}:=(\nu_2,\nu_4)$,
 we obtain
	\begin{equation}\label{Reshetnyak application}
	\lim_{k\rightarrow+\infty}\int_{A_{L,l}}|\grad v_k\tau|~dx
=\int_{A_{L,l}}\left|\frac{Du}{|Du|}\tau\right|d|Du|=|D_\tau u|(A_{L,l}),
	\end{equation}
where in the last equality we have used \eqref{polar dec}.
So we get 
$$
|D_\tau u|(B_l)\geq\int_l^L
|D(u\mres\partial B_r)|(\partial B_r)~dr.
$$
In order to prove the converse inequality, let $g\in C^1_c(A_{L,l};\overline B_1(0))$. Then, in polar coordinates, by definition \eqref{var on circ},
$$
\int_{A_{L,l}}u\cdot\nabla g\tau ~dx=\int_l^L\int_0^{2\pi}\bar{u}(\rho,\theta)\cdot\partial_\theta\bar{g}(\rho,\theta)~d\rho d\theta\leq\int_l^L
|D(u\mres\partial B_\rho)|(\partial B_\rho)~d\rho.
$$
So, we have proved that 
$$
|D_\tau u|(A_{L,l})=\int_l^L|D(u\mres\partial B_r)|(\partial B_r)~dr.
$$
In particular, we deduce that \eqref{chain ineq} is a chain of equalities.
Then, \eqref{lsc on circ} holds as an equality and there exists a subsequence $(v_{k_h})
\subset (v_k)$, depending on $r$, which achieves the full limit. Since $l$ and $L$ are arbitrary, we get the thesis.
\end{proof}
\subsection{Further properties in dimension 1}
In \cite[Proposition 2.4]{BCS} the following is proved:

\begin{Lemma}\label{lem:strict_implies_uniform_1}
Let $(\gamma_k) \subset W^{1,1}((a,b);\R^2)$ be a sequence
converging strictly $BV((a,b);\R^2)$ to $\gamma \in 
W^{1,1}((a,b);\R^2)$. Then $\gamma_k\rightarrow \gamma$ uniformly in $(a,b)$. 
\end{Lemma}

For our purposes, we need an 
improvement of Lemma \ref{lem:strict_implies_uniform_1}, 
where 
discontinuous functions $\gamma$ at a single point, or at
a finite number of points,  
are allowed; we start with one point discontinuity. 

\begin{Lemma}\label{strict implies uniform 2}
	Let 
	$I^-:=[-1,0), I^+:=(0,1]$. Suppose that $(\gamma_k)
	\subset W^{1,1}([-1,1] ;\R^2)$ is a sequence 
	converging strictly $BV([-1,1] ;\R^2)$ to $\gamma\in BV([-1,1] ;\R^2)
\cap W^{1,1}(I^-;\R^2)
\cap W^{1,1}(I^+;\R^2)$,
	with $\gamma^+(0) \neq \gamma^-(0)$.
	Let $S:
	[-1/3,1/3]
	\rightarrow\R^2$ be
	defined by
	$$
	S(\tau):=\frac{3}{2}\left(\left(1/3+\tau\right)\gamma^+(0)+\left(
	1/3-\tau\right)\gamma^-(0)\right), \quad \tau\in [-1/3,1/3].
	$$
	Let $\widetilde\gamma^-$ 
 (resp. 
	$\widetilde\gamma^+$)
	be the reparametrization of $\gamma_{|I^-}$ (resp. $\gamma_{|I^+}$) 
	on $[-1,-\frac{1}{3})$ (resp. 
	$(\frac{1}{3},1]$)
	defined by the composition with 
the increasing linear function taking 
$[-1,-1/3]$ 
onto $[-1,0]$ (resp. $[1/3,1]$ 
onto $[0,1]$).
Define 
	\begin{equation}\label{concatenation}
\widetilde\gamma:
	[-1,1] 
	\rightarrow\R^2, 
\qquad
		\widetilde\gamma:=
		\begin{cases}
			\widetilde\gamma^-&\quad\mbox{in }[-1,-1/3)
			\\
			S&\quad\mbox{in } [-1/3,1/3]
			\\
			\widetilde\gamma^+&\quad\mbox{in } (1/3,1].
		\end{cases}
	\end{equation}
	Then there exist:
	\begin{itemize}
		\item[(a)] a 
Lipschitz strictly increasing surjective function $\repa:[-1,1]\rightarrow[-1,1]$, 
		\item[(b)]
		a subsequence $(k_j)$ and Lipschitz strictly increasing 
surjective functions $\repa_{k_j}:[-1,1]\rightarrow[-1,1]$ for any $j \in \mathbb N$, 
with $\sup_j \Vert \dot \repa_{k_j}\Vert_\infty < +\infty$,
	\end{itemize}
	such that
	\begin{align}\label{uniform_rescalings}
	\lim_{\indice \to +\infty} 
\gamma_{k_j}\circ \repa_{k_j} = \widetilde\gamma\circ h\quad\mbox{ uniformly in }[-1,1].
	\end{align}
\end{Lemma}

\begin{proof}
	The lengths $L_k$ of $\gamma_k$ and $L$ of $\gamma$ are given by 
	\begin{align*}
		&L_k=\int_{-1}^1|\dot{\gamma}_k|~d\tau,
\\
		&L=|\dot{\gamma}|([-1,1])=\int_{-1}^0|\dot\gamma|~d\tau
		+|\gamma^+(0)-\gamma^-(0)|+\int_0^1|\dot\gamma|~d\tau.
	\end{align*}
Since, by assumption, $\gamma_k\rightarrow\gamma$ strictly $BV([-1,1];\R^2)$, we have that $L_k\rightarrow L$ as $k\rightarrow+\infty$.
	Fix $\eta>0$ and for all $k \in \mathbb N$ 
	define the function\footnote{We need $\eta$, since in 
principle $\dot \gamma_k$ could vanish
		somewhere.}
	\begin{equation}\label{eq:s_k}
s_k:[-1,1]\rightarrow[0,L+\eta], \qquad 
		\begin{aligned}
			s_k(t):=
			\frac{L+\eta}{L_k+\eta}\int_{-1}^t\Big(|\dot\gamma_k(\tau)|+\frac{\eta}{2}\Big)~d\tau,
		\end{aligned}
	\end{equation}
	with Lipschitz inverse $\newpar_k:=s_k^{-1}:[0,L+\eta]\rightarrow[-1,1]$. 
Define 
	\begin{align}
\widehat\gamma_k:[0,L+\eta]\rightarrow\R^2, \qquad 
		\widehat\gamma_k(s):=\gamma_k(\newpar_k(s))
		\qquad\forall s\in [0,L+\eta].
	\end{align} 
	Since from \eqref{eq:s_k}
	$$\left|\frac{d\widehat\gamma_k}{ds}(s)\right|\leq 
	\frac{|\dot\gamma_k(\newpar_k(s))|}{|\dot s_k(\newpar_k(s))|}\leq\frac{L_k+\eta}{L+\eta}\leq C \qquad {\rm for~a.e.}~
	s\in [0, L+\eta],
	$$
	for some constant $C>0$ independent of $k$, 
	the sequence $(\widehat\gamma_k)$ is 
	bounded in $W^{1,\infty}([0,L+\eta];\R^2)$. Thus, up to a (not
	relabeled) subsequence, we may assume that there exists $\widehat\gamma\in W^{1,\infty}([0,L+\eta];\R^2)$ such that 
	\begin{align}\label{thesis1}
		\widehat \gamma_k\rightharpoonup\widehat\gamma\text{ weakly* in }W^{1,\infty}([0,L+\eta];\R^2)\text{ and uniformly in }[0,L+\eta].
	\end{align}
	We observe that for any open interval $J\subseteq [0,L+\eta]$,
	$$\int_{J}|\dot{\widehat\gamma}|ds\leq
	\liminf_{k\rightarrow+\infty}\int_{J}|\dot{\widehat\gamma}_{k}|ds \leq |J| \liminf_{k\rightarrow +\infty}\frac{L_{k}+\eta}{L+\eta}=|J|,
	$$
and thus
	\begin{equation}\label{eq:subunitary}
		|\dot{\widehat\gamma}|\leq 1\text{ a.e. in }[0,L+\eta].
	\end{equation}
	Now, in order to conclude the proof, we 
	need to show that $\widehat \gamma$ is a 
	reparametrization of 
	$\widetilde\gamma$. Then the thesis of the lemma will follow by reparametrizing both $\widehat\gamma_{k}$ and $\widehat \gamma$ on $[-1,1]$.
	
	Using that $(\gamma_k)$ strictly converges
$BV([-1,1]; \R^2)$ to $\gamma \in 
W^{1,1}(I^-; \R^2)\cap 
W^{1,1}(I^+; \R^2)$, by Lemma 
\ref{lem:strict_implies_uniform_1}
and a diagonal process, we can find an 
	infinitesimal sequence $(\delta_{k_\indice}) \subset (0,1]$ 
	such that
	\begin{align}\label{uniform_0}
		&\|\gamma_{k_\indice}-\gamma\|_{L^\infty([-1,1] \setminus(-\delta_{k_\indice},\delta_{k_\indice}))}\rightarrow 0
	\end{align}
	and 
	\begin{align*}
		&\int_{-1}^{-\delta_{k_\indice}}|\dot\gamma_{k_\indice}(\tau)|~d\tau\rightarrow \int_{-1}^0|\dot\gamma(\tau)|~d\tau,\qquad\int_{\delta_{k_\indice}}^1|\dot\gamma_{k_\indice}(\tau)|~d\tau
		\rightarrow \int_0^1|\dot\gamma(\tau)|~d\tau
	\end{align*}
	as $\indice \to +\infty$.
	In particular,
	\begin{align}\label{limit delta}
\lim_{\indice \to +\infty}\gamma_{k_\indice}(\pm\delta_{k_\indice})= \gamma^\pm(0)
\end{align}
	and, setting
	\begin{align*}
		r_{k_\indice}^-
		:=&
		s_{k_\indice}(-\delta_{k_\indice})=\frac{L+\eta}{L_{k_j}+\eta}\int_{-1}^{-\delta_{k_\indice}}\big(|\dot\gamma_{k_\indice}|+\frac{\eta}{2}\big)~d\tau,
		\\
		r^+_{k_\indice}
		:=& 
		s_{k_\indice}(\delta_{k_\indice})=
\frac{L+\eta}{L_{k_j}+\eta}\left[
\int_{-1}^{1}\big(|\dot\gamma_{k_\indice}|+\frac{\eta}{2}\big)~d\tau-
\int_{\delta_{k_\indice}}^1\big(|\dot\gamma_{k_\indice}|+\frac{\eta}{2}\big)~d\tau\right],
	\end{align*}
	we have 
	\begin{equation}\label{eq:conv_rk}
		\begin{aligned}
			& 
			\lim_{\indice \to +\infty}
			r_{k_\indice}^- = \frac{\eta}{2}+\int_{-1}^0|\dot\gamma|~d\tau=:r^-,
			\\
			&\lim_{\indice \to +\infty}
			r^+_{k_\indice}
			=\frac{\eta}{2}+\int_{-1}^0|\dot\gamma|~d\tau +|\gamma^+(0)-\gamma^-(0)|=:r^+.
		\end{aligned}
	\end{equation}
	As a consequence of \eqref{thesis1}, \eqref{limit delta}, and \eqref{eq:conv_rk} 
	we get
	$$
	\gamma_{k_\indice}(\newpar_{k_\indice}(
	r_{k_\indice}^\pm
	))=
	\widehat\gamma_{k_\indice}(
	r_{k_\indice}^\pm
	)\rightarrow 
	\widehat\gamma(r^\pm)=
	\gamma^\pm(0).
	$$
	Therefore the curve $\widehat \gamma$ 
	maps the segment $[r^-,r^+]$ into a curve joining
$\gamma^-(0)$ 
	and $\gamma^+(0)$. Now, since $r^+-r^-=|\gamma^+(0)-\gamma^-(0)|$, from \eqref{eq:subunitary}
	we conclude that $\widehat \gamma$ 
	coincides with the unit-speed parametrization of the segment 
joining $\gamma^-(0)$ and $\gamma^+(0)$
	on $[r^-,r^+]$. Hence we have shown that 
	\begin{align}\label{ali}
		\gamma_{k_\indice}\circ \newpar_{k_\indice}\rightarrow  S\circ 
		\widetilde\newpar\text{ uniformly in }[r^-,r^+] {\rm ~as~} \indice \to +\infty,
	\end{align}
	for the affine increasing reparametrization $\widetilde\newpar
	:[r^-,r^+]\rightarrow[-1/3,1/3]$.
	
	We now check that $\widehat\gamma=\gamma\circ \newpar$ on $[0,r^-]$ for some increasing bijection 
	$\newpar:[0,r^-]\rightarrow [-1,0]$,
	and similarly
	$\widehat\gamma=\gamma\circ \beta$ on $[r^+, L+\eta]$ 
	for some increasing bijection 
	$\beta:[r^+, L+\eta]\rightarrow [0,1]$. 
	
	Indeed, the functions $\newpar_k:[0,L+\eta]\rightarrow[-1,1]$ are strictly increasing and satisfy
	$$|\dot\newpar_k(s_k(t))|=\frac{L_k+\eta}{(L+\eta)(|\dot\gamma_k(t)|+\frac\eta2)}\leq \frac{C}{\eta},$$
	so that we may
	assume (up to extracting a further not relabeled subsequence) 
	that 
	$$ 
	\newpar_{k_\indice}\rightharpoonup
	\newpar
	\text{ weakly* in }W^{1,\infty}([0,L+\eta])\text{ and uniformly in }[0,L+\eta],
	$$
	for some nondecreasing map $\newpar\in W^{1,\infty}([0,L+\eta])$.
	Hence, using \eqref{uniform_0}, we find out
	$$\widehat \gamma_{k_\indice}(s)=\gamma_{k_\indice}(\newpar_{k_\indice}(s))\rightarrow\gamma
	(\newpar(s))\text{ for all }s\in [0,r^-).$$
	This, together with \eqref{thesis1}, implies 
	$$\widehat\gamma(s)=\gamma\circ\newpar(s) \text{ for all }s\in [0,r^-).$$
	A similar argument shows that this also holds for all $s\in (r^+,L+\eta]$. 
	
	Finally, we observe that $\newpar$ is strictly 
increasing on $[0,r^-)\cup(r^+,L+\eta]$. For, 
	if $\newpar$ 
is constant on some interval $[s_1,s_2]\subset[0,r^-)$, we have $\lim_{\indice
\rightarrow +\infty}
	\newpar_{k_\indice}(s_1)=\lim_{h\rightarrow +\infty}
	\newpar_{k_\indice}(s_2)$ and hence
	\begin{equation}\label{eq:0}
		0=
		\lim_{\indice\rightarrow +\infty}
		\int_{s_1}^{s_2}\dot\newpar_{k_\indice}(s)ds
		=\lim_{\indice \rightarrow +\infty}
		\int_{t_{k_\indice,1}}^{t_{k_\indice,2}}~d\tau=
		\lim_{j\rightarrow +\infty}
		(t_{k_\indice,2}-t_{k_\indice,1}),
	\end{equation}
	where $t_{k_\indice,i}$ are defined by
	$s_{k_\indice}(
	t_{k_\indice,1}
	)=s_1$ and $s_{k_\indice}(t_{k_\indice,2})=s_2$. By definition 
	\eqref{eq:s_k}
	of $s_{k_\indice}$ we have
	\begin{equation}\label{eq:mag_0}
		0<s_2-s_1=\int_{t_{k_\indice,1}}^{t_{k_\indice,2}}\big(|\dot\gamma_{k_\indice}(\tau)|+\frac\eta2\big)~d\tau.
	\end{equation}
	Possibly passing to a 
(not relabeled) subsequence and using \eqref{eq:0},  
	let $\overline t\in [-1,0]$ be the limit 
	of $(t_{k_\indice,1})$ and $(t_{k_\indice,2})$.
	If $\overline t \neq 0$,
	for any open neighborhood $J \subset (-1,0)$ of $\overline t$, using \eqref{eq:mag_0}, we get
	$$\int_J|\dot \gamma|~d\tau=\lim_{h\rightarrow +\infty}\int_{J}|\dot\gamma_{k_\indice}|~d\tau\geq 
	s_2-s_1,$$
which contradicts the 
	inclusion 
	$\dot\gamma\in L^1((-1,0); \R^2)$.
	The same argument holds if $\overline t=0$, for $J$ a left
	neighbourhood of $0$ in $(-1,0)$. We conclude that $\newpar$ is 
	strictly increasing.
	
	Let 
	$\repa_{k_\indice}$ be a rescaling of $\newpar_{k_j}$ on $[-1,1]$;  rescaling also 
	$\newpar$ from $[0,r^-]$ to $[-1,-1/3]$, and then from $[r^+,L+\eta]$ to $[1/3,1]$, using also $\widetilde\newpar$ in \eqref{ali}, we construct a reparametrization $h:[-1,1]\rightarrow[-1,1]$ such that \eqref{uniform_rescalings} holds, and the lemma is proved.
\end{proof}

Lemma \ref{strict implies uniform 2}
can be readily extended to curves $\gamma$ with finitely many jump points:
\begin{Corollary}\label{corollario_conv}
Let 
$(\gamma_k)
\subset W^{1,1}([0,2\pi];\R^2)$ be a sequence 
converging strictly $BV([0,2\pi];\R^2)$ to a map $\gamma\in SBV([0,2\pi];\R^2)$
having  finitely many jump points $0<z_1<z_2<\dots < z_n<2\pi$. 
Let $\theta_0>0$ be such that 
the intervals $(z_i-\theta_0,z_i+\theta_0)\subset(0,2\pi)$ are 
disjoint, and for all $i=1,\dots,n$ let $S_i:[z_i-\theta_0,z_i+\theta_0]\rightarrow\R^2$ be defined by 	
$$
S_i(\tau):=\frac{1}{2\theta_0}\left(\left(\tau-z_i+\theta_0\right)\gamma^+(z_i)+\left(
z_i+\theta_0-\tau\right)\gamma^-(z_i)\right), \quad \tau\in [z_i-\theta_0,z_i+\theta_0].
$$
Setting $z_0:=0$
and $z_{n+1}:= 2\pi$, 
for all $i=0,\dots,n$ 
let $\widetilde \gamma_i:[z_i+\theta_0,z_{i+1}-\theta_0]\rightarrow\R^2$ be a rescaled reparametrization of $\gamma:[z_i,z_{i+1}]$.
Finally, let $\widetilde\gamma:[0,2\pi]\rightarrow\R^2$ 
be the Lipschitz curve defined as
\begin{align}\label{gamma_tilde}
\widetilde\gamma:=\widetilde \gamma_0\star S_1\star\widetilde \gamma_1\star S_2\star\widetilde \gamma_2\star\dots\star S_{n}\star \widetilde \gamma_n,
\end{align}
where $\star$ denotes the arc composition.
Then there exist a subsequence $(k_\indice)$ and Lipschitz increasing 
surjective functions $\repa, \repa_{k_\indice}:[0,2\pi]\rightarrow[0,2\pi]$ 
such that 
\begin{align}\label{uniform_rescalings2}
	\lim_{\indice \to +\infty}
	\gamma_{k_\indice}\circ \repa_{k_\indice} 
 = \widetilde\gamma\circ h
 \quad\mbox{ uniformly in }[0,2\pi].
\end{align}
	\end{Corollary}
\begin{proof}
	We skecth the proof which is a direct consequence of the arguments used to prove Lemma \ref{strict implies uniform 2}.
Choose points $w_i$, $i=1,\dots,n-1$ so that $z_i+\theta_0<w_i<z_{i+1}-\theta_0$, and let $w_0=0$ and $w_n=2\pi$. Then we can apply Lemma 
\ref{strict implies uniform 2}  to any interval $[w_i,w_{i+1}]$, and taking a suitable subsequence and concatenating the obtained maps  one can easily construct the desired parametrizations.
\end{proof}

\subsection{Planar Plateau-type problem}
\label{subsec:planar_Plateau_type_problem}
Let $\varphi:\unitcircle
\rightarrow\R^2$ be a possibly self-intersecting Lipschitz curve.
Let us consider, as in \cite{Pa} (see also \cite{FFM}), the
 planar Plateau-type problem 
\eqref{Plateau}
spanning $\varphi$. 
Notice that the class of competitors 
is non-empty, since it 
contains the map $v(x)=|x|
\homogeneousmap\left(\frac{x}{|x|}\right)$ for $x \neq 0$, and $v(0)=0$. 
We first observe that $P$ is independent of the radius
of the domain of integration. 
Specifically, for any $r>0$, let 
\begin{align}\label{notation_r}
\varphi_r(y):=\varphi\left(\frac{y}{r}\right) \text{ for all } y\in \partial B_r.
\end{align} 
Setting $y:=rx$,  $y\in \overline{B_r}$ and 
 $v_r(y):=v(\frac yr)$, we have
\begin{equation}\label{eq:int_Jac_radius}
\int_{B_1}|Jv|~dx=\int_{B_r}|Jv_r|dy\qquad\qquad \forall v\in \mathrm{Lip}(B_1;\R^2).
\end{equation}
In particular, for any $r>0$,
\begin{align}\label{Plateaurescaled}
P(\varphi)=\inf\left\{\int_{B_r}|Jv|~dx: \,v\in \mathrm{Lip}(B_r;\R^2), v_{|\partial B_r}=
\homogeneousmap_r
\right\}.
\end{align}

In 
the next
proposition we show that $P(\cdot)$ is invariant under
Lipschitz reparameterizations of $\varphi$.
\begin{Proposition}[\textbf{Invariance}]\label{independence_of_P}
	Let $\homogeneousmap \in {\rm Lip}(\Suno; \R^2)$ 
and $\repa$ be a Lipschitz homeomorphism of $\Suno$. Then 
$$
P(\homogeneousmap\circ \repa)=P(\homogeneousmap).
$$
\end{Proposition}
\begin{proof}

	Since $h$ and the identity map $\mathrm{id}:\Suno\rightarrow\Suno$ 
have the same degree, they are 
homotopic in $\Suno$ by Hopf Theorem (see \cite[pag. 51]{Mi}), namely there exists a Lipschitz map\footnote{The construction of a Lipschitz homotopy between $h$ and id can be done at the level of liftings, by considering the affine interpolation map (for more details, see for instance \cite[Proposition 3.4]{BCS}).} $K:[0,1]\times\Suno\rightarrow\Suno$
such that
	$$
	K(0,\cdot)=\mathrm{id},\quad K(1,\cdot)=h.
	$$
	Define $H:[0,1]\times\Suno\rightarrow\R^2$ as $H(t,\nu)=
	\homogeneousmap
	(K(t,\nu))$. Then, $H$ is Lipschitz and 
	$$
	H(0,\cdot)=\homogeneousmap,\quad H(1,\cdot)=\varphi \circ h.
	$$
	Now, suppose 
$v_k\in\mathrm{Lip}(B_1;\R^2)$ is such that $v_k=
	\homogeneousmap$ on $\partial B_1$ and 
	$$
	\lim_{k\rightarrow+\infty}\int_{B_1}|Jv_k|~dx\rightarrow P(\homogeneousmap).
	$$
Define the map $\widetilde{v}_k:B_1\rightarrow\R^2$ as
	\begin{equation}
		\widetilde{v}_k(x)=\begin{cases}
			v_k(kx)&  {\rm for ~} x \in B_\frac{1}{k},
\\
H\left(k|x|-1,\frac{x}{|x|}\right) & {\rm for~ } x 
\in B_\frac{2}{k}\setminus B_\frac{1}{k},
\\
\homogeneousmap\circ h
\left(\frac{x}{|x|}\right)& {\rm for ~}x 
\in  B_1\setminus B_\frac{2}{k}.
		\end{cases}
	\end{equation}
	Then $\widetilde{v}_k \in {\rm Lip}(B_1; \R^2)$ 
and $\widetilde{v}_k=\homogeneousmap\circ h$ on $\partial B_1$. 
Moreover, since $H$ and $\homogeneousmap\circ h$ take values  in $\varphi(\Suno)$ which is $1$-dimensional, by the 
area formula and \eqref{eq:int_Jac_radius} we have
	$$
	\int_{B_1}|J\widetilde{v}_k(x)|~dx=\int_{B_\frac{1}{k}}|J{v}_k(kx)|~dx=\int_{B_1}|Jv_k|~dx\rightarrow P(\homogeneousmap)
	$$
as $k \to +\infty$.
	In particular $P(\homogeneousmap\circ h)\leq P(\homogeneousmap)$.
	Exchanging the role of $\homogeneousmap$ and $\homogeneousmap\circ h$,
 we obtain the converse inequality.
\end{proof}

\begin{Lemma}\label{lem:P_is_Lip}
Let $\homogeneousmap_1,\varphi_2\in\mathrm{Lip}(\Suno;\R^2)$. Then
\begin{align}
	|P(\varphi_1)-P(\varphi_2)|\leq 
2\|\varphi_1-\varphi_2\|_{\infty}\big(
\Vert \dot\varphi_1\Vert_1
+\Vert \dot\varphi_2\Vert_1
\big).
\end{align} 
\end{Lemma}
\begin{proof}
Let $v\in\mathrm{Lip}(B_1;\R^2)$ be such that $v=\varphi_2$ on $\Suno$. We define 
\begin{equation}
	w(x)=
	\begin{cases}
		v_{\frac{1}{2}}(x)=v({2x}) &\mbox{if } |x|<\frac{1}{2},
\\
		2({1-|x|}){\varphi_2}\left(\frac{x}{|x|}\right)+2\left({|x|-\frac{1}{2}}\right)\homogeneousmap_1\left(\frac{x}{|x|}\right)&\mbox{if } \frac{1}{2}\leq|x|\leq1.
	\end{cases}
\end{equation}
Then
$w\in\mathrm{Lip}(B_1;\R^2)$,
$w(x)=\varphi_2(x/\vert x\vert)$ if $x\in \partial B_{\frac12}$  
and $w=\homogeneousmap_1$ on $\partial B_1$. Let us estimate
$$\int_{B_1\setminus \overline{B_{\frac{1}{2}}}}|Jw|~dx.$$
	Writing $w$ in polar coordinates in the annulus $B_1\setminus \overline{B_{\frac{1}{2}}}$, $\rho\in (\frac12,1)$, $\theta\in [0,2\pi)$,
$$
\bar{w}(\rho,\theta):=w(\rho\cos\theta,\rho\sin\theta)=2(1-\rho)\bar\varphi_2(\theta)+2\left(\rho-\frac{1}{2}\right)\bar\varphi_1(\theta),
$$
where $\bar\varphi_i(\theta):=\varphi_i(\cos\theta,\sin\theta)$, $i=1,2$.
Then
\begin{align*}
	|J\bar{w}(\rho,\theta)|&=4\left|(\bar\varphi_1(\theta)-\bar\varphi_2(\theta))\wedge\left((1-\rho)\dot{\bar\varphi}_2(\theta)+\left(\rho-\frac{1}{2}\right)\dot{\bar\varphi}_1(\theta)\right)\right|
\\
	&\leq 4\left|\bar\varphi_1(\theta)-\bar\varphi_2(\theta)\right|\left|(1-\rho)\dot{\bar\varphi}_2(\theta)+\left(\rho-\frac{1}{2}\right)\dot{\bar\varphi}_1(\theta)\right|
\\
	&\leq 4\|\varphi_1-\varphi_2\|_{\infty}\left(|\dot{\bar\varphi}_2(\theta)|+|\dot{\bar\varphi}_1(\theta)|\right).
\end{align*}
Thus, integrating on $B_1\setminus \overline{B_{\frac{1}{2}}}$,
\begin{align}\label{jacob var annulus}
	\int_{B_1\setminus \overline{B_{\frac{1}{2}}}}|J{w}(x)|~dx&=\int_{\frac{1}{2}}^1\int_0^{2\pi}|J\bar{w}(\rho,\theta)|~d\rho d\theta
\\
	&\leq 2\|\varphi_1-\varphi_2\|_\infty
\int_0^{2\pi}
\left(|\dot{\bar\varphi}_2(\theta)|+|\dot{\bar\varphi}_1(\theta)|\right)d\theta
\\
	&=2
\|\varphi_1-\varphi_2\|_\infty\left(\Vert \dot\varphi_1\Vert_1
+\Vert \dot\varphi_2\Vert_1
\right).
\end{align}
Hence
\begin{align}\label{P estim}
P(\varphi_1)\leq \int_{B_1}|Jw|~dx\leq \int_{B_{\frac12}}
|Jv_{\frac12}|~dx+
2\|\varphi_1-\varphi_2\|_\infty\left(
\Vert \dot\varphi_1\Vert_1
+\Vert \dot\varphi_2\Vert_1
\right).
\end{align}
Since $v$ is arbitrary and (with the notation in \eqref{notation_r}) $v_{\frac12}=(\varphi_2)_\frac12$ on $\partial B_{\frac12}$, using \eqref{Plateaurescaled} with $r=\frac{1}{2}$ we can take the infimum on these maps in \eqref{P estim} and get
\begin{align*}
	P(\varphi_1)-P(\varphi_2)
\leq2\|\varphi_1-\varphi_2\|_\infty
\left(\Vert \dot\varphi_1\Vert_1
+\Vert \dot\varphi_2\Vert_1
\right).
\end{align*}
 Exchanging the role of $\varphi_1$ and $\varphi_2$ we find that 
also $P(\varphi_2)-P(\varphi_1)$ is bounded by the right-hand side 
of the previous expression. This concludes the proof.
\end{proof}

\begin{Remark}\label{affineint_rem}
With a similar argument used in the proof of Lemma 
\ref{lem:P_is_Lip} it is 
immediate to obtain that if 
$[a,b]\subset\R$ is a bounded interval and 
$\gamma_1,\gamma_2:[a,b]\rightarrow\R^2$ are 
Lipschitz curves, then the following holds: Let $\Phi:[a,b]\times[0,1]\rightarrow\R^2$ be
the 
affine interpolation map $\Phi(t,s):=s
\gamma_1(t)+ 
(1-s)
\gamma_2(t)$.
Then, as in \eqref{jacob var annulus},
\begin{align}
	\int_{[a,b]\times [0,1]}|\Phi_t\wedge\Phi_s|~dtds
\leq \|\gamma_1-\gamma_2\|_{\infty}(\Vert \dot\gamma_1\Vert_1
+\Vert \dot\gamma_2\Vert_1).
\end{align}
\end{Remark}

Using Lemma \ref{lem:P_is_Lip} 
we readily obtain the following continuity property for the minimum of the Plateau-type problem \eqref{Plateau}.

\begin{Corollary}[\textbf{Continuity of $P$}]\label{cor:continuity_of_P}
	Let $\homogeneousmap\in\mathrm{Lip}(\Suno;\R^2)$ and suppose that  
	$(\varphi_\indiceP)_\indiceP
\subset\mathrm{Lip}(\Suno;\R^2)$ is 
such that 
$$
\varphi_\indiceP\rightarrow\homogeneousmap\quad\mbox{uniformly}\quad\mbox{and }
\quad \sup_{\indiceP \in \mathbb N} \Vert \dot\varphi_\indiceP\Vert_1<+\infty.
$$ Then $P(\varphi_\indiceP)\rightarrow P(\homogeneousmap)$ as $\indiceP \to +\infty$.
\end{Corollary}

In what follows it is convenient to consider  
the relaxation 
\begin{align}\label{Plateau_rel}
	\rilP(\gamma)
:=\inf\left\{\liminf_{\indiceP
\rightarrow +\infty}P(\varphi_\indiceP): 
\varphi_\indiceP\in \mathrm{Lip}(\Suno;\R^2), \varphi_\indiceP\rightarrow\gamma\text{ strictly }BV(\Suno;\R^2)
\right\} \quad
 \forall \gamma\in BV(\Suno;\R^2)
\end{align}
of $P$ with respect to the 
strict convergence in $BV$ of the boundary datum. It is well known that the infimum in \eqref{Plateau_rel} is taken on a non-empty class of approximation maps. Moreover, by \eqref{eq:int_Jac_radius}, also $\rilP$ is invariant by rescaling, i.e. $\rilP(\gamma)=\rilP(\gamma_r)$.

\begin{Lemma}
	Let $\varphi\in \mathrm{Lip}(\Suno;\R^2)$. Then $\rilP(\varphi)=P(\varphi)$.
\end{Lemma}

\begin{proof}
	If $(\varphi_\indiceP)\subset 
\mathrm{Lip}(\Suno;\R^2)$ is a sequence converging to 
$\varphi$ strictly $BV(\Suno;\R^2)$, then by 
Lemma \ref{lem:strict_implies_uniform_1}  $\varphi_\indiceP
\rightarrow \varphi$ 
uniformly on $\Suno$ as $\indiceP\rightarrow +\infty$. 
Moreover, the strict convergence guarantees 
that the total variations of $\varphi_\indiceP$ are equibounded. So, 
thanks to Corollary \ref{cor:continuity_of_P}, \begin{align}\label{cont_subs}
	P(\varphi_\indiceP)\rightarrow P(\varphi)
	\end{align}
as $\indiceP \to +\infty$.	Since this holds
 for any sequence $(\varphi_\indiceP)$ as above,  the thesis 
follows.
\end{proof}

\begin{Lemma}\label{lem_plateaurel}
	Let $\gamma\in SBV(\Suno;\R^2)$ have a finite number of jump points $z_i\in \Suno$, $i=1,\dots,n$. Let $\widetilde\gamma:\Suno\rightarrow\R^2$ be the Lipschitz map 
in \eqref{gamma_tilde}
(with $\Suno$ identified with $[0,2\pi]$). Then 
	\begin{align}
	\rilP(\gamma)=P(\widetilde\gamma).
	\end{align}
\end{Lemma}
\begin{proof}
	Let $(\varphi_\indiceP)_\indiceP
\subset\mathrm{Lip}(\Suno;\R^2)$ be a 
sequence converging strictly to $\gamma$. Let us consider a not-relabeled 
subsequence of $(\varphi_\indiceP)_\indiceP$; by Corollary \ref{corollario_conv} there are a further subsequence $(\varphi_{\indiceP_j})_j$ and 
Lipschitz reparametrizations $\gamma_{\indiceP_j}=\varphi_{\indiceP_j}
\circ h_{\indiceP_j}\in\mathrm{Lip}(\Suno;\R^2)$ 
of $\varphi_{\indiceP_j}$ such that $\gamma_{\indiceP_j}\rightarrow
\widetilde\gamma\circ \repa$ 
uniformly as $j\rightarrow +\infty$, 
 for some Lipschitz 
homeomorphism $\repa:\Suno\rightarrow\Suno$.  Moreover, since by Lemma \ref{strict implies uniform 2}(b) the 
reparametrization maps $\repa_{\indiceP_j}$ can be chosen with uniformly bounded 
Lipschitz constants, it follows that $\gamma_{\indiceP_j}$ 
have uniformly 
bounded total variations. Hence it follows from Corollary 
\ref{cor:continuity_of_P} that 
$P(\gamma_{\indiceP_j})\rightarrow P(
\widetilde\gamma\circ \repa)$ as $j\rightarrow +\infty$. 
On the other hand, by Proposition \ref{independence_of_P} we also have
	$P(\varphi_{\indiceP_j})\rightarrow P(\widetilde\gamma)$ as $j\rightarrow +\infty$.	Finally, since this argument holds for any 
subsequence of $(\varphi_\indiceP)$, we conclude that the whole sequence satisfies
	$P(\varphi_\indiceP)\rightarrow P(\widetilde\gamma)$, and therefore $\rilP(\gamma)=P(\widetilde\gamma)$.
\end{proof}

As a consequence of the argument in the proof of Lemma \ref{lem_plateaurel}, we easily infer the following continuity property:

\begin{Corollary}\label{cor_continuity_P}
	Let $\gamma\in SBV(\Suno;\R^2)$ and $\widetilde\gamma$ be as in Corollary \ref{corollario_conv}, and assume that $(\varphi_\indiceP)_\indiceP\subset \mathrm{Lip}(\Suno;\R^2)$ is a sequence converging strictly to $\gamma$. Then
	$$\lim_{\indiceP
\rightarrow +\infty}P(\varphi_\indiceP)=\rilP(\gamma)=P(\widetilde\gamma).$$ 
\end{Corollary}

Furthermore, we can refine the previous corollary as follows:

\begin{Corollary}\label{cor_continuity_P2}
	Let $\gamma,\gamma_\indiceP\in SBV(\Suno;\R^2)$, $\indiceP
\geq1$, be maps as in Corollary \ref{corollario_conv}. Assume that $(\gamma_\indiceP)$
 converges to $\gamma$ strictly $BV(\Suno;\R^2)$. Then 
	$$\lim_{\indiceP
\rightarrow +\infty}\rilP(\gamma_\indiceP)=\rilP(\gamma).$$
\end{Corollary}
\begin{proof}
By Corollary \ref{cor_continuity_P} and the density of $\text{Lip}(\Suno;\R^2)$ in $BV(\Suno;\R^2)$ with respect to the strict convergence, for 
all $\indiceP\geq1$ we can find $\varphi_\indiceP
\in \text{Lip}(\Suno;\R^2)$ such that 
$$\|\gamma_\indiceP-\varphi_\indiceP\|_{1}
+\left||\dot\varphi_\indiceP|(\Suno)-|\dot\gamma_\indiceP|(\Suno)\right|
+\left|P(\varphi_\indiceP)-\rilP(\gamma_\indiceP)\right|\leq 
\frac{1}{\indiceP}.$$
Hence the sequence $(\varphi_\indiceP)$ converges to $\gamma$ 
strictly $BV(\Suno;\R^2)$, and by the triangle inequality and  
Corollary \ref{cor_continuity_P} we conclude
$$\lim_{\indiceP\rightarrow +\infty}\rilP(\gamma_\indiceP)=\rilP(\gamma).$$ 
\end{proof}

\section{Relaxation on piecewise Lipschitz maps jumping on a curve}
\label{sec:relaxation_of_area_on_piecewise_Lipschitz_maps}

Recalling that $R=[a,b]\times[-1,1]$, 
consider 
$R^+=\{(x_1,x_2)\in R:x_2>0\}$
and  $R^-=\{(x_1,x_2)\in R:x_2<0\}$.
\begin{Definition}[\textbf{Piecewise Lipschitz map}]
We say that a map $u : \rettangolo \to \R^2$ 
is piecewise
Lipschitz if $u\in BV(R;\R^2)$ and
$u\in\mathrm{Lip}(R^-;\R^2) \cap \mathrm{Lip}(R^+; \R^2)$.
\end{Definition}

Thus $\jump_u \subseteq [a,b] \times \{0\}$;
 we define $u^\pm: [a,b] \times \{0\} \to \R^2$ the traces of $u_{|R^\pm}$,
which are Lipschitz maps. Set $\intervallounitario=[0,1]$ and define $\affinesurface:[a,b]\times \intervallounitario\rightarrow\R^3$ the 
affine interpolation surface spanning $\mathrm{graph}(u^\pm)=\{(t,u^\pm(t)): t\in[a,b]\}\subset\R\times\R^2=\R^3$, namely
\begin{equation}\label{eq:affine_interpolation_surface}
\affinesurface
(t,s)=(t,su^+(t)+(1-s)u^-(t)) =: (t, \affinesurfacelastcomp(t,s))
\qquad \forall (t,s) \in [a,b] \times I.
\end{equation}

\begin{Remark}\label{semicart integrand}
For a (semicartesian) 
map $\Phi:[a,b]\times[c,d]\rightarrow\R^3$ of the form $\Phi(t,\sigma)=(t,\phi(t,\sigma))=(t,\phi_1(t,\sigma),\phi_2(t,\sigma))$, the area integrand is given by
$$
|\partial_t\Phi\wedge\partial_\sigma\Phi|=\sqrt{|\partial_\sigma\phi_1|^2+|\partial_\sigma\phi_2|^2+(\partial_t\phi_1\partial_\sigma\phi_2-\partial_\sigma\phi_1\partial_t\phi_2)^2}=\sqrt{|\partial_\sigma\phi|^2+|J\phi|^2}.
$$
\end{Remark}

The main result of this section is the following:
\begin{Theorem}
[\textbf{Relaxed area of piecewise Lipschitz maps: straight
jump}]
\label{teo:main_thm}
Let $u : R \to\R^2$ be a piecewise Lipschitz map.
Then 
\begin{equation}\label{eq:ril_area}
\overline{\mathcal{A}}_{BV}(u, R)={\mathcal{A}}(u, R^+)+{\mathcal{A}}(u, R^-)+\int_{[a,b]\times
\intervallounitario}|\partial_t \affinesurface
\wedge\partial_s \affinesurface|~dtds.
\end{equation}
\end{Theorem}
\noindent

Notice that the Lipschitz
regularity of $u$ on $R^\pm$ ensures that the 
area functional has the classical expression 
$$
{\mathcal{A}}(u, R^\pm)=\int_{R^\pm}\sqrt{1+|\nabla u|^2+|\mathrm{det}\nabla u|^2}~dx;
$$
therefore,
the singular contribution produced by the relaxation in 
\eqref{eq:ril_area} is 
given by the area of $\affinesurface$.

We divide the proof of \eqref{eq:ril_area} in two parts:
the lower bound 
(Proposition \ref{prop:lower_bound}) and the 
upper bound (Proposition
\ref{prop:upper_bound}). 

\begin{Proposition}[\textbf{Lower bound for} \eqref{eq:ril_area}]
\label{prop:lower_bound}
Let $u:R\rightarrow\R^2$ be a
piecewise Lipschitz map, and
$(v_k)\subset C^1(\rettangolo;\R^2) \cap BV(\rettangolo;\R^2)$ be a sequence 
converging to $u$ strictly $BV(R;\R^2)$. Then
\begin{equation}\label{eq:lower_bound_piecewise_lip}
\liminf_{k\rightarrow+\infty}\mathcal{A}(v_k, R)\geq{\mathcal{A}}(u, R^+)+{\mathcal{A}}(u, R^-)+\int_{[a,b]\times
\intervallounitario}|\partial_t
\affinesurface
\wedge\partial_s\affinesurface|~dtds.
\end{equation}
\end{Proposition}
\begin{proof}
Fix $\eps\in(0,1)$. We have 
$$
\begin{aligned}
 \liminf_{k\rightarrow +\infty}\mathcal{A}(v_k, R)
&\geq 
\liminf_{k\rightarrow +\infty}\mathcal{A}(v_k, R\setminus 
(\segmentoingrassato)
)+\liminf_{k\rightarrow +\infty}\mathcal{A}(v_k, 
\segmentoingrassato
)
\\
&\geq \mathcal{A}(u, R\setminus 
(\segmentoingrassato))+\liminf_{k\rightarrow +\infty}\mathcal{A}(v_k, 
\segmentoingrassato),
\end{aligned}
$$
where in the last inequality we used \cite[Theorem 3.7]{AD}.
Sending $\eps$ to $0^+$, by dominated convergence 
it follows 
$\mathcal{A}(u, R\setminus (\segmentoingrassato))\rightarrow{\mathcal{A}}(u, R^+)+{\mathcal{A}}(u, R^-)$,  so 
\eqref{eq:lower_bound_piecewise_lip} will be proven provided
we show that
\begin{align}\label{eq:main_estimate}
\lim_{\eps\rightarrow 0^+}\liminf_{k\rightarrow +\infty}\mathcal{A}(v_k, 
\segmentoingrassato
)\geq
\int_{[a,b]\times \intervallounitario}|\partial_t
\affinesurface
\wedge\partial_s\affinesurface|~dtds.
\end{align}
Consider the maps 
$$
\graphofmapvk^\eps:R\rightarrow\R^3,
\quad
\graphofmapvk^\eps(t,\secondavariabileinR)=(t,v_k(t,\eps \secondavariabileinR)),
$$
and the associated
integer multiplicity 2-currents in $\R^3$
$$
\currvk^\eps
= 
{\graphofmapvk^\eps}_\sharp[\![R]\!].
$$
Notice that, neglecting the term $1 + \vert \partial_{x_1} v_k\vert^2$, we get
\begin{equation}\label{eq:neglecting}
\begin{aligned}
\mathcal{A}(v_k, \segmentoingrassato)
\geq& 
\int_{\segmentoingrassato}\sqrt{|\partial_{x_2}v_k|^2+|Jv_k|^2}~dx
\\
= & \int_{R}|\partial_t\graphofmapvk^\eps\wedge\partial_\sigma\graphofmapvk^\eps|~dt
d\secondavariabileinR
=|\currvk^\eps|,
\end{aligned}
\end{equation}
where we used Remark \ref{semicart integrand}, and 
$|\cdot|$ stands for the mass current. 
Consider also the maps
\begin{equation}\label{eq:def_U_eps}
U_{\pm}^\eps:R^\pm\rightarrow\R^3,
\quad
\graphofmapu_{\pm}^\eps(t,\secondavariabileinR)=(t,u(t,\eps 
\secondavariabileinR)),
\end{equation}
and the current 
\begin{align}\label{def S_eps}
S_\eps= 
\affinesurface_\sharp[\![[a,b]\times \intervallounitario]\!]
+
{\graphofmapu_+^\eps}_\sharp[\![R^+]\!]
+
{\graphofmapu_-^\eps}_\sharp[\![R^-]\!],
\end{align}
see Fig. \ref{fig:currents}. 
We want now prove the following crucial inequality:
\begin{align}\label{eq:semicontinuity_of_mass}
\liminf_{k\rightarrow +\infty}|\currvk^\eps|\geq|S_\eps|.
\end{align}
To show \eqref{eq:semicontinuity_of_mass} we prove that
$\currvk^\eps$ 
are close to suitable currents $\globalminimalmasscurrent_n^\eps$ independent of $k$
(see \eqref{eq:mathcal_T_eps}) which 
converge to $S_\eps$ as $n \to +\infty$.

For any $n\in\N$, $n \geq 1$, consider a 
partition $\{t_0=a,t_1,\ldots,t_{{n}+1}=b\}$
of $[a,b]$ in $(n+1)$ intervals $[t_{i-1}, t_i)$, with
\begin{align}\label{scelta t_i}
t_i - t_{i-1} \in \left(\frac{b-a}{2n}, 2 \frac{(b-a)}{n}\right).
\end{align}
Moreover, set
$$
R_i=[t_{i-1},t_i)\times[-1,1],
\quad
R_i^+=[t_{i-1},t_i)\times(0,1],
\quad
R_i^-=[t_{i-1},t_i)\times[-1,0),
$$
 and define the currents
\begin{equation}\label{eq:def_V_k_eps}
\currvkRi^\eps= {\graphofmapvk^\eps}_\sharp[\![R_i]\!],
\quad 
 S_{\eps,i}=
\affinesurface_\sharp[\![[t_{i-1},t_i)\times\intervallounitario]\!]
+
{\graphofmapu_+^\eps}_\sharp[\![R_i^+]\!]
+{\graphofmapu_-^\eps}_\sharp[\![R_i^-]\!],
\end{equation}
see Fig. \ref{fig:currents}. 
By definition, we have
\begin{equation}\label{quasi_disj spt}
\begin{aligned}
&\currvk^\eps
=\sum_{i=1}^{n+1} \currvkRi^\eps\qquad\mbox{and}\qquad
\mathcal H^2(\mathrm{spt}\currvkRi^\eps\cap\mathrm{spt}
\currvkRj^\eps)=0
\quad 
{\rm for}~ 
i\neq j,
\\
&S_\eps=\sum_{i=1}^{n+1}  S_{\eps,i}\qquad\mbox{and}\qquad
\mathcal H^2(\mathrm{spt} S_{\eps,i}\cap\mathrm{spt} S_{\eps,j})=
0\quad {\rm for}~ 
i\neq j.
\end{aligned}
\end{equation}
Furthermore,
\begin{equation}\label{bordoS_ieps}
\begin{aligned}
\partial  S_{\eps,i}=&-\Big({U^\varepsilon_-}_\sharp[\![\{t_{i-1}\}\times[-1,0)]\!]+\affinesurface_\sharp[\![\{t_{i-1}\}\times\intervallounitario]\!]+{U^\varepsilon_+}_\sharp[\![\{t_{i-1}\}\times(0,1]]\!]\Big)
\\
&-{U^\varepsilon_+}_\sharp[\![(t_{i-1},t_i)\times\{1\}]\!]
\\
&+\Big({U^\varepsilon_-}_\sharp[\![\{t_{i}\}\times[-1,0)]\!]+\affinesurface_\sharp[\![\{t_{i}\}\times\intervallounitario]\!]+{U^\varepsilon_+}_\sharp[\![\{t_{i}\}\times(0,1]]\!]\Big)
\\
&+{U^\varepsilon_-}_\sharp[\![(t_{i-1},t_i)\times\{-1\}]\!].
\end{aligned}
\end{equation}
Now, for fixed $i\in\{1,\ldots,n\}$ , set
\begin{align*}
	&\gamma_{-,i}^{u,\eps}(\secondavariabile)=u(t_i,\eps \secondavariabile)\quad&\forall \secondavariabile\in[-1,0),
\\
	&\gamma_{+,i}^{u,\eps}(\secondavariabile)=u(t_i,\eps \secondavariabile)\quad&\forall \secondavariabile\in(0,1],
\\
	&\gamma^0_i(s)=su^+(t_i)+(1-s)u^-(t_i)\quad&\forall s\in 
\intervallounitario,
\\
	&
\Lambda^{\pm,\eps}_{u,i}(t)=(t,u(t,\pm\varepsilon ))
\quad &\forall t\in [t_{i-1},t_{i}],
\end{align*}
and define $\gamma_i^{u,\eps}:[-1,1]\rightarrow \R^2$ as in \eqref{concatenation} where $\widetilde \gamma^-$, $S$, and $\widetilde \gamma^+$ 
are replaced  by
$\gamma_{-,i}^{u,\eps}$, 
$\gamma_i^0$ 
and $\gamma_{+,i}^{u,\eps}$ in the order, after a rescaling on $\left[-1,-\frac{1}{3}\right]$, $\left[-\frac{1}{3},\frac{1}{3}\right]$, and $\left[\frac{1}{3},1\right]$, respectively, as in the statement of Lemma 
\ref{strict implies uniform 2}. 
Also, define $\Gamma_i^{u,\eps}:[-1,1]\rightarrow (\{t_i\}\times\R^2)$ as 
\begin{align*}
	\Gamma_i^{u,\eps}(\sigma)
:=(t_i,\gamma_i^{u,\eps}(\sigma))\qquad \forall \sigma\in[-1,1].
\end{align*}
Using the definition of $U_\pm^\varepsilon$ 
and $\mappaaffine$, by \eqref{bordoS_ieps} we infer
\begin{align}\label{bordoS_ieps2}
\partial  S_{\eps,i}
=
-{\Gamma_{i-1}^{u,\eps}}_\sharp[\![[-1,1]]\!]
-
{\Lambda^{+,\eps}_{u,i}}_\sharp[\![(t_{i-1},t_i)]\!]
+
{\Gamma_i^{u,\eps}}_\sharp[\![[-1,1]]\!]
+{\Lambda^{-,\eps}_{u,i}}_\sharp[\![(t_{i-1},t_i)]\!].
\end{align}
Moreover, set 
\begin{align*}
	&\gamma_{k,i}^\eps(\secondavariabile)=
v_k(t_i,\eps \secondavariabile),\qquad 
\Gamma_{k,i}^\eps(\sigma)=(t_i,\gamma_{k,i}^\eps(\sigma))\qquad&\forall \secondavariabile
	\in[-1,1],
\\
	&
\Lambda^{\pm, \eps}_{k,i}(t)=(t,v_k(t,\pm\varepsilon ))
\quad &\forall t\in [t_{i-1},t_{i}].
\end{align*}
By definition of $\currvkRi^\varepsilon$ in 
\eqref{eq:def_V_k_eps}, we also have
\begin{align}\label{bordoS_kieps}
\partial \currvkRi^\varepsilon=-{\Gamma_{k,i-1}^\eps}_\sharp[\![[-1,1]]\!]
-{\Lambda^{+,\eps}_{k,i}}_\sharp[\![(t_{i-1},t_i)]\!]
+{\Gamma_{k,i}^\eps}_\sharp[\![[-1,1]]\!]
+{\Lambda^{-,\eps}_{k,i}}_\sharp[\![(t_{i-1},t_i)]\!].
\end{align}
We now define  $F_{k,i}^\eps\in \mathcal D_2(\R^3)$ 
as a suitable affine interpolation between 
$\partial \currvkRi^\eps$ and $\partial S_{\eps,i}$, see Fig. \ref{fig:currents}.
First observe that by Lemma \ref{lem:inheritance}, we can suppose that, for 
our choice of $\eps$ 
and 
$\{t_1,\ldots,t_{n}\}$, 
there exists a (not relabeled) subsequence of $(v_k)_k$, such that 
\begin{align}
	&v_k(t_i,\eps\cdot)\rightarrow u(t_i,\eps\cdot)\quad\mbox{strictly }BV([-1,1];\R^2)\quad\forall i=1,\ldots,{n}\label{strict on ti},
\\
	&v_k(\cdot,\pm\eps)\rightarrow u(\cdot,\pm\eps)\quad\mbox{strictly }BV([a,b];\R^2)\label{strict on eps}.
\end{align}
In particular, by Lemma \ref{strict implies uniform 2}, we know that 
there are increasing Lipschitz 
bijections $\repa^\eps_{k,i},\repa^\eps_i
:[-1,1]\rightarrow[-1,1]$ such that 
$\gamma_{k,i}^\eps\circ \repa^\eps_{k,i}\rightarrow
\gamma_i^{u,\eps}\circ \repa^\eps_i$ 
uniformly in $[-1,1]$ as $k \to +\infty$.

For $i=1,\dots,n$, we define 
$$
\begin{aligned}
\Phi^\varepsilon_{k,i} (\secondavariabileinR,s)
:=s&(\Gamma_{k,i}^\eps\circ \repa^\eps_{k,i}(\secondavariabileinR))
+(1-s)(\Gamma_i^{u,\eps}\circ \repa^\eps_{i}(\secondavariabileinR)),
\qquad (\secondavariabileinR,s)\in [-1,1]\times I, 
\\
\Psi^{\pm,\varepsilon}_{k,i}(t,s)
:=s&\Lambda^{\pm,\eps}_{k,i}(t)+(1-s)
\Lambda^{\pm,\eps}_{{u,i}}(t),
\qquad
(t,s)\in [t_{i-1},t_i]\times I.
\end{aligned}
$$ 
Therefore we set 
\begin{equation}\label{eq:def_F_eps_k_i}
\begin{aligned}
F_{k,i}^\eps=&-{\Phi^\varepsilon_{k,i-1}}_\sharp[\![[-1,1]\times \intervallounitario]\!]-{\Psi^{+,\varepsilon}_{k,i}}_\sharp[\![[t_{i-1},t_i]\times \intervallounitario]\!]
\\
&+{\Phi^\varepsilon_{k,i}}_\sharp[\![[-1,1]\times \intervallounitario]\!]+
{\Psi^{-,\varepsilon}_{k,i}}_\sharp[\![[t_{i-1},t_i]\times \intervallounitario]\!].
\end{aligned}
\end{equation}
 In particular, from \eqref{bordoS_ieps2} and \eqref{bordoS_kieps}, a 
direct check shows that 
\begin{align}\label{boundary of F}
\partial F_{k,i}^\eps=\partial \currvkRi^\eps-\partial  S_{\eps,i}. 
\end{align}
Eventually, 
we let $\minimalmasscurrent_{\eps,i}$ 
be an integer multiplicity $2$-current of $\R^3$ with minimal mass and boundary 
$\partial  S_{\eps,i}$ (the existence of 
$\minimalmasscurrent_{\eps,i}$ is guaranteed, for instance, by \cite[Theorem 8.3.3]{KP})
 and set 
\begin{equation}\label{eq:mathcal_T_eps}
\globalminimalmasscurrent_n^\eps := \sum_{i=2}^{n} \minimalmasscurrent_{\eps,i}. 
\end{equation}
Note carefully that we do not sum over $i$ from $1$ to $n+1$, but only from $2$ to $n$. 
In particular, setting $S^n_{\eps} =
S_\eps-S_{\eps,1}-S_{\eps,n+1}$, we have
\begin{equation}\label{314}
\partial \globalminimalmasscurrent_n^\eps=
\partial  S_{\eps}^n=-{\Gamma_1^{u,\eps}}_\sharp[\![[-1,1]]\!]
+{\Gamma_n^{u,\eps}}_\sharp[\![[-1,1]]\!]-
{\Lambda_{u}^{+,\eps}}_\sharp[\![[t_{1},t_n]]\!]
+{\Lambda_{u}^{-,\eps}}_\sharp[\![[t_{1},t_n]]\!],
\end{equation}
where
$$
\Lambda_{u}^{\pm,\eps}(t):=(t,u(t,\pm\eps)), \qquad t\in(t_1,t_n).
$$
Thus, we have
$$
|\currvkRi^\eps|\geq|\currvkRi^\eps-F_{k,i}^\eps|-|F_{k,i}^\eps|
\geq
|\minimalmasscurrent_{\eps,i}|-|F_{k,i}^\eps|\quad\rm{for }\; i=2,\ldots,n,
$$
where we used the minimality of $\minimalmasscurrent_{\eps,i}$ and \eqref{boundary of F}.
By summing up, 
using \eqref{quasi_disj spt}, we get\footnote{In \eqref{inequality sum} we had to remove the first and last term of the sum, because condition 
(i) can be false for $i=1$ and $i=n+1$, 
since the strict convergence is inherited only on almost every line, as stated in Lemma \ref{lem:inheritance}.}
\begin{align}\label{inequality sum}
|\currvk^\eps|=\sum_{i=1}^{n+1} |\currvkRi^\eps|
\geq\sum_{i=2}^{n} |\currvkRi^\eps|\geq \sum_{i=2}^{n} 
|\minimalmasscurrent_{\eps,i}|-\sum_{i=2}^{n} |F_{k,i}^\eps|\geq
|\globalminimalmasscurrent_n^\eps|-\sum_{i=2}^{n} |F_{k,i}^\eps|.
\end{align}
Therefore,
\begin{equation}\label{eq:Tau_eps_n}
\liminf_{k\rightarrow +\infty}|\currvk^\eps|\geq|\globalminimalmasscurrent_n^\eps|-\sum_{i=2}^{n}\limsup_{k\rightarrow +\infty} |F_{k,i}^\eps|.
\end{equation}
In order to obtain \eqref{eq:semicontinuity_of_mass}, 
we have to prove that:
\begin{enumerate}[(i)]
\item$|F_{k,i}^\eps| \rightarrow 0$ as $k\rightarrow+\infty$ for every $i=2,\ldots,n$;
\item$\globalminimalmasscurrent_n^\eps\rightharpoonup S_\eps$ as $n\rightarrow+\infty$,
\end{enumerate}
so that \eqref{eq:semicontinuity_of_mass} would follow by lower semicontinuity of the mass and \eqref{eq:Tau_eps_n}.

(i). Since $\gamma_{k,i}^\eps\circ \repa^\eps_{k,i}\rightarrow \gamma_i^{u,\eps}
\circ \repa^\eps_i$ uniformly in $[-1,1]$ as $k \to +\infty$, 
also $\Gamma_{k,i}^\eps\circ \repa^\eps_{k,i}\rightarrow \Gamma_i^{u,\eps}\circ\repa^\eps_{i}$ uniformly; 
moreover, by Lemma \ref{lem:strict_implies_uniform_1} and thanks to \eqref{strict on eps}, $v_k(\cdot,\pm\eps)\rightarrow u(\cdot,\pm\eps)$ uniformly on $[t_{i-1},t_i]$, and the 
same holds for $\Lambda_{k,i}^{\pm, \eps}$ and $\Lambda_{u,i}^{\pm,\eps}$.
 Finally, by \eqref{strict on ti} and \eqref{strict on eps}, and recalling also Lemma \ref{strict implies uniform 2} (b), the $L^1$-norm of the derivative of $\Gamma_{k,i}^\eps\circ \repa^\eps_{k,i}$ and of $\Lambda^{\pm,\eps}_{{k,i}}$ is uniformly bounded with respect to $k$.
Hence (i) readily follows from the definition of $F^\eps_{k,i}$ 
in \eqref{eq:def_F_eps_k_i}
and Remark \ref{affineint_rem}.

\begin{figure}[t]
    \centering
    \includegraphics[scale=0.45]{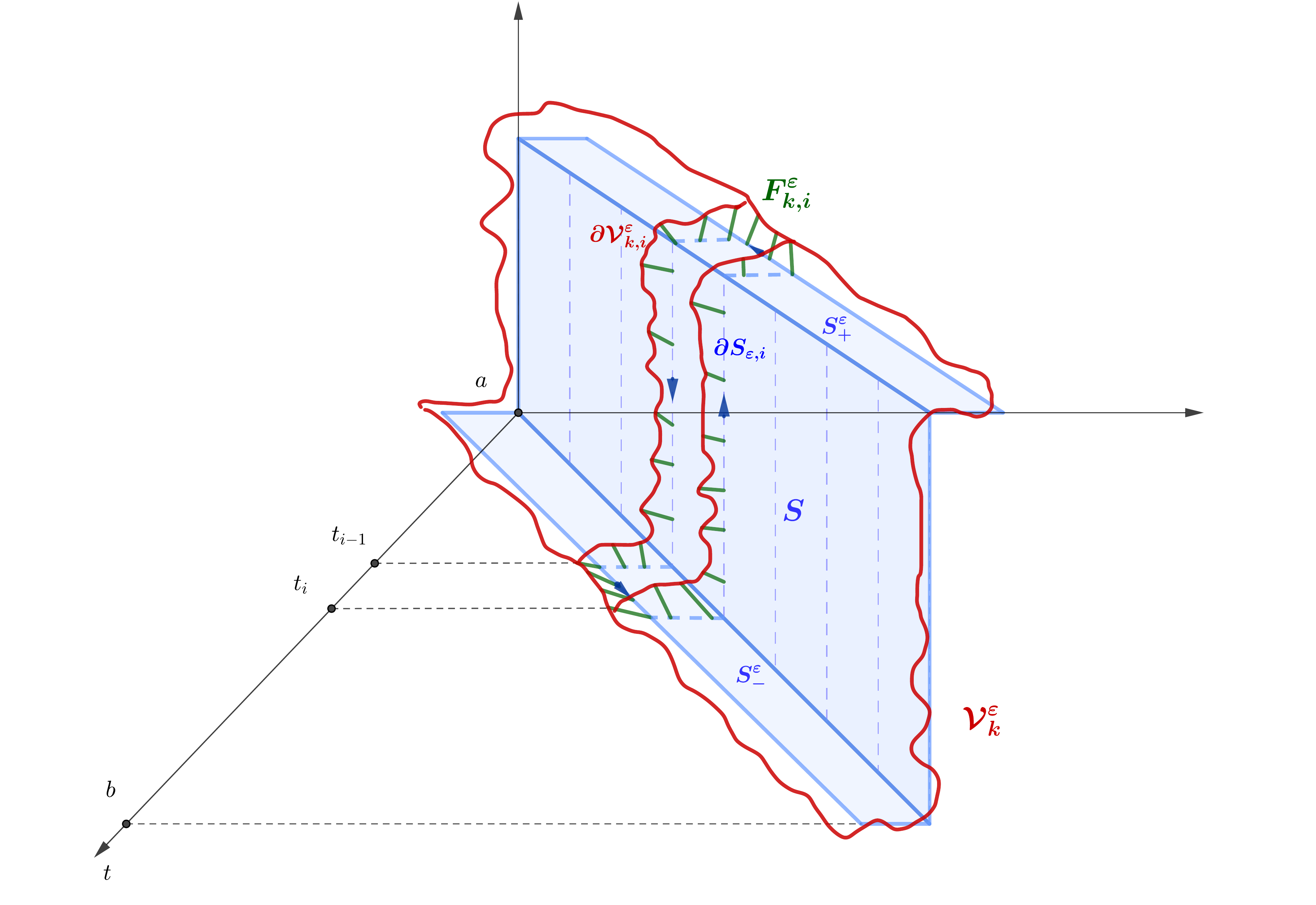}
    \caption{Here $S = \affinesurface_\sharp[\![[a,b]\times\intervallounitario]\!]$, 
$S^\eps_\pm = U^\eps_{\pm \sharp}[\![\rettangolo^\pm]\!]$. The horizontal and vertical axes span the target space $\R^2$. The approximating current $\currvk^\eps$ is depicted in bold, as well as the boundary of its restriction to $R_i$, i.e. the current $\partial\currvkRi^\eps$. The current $\partial S_{\eps,i}$ is depicted with the oriented dotted straight lines, while $F^\eps_{k,i}$ is the oriented surface obtained as the union of the short 
segments connecting $\partial\currvkRi^\eps$ and $\partial S_{\eps,i}$. Finally, 
for simplicity, we depict with straight 
segments the graph of $u^\pm$ and the (semi)graph of $u$ on $\{(t,\sigma):\sigma=\pm\eps\}$, but it is worth to remember that they are graph of Lipschitz maps. 
}
    \label{fig:currents}
\end{figure}

(ii). First observe that 
$\partial \globalminimalmasscurrent_n^\eps$  has mass uniformly bounded 
with respect to $n$. Indeed by \eqref{314} 
\begin{align*}
|\partial\globalminimalmasscurrent_n^\eps|&=|\partial {S}^n_\eps|\leq |\dot\gamma_1^{u,\eps}|([-1,1])+|\dot\gamma_n^{u,\eps}|([-1,1])+\int_a^b\!\sqrt{1+|\partial_tu(t,\eps)|^2}dt+\int_a^b\!\sqrt{1+|\partial_tu(t,-\eps)|^2}dt
\\
&\leq C(\eps,\|u\|_\infty,\mathrm{lip}(u_{|R^+}),\mathrm{lip}(u_{|R^-})).
\end{align*}
 Moreover, by minimality of $\globalminimalmasscurrent_n^\eps$ and \eqref{quasi_disj spt}, 
$|\globalminimalmasscurrent_n^\eps|\leq|{S}^n_\eps|\leq
|S_\eps|$, hence the sequence 
$\big(\globalminimalmasscurrent_n^\eps\big)_{n}$ is compactly supported in $\R^3$ 
and has bounded mass and bounded boundary mass. Then, by \cite[Theorem 8.2.1]{KP}, we have
$$
\globalminimalmasscurrent_n^\eps\rightharpoonup S_\eps
\Longleftrightarrow\|\globalminimalmasscurrent_n^\eps- S_\eps\|_F 
\rightarrow 0 \quad\mbox{as }n\rightarrow +\infty,
$$
where $\|\cdot \|_F$ stands for the flat norm. Then, we are reduced to show that $\|\globalminimalmasscurrent_n^\eps- S_\eps\|_F 
\rightarrow 0 $ as $n\rightarrow +\infty$.
Notice that
\begin{equation}\label{eq:T_meno_S}
\|\globalminimalmasscurrent_n^\eps- S_\eps\|_F\leq\sum_{i=2}^{n}
\|\minimalmasscurrent_{\eps,i}-  S_{\eps,i}\|_F+\| S_{\eps,1}\|_F+\| S_{\eps,n+1}\|_F,
\end{equation}
where, by definition of flat norm (see \cite[Sec. 5.1.3]{GMS1}),
$$
\|\minimalmasscurrent_{\eps,i}-  S_{\eps,i}\|_F\leq\inf\{|G_i^\eps|: 
G_i^\eps \mbox{ integer multiplicity 3-current s.t. } \partial G_i^\eps=
\minimalmasscurrent_{\eps,i}-  S_{\eps,i}\}.
$$
Observe that the class of competitors in the
above minimum problem 
is non empty, since it contains 
the affine interpolation current between $\minimalmasscurrent_{\eps,i}$ and $ S_{\eps,i}$.
So, pick a 3-current $G_i^\eps$ such that $\partial G_i^\eps
=\minimalmasscurrent_{\eps,i}-  S_{\eps,i}$;
then 
$$
|G_i^\eps|\leq C|\partial G_i^\eps|^{\frac{3}{2}}
$$
by the isoperimetric inequality \cite[Theorem 7.9.1]{KP}, for an absolute positive constant $C>0$.
For $i=2,\ldots,n$, we have
\begin{equation}\label{eq:appl_isoperimetric}
\|\minimalmasscurrent_{\eps,i}-  S_{\eps,i}\|_F\leq|G_i^\eps|\leq 
C|\partial G_i^\eps|^{\frac{3}{2}}=C|\minimalmasscurrent_{\eps,i}-  
S_{\eps,i}|^\frac{3}{2}\leq C\left(|\minimalmasscurrent_{\eps,i}|^\frac{3}{2}+| S_{\eps,i}|^\frac{3}{2}\right)\leq 2C| S_{\eps,i}|^\frac{3}{2},
\end{equation}
where in the last inequality we used the minimality of $\minimalmasscurrent_{\eps,i}$.
Now let us prove that $| S_{\eps,i}|\leq\frac{C}{n}$ for every $i=1,\ldots,n+1$, where $C$ is a constant independent of $n$. We start observing that 
\begin{align*}
& |\affinesurface_\sharp[\![[t_{i-1},t_i)\times\intervallounitario]\!]|
=\int_{[t_{i-1},t_i]\times \intervallounitario}|\partial_t
\affinesurface\wedge\partial_s
\affinesurface|~dtds
\\
=& \int_{t_{i-1}}^{t_i}\int_
\intervallounitario
|(1,s\dot{u}^++(1-s)\dot{u}^-)\wedge(0,u^+-u^-)|~dtds
\\
\leq& \int_{t_{i-1}}^{t_i}\int_\intervallounitario
\left(|u^+-u^-|+\left|(s\dot{u}_1^++(1-s)\dot{u}_1^-)(u_2^+-u_2^-)-(s\dot{u}_2^++(1-s)\dot{u}_2^-)(u_1^+-u_1^-)\right|\right)~dtds
\\
\leq\,& \frac{C_1}{n}\|u^+-u^-\|_{L^\infty(a,b)}+\frac{C_2}{n}\|u^+-u^-\|_{L^\infty(a,b)}\left(\|\dot{u}^+\|_{L^\infty(a,b)}+\|\dot{u}^-\|_{L^\infty(a,b)}\right)\\
=& \,\frac{C}{n},
\end{align*}
where we used \eqref{scelta t_i}.
Moreover, recalling \eqref{eq:def_U_eps},
we have
\begin{equation}\label{stima U}
\begin{aligned}
&|{\graphofmapu_\pm^\eps}_\sharp[\![R_i^\pm]\!]|
=\int_{R_i^\pm}|\partial_t\graphofmapu_{\pm}^\eps
\wedge\partial_\secondavariabileinR\graphofmapu_{\pm}^\eps|~dtd\secondavariabileinR
\\
=&\int_{R_i^\pm}|(1,\partial_tu(t,\eps \secondavariabileinR))
\wedge(0,\eps\partial_\secondavariabileinR u(t,\eps \secondavariabileinR))|~dtd\secondavariabileinR
\\
\leq&\,\eps\int_{R_i^\pm}|\partial_\secondavariabileinR u(t,\eps \secondavariabileinR)|~dtd\secondavariabileinR
+\eps 
\int_{R_i^\pm}|\partial_tu_1(t,\eps \secondavariabileinR)\partial_\secondavariabileinR u_2(t,\eps \secondavariabileinR)-
\partial_tu_2(t,\eps \secondavariabileinR)\partial_\secondavariabileinR u_1(t,\eps \secondavariabileinR)|~dtd\secondavariabileinR
\\
\leq&\,\eps\frac{C_3}{n}\left(\|\nabla u\|_{L^\infty(R^\pm)}
+\|\nabla u\|^2_{L^\infty(R^\pm)}\right)
\\
=&\,\frac{C\eps}{n}.
\end{aligned}
\end{equation}
Thus, 
$$| S_{\eps,i}|\leq|\affinesurface_\sharp[\![[t_{i-1},t_i)\times\intervallounitario]\!]|
+
|{\graphofmapu_+^\eps}_\sharp[\![R_i^+]\!]|
+|{\graphofmapu_-^\eps}_\sharp[\![R_i^-]\!]|
\leq \frac{C}{n},$$
as claimed.
Finally, by definition of flat norm and the isoperimetric inequality, $\| S_{\eps,i}\|_F\leq| S_{\eps,i}|^\frac{3}{2}$ for $i=1,\ldots,n+1$, so that, from \eqref{eq:appl_isoperimetric} and \eqref{eq:T_meno_S}, we obtain
$$
\|\globalminimalmasscurrent_n^\eps- S_\eps\|_F\leq C (n-1)\frac{1}{n^\frac{3}{2}}
+\frac{C}{n^\frac{3}{2}}\leq \frac{C}{n^\frac{1}{2}}+\frac{C}{n^\frac{3}{2}} \rightarrow 0.
$$
This concludes the proof of (ii) and hence of \eqref{eq:semicontinuity_of_mass}.

\medskip

We are now in a position to show \eqref{eq:main_estimate}.
{}From 
\eqref{eq:neglecting} and \eqref{eq:semicontinuity_of_mass},
\begin{align}\label{eps thesis}
\liminf_{k \to +\infty}
{\mathcal{A}(v_k, \segmentoingrassato)
\geq 
\liminf_{k\to +\infty}} 
|\currvk^\eps| \geq 
\vert S_\eps\vert.
\end{align}
As in \eqref{stima U}, we have
$$
|{\graphofmapu_\pm^\eps}_\sharp[\![R^\pm]\!]|\leq\eps\left(\|\nabla u\|_{L^\infty(R^\pm)}+\|\nabla u\|^2_{L^\infty(R^\pm)}\right)
\rightarrow 0\quad\mbox{as }\eps\rightarrow 0^+,
$$
so, from \eqref{eps thesis} and \eqref{def S_eps}, we conclude
$$
\lim_{\eps \to 0^+}
\liminf_{k \to +\infty}
\mathcal{A}(v_k, \segmentoingrassato)
\geq 
\lim_{\eps \to 0^+}
\vert S_\eps\vert
=
|\affinesurface_\sharp[\![[a,b]\times \intervallounitario]\!]|
=\int_{[a,b]\times \intervallounitario}|\partial_t
\affinesurface
\wedge\partial_s\affinesurface|~dtds.
$$
\end{proof}

\begin{Proposition}[\textbf{Upper bound for} \eqref{eq:ril_area}]\label{prop:upper_bound}
Let $u:R\rightarrow\R^2$ be a
piecewise Lipschitz map. Then there exists a sequence $(v_k)_k
\subset C^1(R;\R^2)$ converging to $u$ strictly $BV(R;\R^2)$ such that 
\begin{align}\label{limsup ineq}
\limsup_{k\rightarrow+\infty}\mathcal{A}(v_k, R)\leq{\mathcal{A}}(u, R^+)+
{\mathcal{A}}(u, R^-)+\int_{[a,b]\times
\intervallounitario}|\partial_t
\affinesurface
\wedge\partial_s \affinesurface|~dtds.
\end{align}
\end{Proposition}

\begin{proof}
Although $v_k$ needs to be of class $C^1$, we claim 
that it suffices to build $v_k$ just Lipschitz continuous. 
Indeed, assume that $(v_k)_k\subset 
W^{1,\infty}(R;\R^2) \cap C^1(R; \R^2)$ 
converges to $u$ strictly $BV(R;\R^2)$ and \eqref{limsup ineq} holds. Consider, for all $k \in \mathbb N$, a sequence $\big(v^k_h\big)_h
\subset C^1(R;\R^2)$ approaching $v_k$ in $W^{1,2}(R;\R^2)$ 
as $h\rightarrow +\infty$. In particular, we get the $L^1$-convergence of 
all  minors of $\nabla v^k_{h}$ to the corresponding 
ones of $\nabla v_k$. Then, by dominated convergence,
\begin{align}\label{densityenergy}
\lim_{h\rightarrow+\infty}\mathcal{A}(v^k_h;R)=\mathcal{A}(v_k, R).
\end{align}
Hence, by a diagonal argument, we find a sequence $\big(v^k_{h_k}\big)_k$ 
converging to $u$ strictly $BV(R;\R^2)$ such that \eqref{limsup ineq} holds for $v^k_{h_k}$ in place of $v_k$.

Set for simplicity $\eps=\eps_k=\frac{1}{k}$,
and define the sequence $(v_\eps) \subset \mathrm{Lip}(R;\R^2)$ as
\begin{equation}\label{eq:def_ve}
v_\eps(t,\secondavariabileinR) :=\begin{cases}
u(t,\secondavariabileinR)& (t,\secondavariabileinR) \in R\setminus (\segmentoingrassato),
\\
\frac{\eps+\secondavariabileinR}{2\eps}u(t,\eps)+
\frac{\eps-\secondavariabileinR}{2\eps}u(t,-\eps)& (t,\secondavariabileinR) \in [a,b] \times (-\eps, \eps).
\end{cases}
\end{equation}
First, let us check that $v_\eps\rightarrow u$ strictly $BV(R;\R^2)$ 
as $\eps \to 0^+$. Clearly, $v_\eps\rightarrow u$ in $L^1(R;\R^2)$. Hence, by lower semicontinuity of the total variation, it is 
enough to show that $$\limsup_{\eps\rightarrow 0^+}\int_R|\nabla v_\eps|dtd\sigma\leq|Du|(R),$$
which in turn reduces to prove 
$$\limsup_{\eps\rightarrow 0^+}\int_{\segmentoingrassato}|\nabla v_\eps|dtd\sigma
\leq|Du|([a,b]\times \{0\}),$$
since $$\int_{R\setminus (\segmentoingrassato)}|\nabla v_\eps|dtd\sigma=\int_{R\setminus (\segmentoingrassato)}|\nabla u|dtd\sigma\rightarrow \int_{R^+}|\nabla u|dtd\sigma+\int_{R^-}|\nabla u|dtd\sigma\quad \mbox{as } \eps\rightarrow 0^+.$$
For almost every $t\in[a,b]$ and every $\sigma\in[-\eps,\eps]$, one has
$$
\partial_t v_\eps(t,\secondavariabileinR)=\frac{\eps+\secondavariabileinR}{2\eps}\partial_tu(t,\eps)
+\frac{\eps-\secondavariabileinR}{2\eps}\partial_tu(t,-\eps),\quad
\partial_\secondavariabileinR v_\eps(t,\sigma)=\frac{1}{2\eps}(u(t,\eps)-u(t,-\eps)).
$$
Thus, 
setting $M:=\max\{\mathrm{lip}(u_{|{R^-}}),\mathrm{lip}(u_{|{R^+}})\}$, we get
\begin{align*}
\int_{\segmentoingrassato}|\nabla v_\eps|~dtd\secondavariabileinR
\leq& \int_{\segmentoingrassato}{|\partial_t v_\eps(t,\secondavariabileinR)|~dtd\secondavariabileinR
+\int_{\segmentoingrassato}|\partial_\secondavariabileinR v_\eps(t,\secondavariabileinR)|}~dtd\secondavariabileinR
\\
\leq& M\int_{\segmentoingrassato}~dtd\secondavariabileinR+\int_{\segmentoingrassato}\frac{1}{2\eps}|u(t,\eps)-u(t,-\eps)|~dtd\secondavariabileinR
\\
=&M(b-a)2\eps+\int_a^b|u(t,\eps)-u(t,-\eps)|~dt
\\
&\stackrel{\eps \to 0^+}{\longrightarrow}\int_a^b|u^+(t)-u^-(t)|~dt
=|Du|([a,b]\times \{0\}).
\end{align*} 
Furthermore, since $u$ is piecewise Lipschitz, we have 
$$
\mathcal{A}(v_\eps;R\setminus 
\segmentoingrassato
)=\mathcal{A}(u,R\setminus \segmentoingrassato)\rightarrow\mathcal{A}(u,R^+)+\mathcal{A}(u,R^-)\quad\mbox{ as }  \eps\rightarrow 0^+.
$$
So it remains to prove that
\begin{align}\label{liminf}
\limsup_{\eps\rightarrow 0^+}\mathcal{A}(v_\eps;
\segmentoingrassato)\leq\int_{[a,b]\times
\intervallounitario}|\partial_t\affinesurface\wedge\partial_s
\affinesurface|~dtds.
\end{align}
Let us linearly 
reparametrize $\affinesurface$ on $R=[a,b]\times[-1,1]$, 
namely consider $Y$, having the same image as $\affinesurface$, given by
$$
Y(t,\secondavariabileinR)=(t,\widehat{Y}(t,\secondavariabileinR))
=\left(t,\frac{1+\secondavariabileinR}{2}u^+(t)+\frac{1-\secondavariabileinR}{2}u^-(t)\right),
\qquad (t, \secondavariabileinR) \in R.
$$
Now, using the trivial inequality $\sqrt{1 + a^2 + b^2 + c^2}
\leq 1 + \vert a\vert + \sqrt{b^2+c^2}$, we find
\begin{align}\label{estim}
\mathcal{A}(v_\eps;\segmentoingrassato)&\leq\int_{\segmentoingrassato}~dtd\secondavariabileinR
+\int_{
\segmentoingrassato
}|\partial_tv_\eps|~dtd\secondavariabileinR
+\int_{\segmentoingrassato}\!\!\sqrt{|\partial_\secondavariabileinR v_\eps|^2+|Jv_\eps|^2}dtd\secondavariabileinR\nonumber
\\
&=\eps(b-a)+\eps\int_{R}|\partial_t\tilde{v}_\eps|~dtd
\secondavariabileinR
+\int_{R}\sqrt{|\partial_\secondavariabileinR\tilde{v}_\eps|^2+|J\tilde{v}_\eps|^2}~dtd\secondavariabileinR,
\end{align}
where $\tilde{v}_\eps:R\rightarrow\R^2$ is defined as $\tilde{v}_\eps(t,\secondavariabileinR)=v_\eps(t,\eps\secondavariabileinR)$. A direct computation 
based in \eqref{eq:def_ve}
gives
\begin{align*}
&\partial_t\tilde{v}_\eps(t,\secondavariabileinR)=\frac{1+\secondavariabileinR}{2}\partial_tu(t,\eps)+\frac{1-\secondavariabileinR}{2}\partial_tu(t,-\eps)\quad\mbox{for a.e. }t\in[a,b]\quad\forall\sigma\in[-1,1]
\\
&\partial_\secondavariabileinR\tilde{v}_\eps(t,\secondavariabileinR)=\eps\partial_\secondavariabileinR v_\eps(t,\eps\secondavariabileinR)=\frac{u(t,\eps)-u(t,-\eps)}{2}\quad\mbox{for a.e. }t\in[a,b]\quad\forall\sigma\in[-1,1].
\end{align*}
Then we have
\begin{align*}
&\partial_t\tilde{v}_\eps(t,\secondavariabileinR)
\rightarrow\frac{1+\secondavariabileinR}{2}\dot{u}^+(t)+\frac{1-\secondavariabileinR}{2}\dot{u}^-(t)=\partial_t\widehat{Y}(t,\secondavariabileinR)\quad\mbox{a.e. in } \rettangolo,
\\
&\partial_\secondavariabileinR\tilde{v}_\eps(t,\secondavariabileinR)\rightarrow\frac{u^+(t)-u^-(t)}{2}=\partial_\secondavariabileinR\widehat{Y}(t,\sigma)\quad\mbox{a.e. in }R.
\end{align*}
Since $\partial_\secondavariabileinR\widehat{Y}$ and $\partial_t\widehat{Y}$ are in $L^\infty(R;\R^2)$, by dominated convergence 
we can pass to the limit in \eqref{estim}  as $\eps \to 0^+$,
so that, using Remark \ref{semicart integrand}, we obtain \eqref{liminf}.
\end{proof}

\begin{Remark}\label{rem:link_Mucci}
After having proved the upper bound inequality 
in Proposition \ref{prop:upper_bound}, we readily infer that $\overline {\mathcal A}_{BV}(u,R)<+\infty$. Hence 
Proposition \ref{prop:lower_bound} can be deduced from an
argument independently developed in \cite{Mc2}, 
based on the theory of Cartesian currents \cite{GMS2}. Indeed, 
consider $T_u:=G_u+S$, where $G_u$ is the $2$-current on $R\times\R^2$ carried by the graph of $u$ and $S$ is the $2$-current on $R\times\R^2$ given by $S:=\tilde{X}_\sharp[\![[a,b]\times
\intervallounitario]\!]$, 
where 
$$
\tilde{X}(t,s):=(t,0,\affinesurfacelastcomp(t,s))=(t,0,su^+(t)+(1-s)u^-(t)), \qquad
t\in[a,b],  s\in 
\intervallounitario.
$$
Clearly, the mass ot $T_u$ is given by
\begin{align*}
|T_u|=|G_u|+|S|&={\mathcal{A}}(u, R^+)+{\mathcal{A}}(u, R^-)+\int_{[a,b]\times
\intervallounitario
}|\partial_t\tilde{X}\wedge\partial_s\tilde{X}|~dtds
\\
&={\mathcal{A}}(u, R^+)+{\mathcal{A}}(u, R^-)+\int_{[a,b]\times
\intervallounitario
}|\partial_t{\affinesurface}\wedge\partial_s{\affinesurface}|~dtds.
\end{align*}
Now 
we claim that $T_u$ is the unique Cartesian current on $R\times\R^2$ with minimal completely vertical lifting associated to $u$, 
according to \cite[Definition 3.1]{Mc2}. Borrowing the notation from \cite{Mc2}, this definition is given by 
imposing that the mixed components of
$T_u$ are the minimal lifting measures $\mu_i^j[u]$ associated to $u$ in the sense of Jerrard and Jung 
\cite{JJ}. 
Once the claim is proven, 
by the lower semicontinuity of the mass and the continuity of the lifting measures with respect to the strict convergence (see \cite[Theorem 1.1]{JJ}), we deduce
$$
|T_u|\leq\overline{\mathcal{A}}_{BV}(u;R),
$$
i.e., inequality \eqref{eq:lower_bound_piecewise_lip}.

In order to show the claim, we start to prove
that $T_u\in\mathrm{cart}(\rettangolo,\R^2)$. For this, it is 
enough to see that $(\partial T_u)\mres ( R\times\R^2)=0$: We get  
$$
(\partial G_u)\mres (R\times\R^2)
=\affinesurfacelastcomp^-_\sharp[\![[a,b]]\!]-\affinesurfacelastcomp^+_\sharp[\![[a,b]]\!]=-\partial \tilde{X}_\sharp[\![[a,b]\times
\intervallounitario
]\!]
=-(\partial S)\mres (R\times\R^2),
$$
where $\affinesurfacelastcomp^\pm(t):=(t,0,u^\pm(t))$, $t\in[a,b]$.
Next, what remains to prove is that the vertical component of $T_u$ 
is the minimal completely vertical lifting associated to $u$.
To this purpose, denote by $x=(x^1,x^2)$ the 
(horizontal) variable of $R$, $y=(y^1,y^2)$ the vertical 
variable of $\R^2$ and $u=(u^1,u^2)$ the components of $u$. We have to check that 
\begin{equation}\label{eq:JJ_to_check}
\mu^j_i[T_u]=\mu_i^j[u]\quad\forall i,j=1,2,
\end{equation}
 where $\mu_i^j[T_u]:=T_u\mres ((-1)^i dx^{\bar{i}}\wedge dy^j)$. By \cite[Theorem 2.2]{JJ}, 
for every $f\in C^\infty_c(R\times\R^2)$, 
$$
\int_{R\times\R^2}f(x,y)d\mu^i_j[u]=\int_{R^+\cup R^-}f(x,u(x))\partial_iu^j dx+\int_a^b\left(\int_0^1f(t,0,\affinesurfacelastcomp(t,s))ds\right)({u^j}^+-{u^j}^-)\delta_{i2}~dt.
$$
On the other hand, setting $\omega(x,y):=(-1)^if(x,y)dx^{\bar{i}}\wedge dy^j$, we have 
\begin{align*}
\int_{R\times\R^2}f(x,y)d\mu^i_j[T_u]&=\int_{R^+\cup R^-}f(x,u(x))\partial_iu^jdx+\int_{\tilde{X}([a,b]\times \intervallounitario)}\omega
\\
&=\int_{R^+\cup R^-}f(x,u(x))\partial_iu^jdx+\int_{[a,b]\times \intervallounitario}\omega(\tilde{X}(t,s))d\tilde{X}^{\bar{i}j},
\end{align*}
where, if $\tilde{X}=(\tilde{X}^1_1,\tilde{X}^1_2,\tilde{X}^2_1,\tilde{X}^2_2)$, then $d\tilde{X}^{\bar{i}j}=d\tilde{X}^{\bar{i}}_1\wedge d\tilde{X}^j_2$. 
Notice that $d\tilde{X}^{\bar{i}j}=0$ if $\bar{i}=2$ and $d\tilde{X}^{1j}=({u^j}^+-{u^j}^-)~dt\wedge ds$, so we get
\begin{align*}
\int_{[a,b]\times \intervallounitario}\omega(\tilde{X}(t,s))d\tilde{X}^{\bar{i}j}&=\int_{[a,b]\times
\intervallounitario
}(-1)^if(\tilde{X}(t,s))({u^j}^+-{u^j}^-)\delta_{i2}~dt\wedge ds
\\
&=\int_a^b\left(\int_0^1f(t,0,\affinesurfacelastcomp(t,s))ds\right)({u^j}^+-{u^j}^-)\delta_{i2}~dt,
\end{align*}
and \eqref{eq:JJ_to_check} follows.
\end{Remark}

\medskip

\subsection{Extension of Theorem \ref{teo:main_thm}}

The validity of Theorem \ref{teo:main_thm} 
is guaranteed also when the two traces $u^\pm$ of $u$ on 
$[a,b]\times\{0\}$ coincide on some subset of $[a,b]\times\{0\}$. In particular, \eqref{eq:ril_area} extends
to maps $u$ whose jump set $\jump_u$ is a subset of $[a,b]\times\{0\}$. 
However, the situation is different when the jump set is curvilineous. Specifically,
assume $\Omega \subset \R^2$ 
is a bounded open and connected set,
and:
\begin{itemize}
\item[(H1)]
 $\jumpset = \alpha([a,b])
\subset\Om$ is a simple 
curve, arc-length parametrized by 
$\alpha:[a,b]\rightarrow \Om$ of class $C^2$ 
 injective in $[a,b)$;
\item[(H2)] If $\alpha(a)=\alpha(b)$, then 
$\dot\alpha(a^+)=\dot \alpha(b^-)$ and $\ddot \alpha(a^+)=\ddot\alpha(b^-)$;
\item[(H3)]
$u\in W^{1,\infty}(\Omega\setminus\jumpset; \R^2)$; 
as usual, we denote by $u^{\pm}$ the traces of $u$ on $\jumpset$, satisfying
 $u^\pm\in\mathrm{Lip}(\jumpset;\R^2)$.  
\end{itemize}


Again, we introduce the 
affine interpolation surface 
$\affinesurface:[a,b]\times \intervallounitario\rightarrow\R^3$ 
spanning $\mathrm{graph}(u^\pm)=
\{(t,u^\pm(\alpha(t))): t\in[a,b]\}\subset\R\times\R^2=\R^3$, namely
\begin{equation}\label{eq:affine_curve}
\affinesurface(t,s)=(t,su^+(\alpha(t))+(1-s)u^-(\alpha(t)))
\qquad \forall (t,s) \in [a,b] \times I.
\end{equation}
\begin{Theorem}[\textbf{Relaxed area of piecewise Lipschitz maps: curved
jump}] 
\label{teo:main_thm_curve} 
 Suppose\\
(H1)-(H3). Then
\begin{equation}
\label{eq:relaxed_area_of_piecewise_Lipschitz_maps:curved_jump}
\overline{\mathcal{A}}_{BV}(u, \Omega)=\int_{\Omega\setminus\jumpset}|\mathcal{M}(\nabla u)|~dx+\int_{[a,b]\times
\intervallounitario
}|\partial_t
\mappaaffine
\wedge\partial_s\mappaaffine|~dtds.
\end{equation}
\end{Theorem}

\begin{Remark}
The image of the map $\affinesurface$ sits in $\R^3$ and it 
is not exactly the interpolation surface 
which closes the holes in the graph of $u$,
which is instead given by
\begin{align}\label{true interpolation map}
\Psi(t,s)=(\alpha(t),su^+(\alpha(t))+(1-s)u^-(\alpha(t)))
\in \R^4 
\qquad \forall t \in [a,b] \times I.
\end{align}
 However, since $\vert \dot
\alpha\vert=1$,
\begin{equation}\label{eq:two_areas_equal}
\int_{[a,b]\times \intervallounitario}
|\partial_t\Psi\wedge\partial_s\Psi|~dtds
=\int_{[a,b]\times
\intervallounitario
}|\partial_t
\mappaaffine
\wedge\partial_s
\mappaaffine
|~dtds.
\end{equation}
\end{Remark}

To prove Theorem \ref{teo:main_thm_curve}, we borrow from \cite{BePaTe:15} some notation. We denote by $x=(x_1,x_2)$ 
coordinates in $\Omega$ and by $(t,
\secondavariabileinR)$ 
coordinates in $R=[a,b]\times[-1,1]$. 
 Since $\jumpset$ 
is simple and of class $C^2$, we can find $\delta>0$ and a $C^1$-diffeomorphism $\Lambda:R_\delta\rightarrow
\Lambda(R_\delta)$, where $R_\delta=[a,b]\times[-\delta,\delta]$ and 
$\Lambda(R_\delta)\subset\Omega$ is a curvilineous strip containing $\jumpset$ of width $2\delta$. 
Explicitely we have
\begin{align}\label{def_Lambda}
\Lambda(t,\secondavariabileinR)
=\alpha(t)+\secondavariabileinR\dot\alpha(t)^\perp\quad\forall (t,\secondavariabileinR)\in R_\delta,
\end{align}
with $\dot\alpha(t)^\perp$ the 
counter-clockwise $\frac{\pi}{2}$-rotation of $\dot\alpha(t)$.
For $(x_1,x_2)\in \Lambda(R_\delta)$, we can write the 
inverse 
$\Lambda^{-1}(x_1,x_2)=(t(x_1,x_2),\secondavariabileinR(x_1,x_2))$, where:
\begin{itemize}
	\item $\secondavariabileinR(x_1,x_2)=d_{\jumpset}(x_1,x_2)$ is the signed distance\footnote{The sign of $d_\jumpset$ is determined by the orientation induced on $\jumpset$ by $\alpha$, so that $d_\jumpset>0$ in the part of $\Lambda(R_\delta)$ which is pointed by $\dot\alpha^\perp$.} of $(x_1,x_2)$ from $\jumpset$;
	\item $t(x_1,x_2)$ is the 
unique number in $[a,b]$ such that $\alpha(t(x_1,x_2))=\pi_\jumpset(x_1,x_2)$, 
where $\pi_\jumpset(x_1,x_2)=(x_1,x_2)-d_\jumpset(x_1,x_2)\nabla d_\jumpset(x_1,x_2)$ is the orthogonal projection on $\jumpset$.
\end{itemize}
Since $\alpha$ is of class $C^2$, we have that $\sigma$ is of class $C^2$ as well and $t$ is of class $C^1$ on $\overline{\Lambda(R_\delta)}$. Moreover, for $(x_1,x_2)\in\overline{\Lambda(R_\delta)}$, we have
\begin{align}
	&|\nabla \secondavariabileinR(x_1,x_2)|
=|\nabla d_\jumpset(x_1,x_2)|=1\label{grad dist},
\\
&|\nabla t(x_1,x_2)|
=1+\delta\|\nabla d_\jumpset\|_\infty\leq1+C\delta\label{grad t}.
\end{align}
We divide the proof of Theorem \ref{teo:main_thm} 
in two parts, the lower and the upper bound inequalities.

\begin{Proposition}[\textbf{Lower bound for} \eqref{eq:relaxed_area_of_piecewise_Lipschitz_maps:curved_jump}]
\label{prop:lower_bound_curve}
	Let $u:\Omega\rightarrow\R^2$ as in Theorem \ref{teo:main_thm_curve} and $(v_k)\subset C^1(\Omega;\R^2)$ be a sequence converging to $u$ 
	strictly $BV(\Omega;\R^2)$. Then
\eqref{eq:lower_bound_piecewise_lip}
holds with $\affinesurface$ in \eqref{eq:affine_curve}.
\end{Proposition}
\begin{proof}
	It is enough to show that
	\begin{align}\label{eq:main_estimate_curve}
		\lim_{\eps\rightarrow 0^+}\liminf_{k\rightarrow +\infty}\mathcal{A}(v_k, \Lambda(
		\segmentoingrassato))\geq\int_{[a,b]\times
			\intervallounitario
		}|\partial_t
\mappaaffine
\wedge\partial_s\mappaaffine
		|~dtds.
	\end{align}
	We start by defining the maps $\Psi_k^\eps:R\rightarrow\R^4$ and $\Psi_{\pm}^\eps:R^\pm\rightarrow\R^4$ given by
	$$
	\Psi_k^\eps(t,\secondavariabileinR)
=(\Lambda(t,\eps \secondavariabileinR),v_k(\Lambda(t,\eps \secondavariabileinR))),\qquad\Psi_{\pm}^\eps(t,\secondavariabileinR)
=(\Lambda(t,\eps \secondavariabileinR),u(\Lambda(t,\eps \secondavariabileinR))).
	$$
	Introduce the following integer multiplicity 2-currents in $\R^4$:
	$$
	\currvk^\eps= {\Psi_k^\eps}_\sharp[\![R]\!],\quad S^\eps=\Psi_\sharp[\![[a,b]\times
	\intervallounitario
	]\!]+{\Psi_-^\eps}_\sharp[\![R^-]\!]+{\Psi_+^\eps}_\sharp[\![R^+]\!],
	$$
	where $\Psi$ is defined 
in \eqref{true interpolation map}.
	Using that $Av\wedge Aw=\mathrm{det}A\,v\wedge w$ for any $A\in\R^{2\times2}$ and $v,w\in\R^2$, by direct computation, 
we have 
$$
|\partial_t\Psi_k^\eps\wedge\partial_\secondavariabileinR \Psi_k^\eps|^2=\eps^2|\partial_t\Lambda(t,\eps \secondavariabileinR)\wedge\partial_\secondavariabileinR\Lambda(t,\eps \secondavariabileinR)|^2
\Big[1+|\nabla v_k(\Lambda(t,\eps \secondavariabileinR))|^2+|Jv_k(\Lambda(t,\eps \secondavariabileinR))|^2\Big].
$$
	Hence, making the change of variable 
$x=\Lambda(t,\eps \secondavariabileinR)$, we obtain
	$$
	\mathcal{A}(v_k, \Lambda(\segmentoingrassato))=\int_{\Lambda(
		\segmentoingrassato)}{|\mathcal{M}(\nabla v_k)}|~dx=\int_{R}|\partial_t\Psi_k^\eps\wedge\partial_\secondavariabileinR\Psi_k^\eps|~dtd\secondavariabileinR=|\mathcal V_k^\eps|.
	$$
	We notice that $|{\Psi_\pm^\eps}_\sharp[\![R^\pm]\!]|\rightarrow 0$ as $\eps \to 0^+$, as in \eqref{stima U}, where $\|\nabla u\|_{L^\infty(R^\pm)}$ is replaced with $\|u\|_{W^{1,\infty}(\Omega)}$ and it is used that $|\ddot\alpha|\leq C$. 
Therefore, recalling also \eqref{eq:two_areas_equal},
	$$
	\lim_{\eps\rightarrow 0^+}|S^\eps|=|\Psi_\sharp[\![[a,b]\times
	\intervallounitario
	]\!]|=\int_{[a,b]\times
		\intervallounitario
	}|\partial_t\Psi\wedge\partial_s\Psi|~dtds=\int_{[a,b]\times
		\intervallounitario
	}|\partial_t
	\affinesurface
	\wedge\partial_s
	\affinesurface
	|~dtds.
	$$
	So it is enough to show
$\liminf_{k\rightarrow +\infty}|\currvk^\eps|\geq|S^\eps|$,
which can be proved proceeding 
as in the proof of Proposition \ref{prop:lower_bound}, 
once we have checked that $v_k\circ \Lambda(\cdot,\varepsilon\cdot)\rightarrow u\circ \Lambda(\cdot,\varepsilon\cdot)$ strictly $BV(R;\R^2)$. This is a straightforward computation, and we omit the details.
\end{proof}

\begin{Proposition}[\textbf{Upper bound for}
\eqref{eq:relaxed_area_of_piecewise_Lipschitz_maps:curved_jump}]\label{upper bound curve}
Let $u:\Omega\rightarrow\R^2$ be 
as in Theorem \ref{teo:main_thm_curve}. Then, there exists a sequence
$(v_k)\subset C^1(\Omega;\R^2)$ 
converging to $u$ strictly $BV(\Omega;\R^2)$ and
such that \eqref{limsup ineq}
holds with $\affinesurface$ in \eqref{eq:affine_curve}.
\end{Proposition}
\begin{proof}
For simplicity, we assume that  $\alpha(a)\neq\alpha(b)$ (the case of closed curves is simpler and the following proof can be straightforwardly adapted). 
We start by fixing $\eta>0$ small enough 
and we extend the curve $\alpha$ to $[a-\eta,b+\eta]$ in a $C^2$-way, 
so that $\jumpset^\eta:=\alpha([a-\eta,b+\eta])\subset\Om$, 
keeping the validity of (H1) on $\jumpset^\eta$ .
 With this extension, we can assume (by choosing a different $\delta$ if necessary) that $\Lambda$ in \eqref{def_Lambda} is defined on $R^\eta:=[a-\eta,b+\eta]\times[-\delta,\delta]$.
We observe that 
\begin{equation}\label{boundaryJ}
u^+(\alpha(t))=u^-(\alpha(t))\qquad \text{ for all }t\in [a-\eta,a]\cup[b,b+\eta].
\end{equation}
Now, set $\eps=\frac{1}{k}$ and, for $k$ large enough, 
\begin{align*}
&\Delta^a_\eps:=\{x\in \Lambda([a-\eps,a]\times[-\eps,\eps]):|\secondavariabileinR(x)|\leq t(x)-a+\eps\},
\\
&\Delta^b_\eps:=\{x\in \Lambda([b,b+\eps]\times[-\eps,\eps]):|\secondavariabileinR(x)|\leq b+\eps-t(x)\}.
\end{align*}
We define the 
recovery sequence $(v_\eps)\subset\mathrm{Lip}(\Omega;\R^2)$ as
\begin{equation}
v_\eps(x)=
\begin{cases}
\frac{\eps+\secondavariabileinR(x)}{2\eps}u\big(\Lambda(t(x),\eps)\big)+\frac{\eps-\secondavariabileinR(x)}{2\eps}
u\big(\Lambda(t(x),-\eps)\big)&\mbox{in }\Lambda(\segmentoingrassato),
\\
u(x)&\mbox{in  }\Omega\setminus \Big(\Lambda(\segmentoingrassato))
\cup \overline\Delta^a_\eps \cup \overline \Delta^b_\eps \Big).
\label{v_k}
\end{cases}
\end{equation}
In order to define $v_\eps$ in $\Delta^a_\eps\cup\Delta^b_\eps$ it is sufficient to observe that, by \eqref{boundaryJ}, the restriction of $v_\eps$ on $\partial \Delta_\eps^a$ and $\partial \Delta_\eps^b$ is Lipschitz continuous with Lipschitz constant 
bounded by 
$\|u\|_{W^{1,\infty}}$.
 Hence, we can take 
a Lipschitz extension of 
$v_\eps$ in  $\Delta^a_\eps\cup\Delta^b_\eps$
keeping the Lipschitz constant (up to a dimensional factor independent
of $\eps$).
Thus
\begin{align}\label{cont_Delta}
	\int_{\Delta^a_\eps\cup\Delta^b_\eps}|\mathcal M(\nabla v_\eps)|~dx
\rightarrow 0\qquad
\text{ as }\eps\rightarrow 0^+.
	\end{align}
Let us check that $v_\eps\rightarrow u$ strictly $BV(\Omega;\R^2)$ as $\eps
\to 0^+$. 
Clearly, $v_\eps\rightarrow u$ in $L^1(\Omega;\R^2)$, since $|\Lambda(
\segmentoingrassato)|\rightarrow 0$ and $|\Delta^a_\eps\cup\Delta^b_\eps|
\rightarrow0$. So, by \eqref{cont_Delta}, as in the proof of Proposition \ref{prop:upper_bound}, it is enough to show that 
$$
\limsup_{\eps\rightarrow 0^+}\int_{\Lambda(
\segmentoingrassato
)}|\nabla v_\eps|~dx\leq|Du|(\jumpset)=\int_a^b|u^+(\alpha(t))-u^-(\alpha(t))|~dt.
$$
Almost everywhere in $\Lambda(\segmentoingrassato)$, we have
\begin{align*}
\nabla v_\eps 
=&\frac{\eps+\secondavariabileinR}{2\eps}\nabla u(\Lambda(t,\eps))\partial_t\Lambda(t,\eps)\otimes\nabla t+\frac{\eps-\secondavariabileinR}{2\eps}\nabla u(\Lambda(t,-\eps))\partial_t\Lambda(t,-\eps)\otimes\nabla t
\\
&+\frac{1}{2\eps}{\nabla \secondavariabileinR}\otimes(u(\Lambda(t,\eps))-u(\Lambda(t,-\eps))).
\end{align*}
Therefore, 
\begin{align*}
|\nabla v_\eps|
\leq& \frac{1}{2\eps}\Big[(\eps+\secondavariabileinR)\|\partial_t\Lambda\|_\infty|\nabla u(\Lambda(t,-\eps))||\nabla t|+(\eps-\secondavariabileinR)\|\partial_t\Lambda\|_\infty|\nabla u(\Lambda(t,\eps))||\nabla t|
\\
&+|\nabla \secondavariabileinR||u(\Lambda(t,\eps))-u(\Lambda(t,-\eps))|\Big]
\\
\leq& \frac{1}{2\eps}\Big[2\eps \|u\|_{W^{1,\infty}}
\|\partial_t\Lambda\|_\infty(1+C\eps)+|u(\Lambda(t,\eps))-u(\Lambda(t,-\eps))|\Big],
\end{align*}
where we used \eqref{grad dist} and\eqref{grad t} with $\eps$ in place of $\delta$.
Thus, we get
\begin{align*}
\int_{\Lambda(\segmentoingrassato)}|\nabla v_\eps|~dx\leq& C(\delta)(1+C\eps)|\Lambda(
\segmentoingrassato)|\\
&+\frac{1}{2\eps}\int_{\Lambda(
\segmentoingrassato)}|u(\Lambda(t,\eps))-u(\Lambda(t,-\eps))|~dx
\\
=&o_\eps(1)+\frac{1}{2\eps}\int_{\Lambda(
\segmentoingrassato)}|u(\Lambda(t,\eps))-u(\Lambda(t,-\eps))|dx.
\end{align*}
Consider the last integral and perform the change of variable $x=(x_1,x_2)=\Lambda(t,\secondavariabileinR)$, with 
$$
|\mathrm{det}\nabla\Lambda(t,\secondavariabileinR)|=|\partial_t\Lambda\wedge\partial_\secondavariabileinR\Lambda|=|1+\secondavariabileinR\ddot\alpha\wedge\dot\alpha|=|1-\kappa_\jumpset\secondavariabileinR|,
$$
where $\kappa_\jumpset$ is the curvature of $\jumpset$. We get
\begin{align*}
& \frac{1}{2\eps}\int_{\Lambda(\segmentoingrassato)}|u(\Lambda(t,\eps))-u(\Lambda(t,-\eps))|dx 
=\frac{1}{2\eps}\int_{
\segmentoingrassato}|u(\Lambda(t,\eps))-u(\Lambda(t,-\eps))||1-\kappa_\jumpset\secondavariabileinR|dtd\secondavariabileinR
\\
\leq&\frac{1}{2\eps}\int_a^b\int_{-\eps}^\eps|u(\Lambda(t,\eps))-u(\Lambda(t,-\eps))|~dtd\secondavariabileinR+o_\eps(1)
=\int_a^b|u(\Lambda(t,\eps))-u(\Lambda(t,-\eps))|~dt+o_\eps(1)
\\
&\longrightarrow\int_a^b|u^+(\alpha(t))-u^-(\alpha(t))|~dt\quad\mbox{as }
\eps\rightarrow 0^+.
\end{align*}
It remains to prove \eqref{limsup ineq} with $\affinesurface$ in \eqref{eq:affine_curve}.
 To this 
purpose it is enough to show that
$$
\liminf_{\eps\rightarrow 0^+}\mathcal{A}(v_\eps;\Lambda(
\segmentoingrassato))\leq\int_{[a,b]\times
\intervallounitario
}|\partial_t
\affinesurface
\wedge\partial_s
\affinesurface
|~dtds.
$$
Let us define $\varphi_\eps:R\rightarrow\R^2$ as
$$
\varphi_\eps(t,\secondavariabileinR):=\frac{1+\secondavariabileinR}{2}u(\Lambda(t,\eps))+\frac{1-\secondavariabileinR}{2}u(\Lambda((t,-\eps))).
$$
Thus, for $x\in\Lambda(\segmentoingrassato)$
$$
v_\eps(x)=\varphi_\eps\left(t(x),\frac{\secondavariabileinR(x)}{\eps}\right)
$$
and, almost everywhere in $\Lambda(\segmentoingrassato)$,
\begin{align*}
\nabla v_\eps=\partial_t\varphi_\eps\nabla t+\frac{1}{\eps}\partial_\secondavariabileinR\varphi_\eps\nabla \secondavariabileinR,\qquad
Jv_\eps=\frac{1}{\eps}|\partial_t\varphi_\eps\wedge\partial_\secondavariabileinR\varphi_\eps||\nabla t\wedge\nabla \secondavariabileinR|,
\end{align*}
where from now on, $\nabla t$ and $\nabla \secondavariabileinR$ are evaluated at $x$, while $\partial_t\varphi_\eps$ and $\partial_\secondavariabileinR\varphi_\eps$ are evaluated at $\left(t(x),\frac{\secondavariabileinR(x)}{\eps}\right)$. Then, we get
\begin{align*}
|\mathcal{M}(\nabla v_\eps)|^2&=1+|\partial_t\varphi_\eps|^2|\nabla t|^2+\frac{2}{\eps}\partial_t\varphi_\eps\cdot\partial_\secondavariabileinR\varphi_\eps\nabla t\cdot\nabla \secondavariabileinR+\frac{1}{\eps^2}\Big[|\partial_\secondavariabileinR\varphi_\eps|^2|\nabla \secondavariabileinR|^2+|\partial_t\varphi_\eps\wedge\partial_\secondavariabileinR\varphi_\eps|^2|\nabla t\wedge\nabla \secondavariabileinR|^2\Big]
\\
&\leq1+|\partial_t\varphi_\eps|^2(1+o_\eps(1))+\frac{2}{\eps}|\partial_t\varphi_\eps\cdot\partial_\secondavariabileinR\varphi_\eps|(1+o_\eps(1))
\\
&\quad+\frac{1}{\eps^2}\Big[|\partial_\secondavariabileinR\varphi_\eps|^2+|\partial_t\varphi_\eps\wedge\partial_\secondavariabileinR\varphi_\eps|^2(1+o_\eps(1))\Big],
\end{align*}
where we used \eqref{grad dist} and\eqref{grad t} with $\eps$ in place of $\delta$. Now, since $o_\eps(1)\sim\eps$ and $\varphi_\eps$ is Lipschitz with Lipschitz constant independent of $\eps$, we obtain
\begin{align*}
&\mathcal{A}(v_\eps;\Lambda(\segmentoingrassato))
\\
\leq&\int_{\Lambda(\segmentoingrassato)}\sqrt{1+|\partial_t\varphi_\eps|^2+\frac{2}{\eps}|\partial_t\varphi_\eps\cdot\partial_\secondavariabileinR\varphi_\eps|+\frac{1}{\eps^2}\Big[|\partial_\secondavariabileinR\varphi_\eps|^2+|\partial_t\varphi_\eps\wedge\partial_\secondavariabileinR\varphi_\eps|^2(1+o_\eps(1))\Big]}~dx+o_\eps(1)
\\
\leq&\int_{\segmentoingrassato}\sqrt{1+|\partial_t\varphi_\eps|^2+\frac{2}{\eps}|\partial_t\varphi_\eps\cdot\partial_\secondavariabileinR\varphi_\eps|+\frac{1}{\eps^2}\Big[|\partial_\secondavariabileinR\varphi_\eps|^2+|\partial_t\varphi_\eps\wedge\partial_\secondavariabileinR\varphi_\eps|^2(1+o_\eps(1))\Big]}|1-\kappa_\jumpset\secondavariabileinR|~dtd\secondavariabileinR
\\
&+o_\eps(1),
\end{align*}
where we made the change of variable $x=\Lambda(t,\secondavariabileinR)$, and so $\partial_t\varphi_\eps$ and $\partial_\secondavariabileinR\varphi_\eps$ are 
computed at $\left(t,\frac{\secondavariabileinR}{\eps}\right)$.
Finally, by the change of variable $\frac{\secondavariabileinR}{\eps}\rightarrow \secondavariabileinR$, we get
\begin{align*}
&\mathcal{A}(v_\eps;\Lambda(\segmentoingrassato))
\\
\leq&\int_{R}\sqrt{o_\eps(1)+|\partial_\secondavariabileinR\varphi_\eps(t,\secondavariabileinR)|^2+|\partial_t\varphi_\eps(t,\secondavariabileinR)\wedge\partial_\secondavariabileinR\varphi_\eps(t,\secondavariabileinR)|^2(1+o_\eps(1))}|1-\kappa_\jumpset\eps \secondavariabileinR|~dtd\secondavariabileinR
\\
&+o_\eps(1) \longrightarrow\int_{[a,b]\times
\intervallounitario
}|\partial_t
\affinesurface
\wedge\partial_s
\affinesurface
|~dtds,
\end{align*}
where, to pass to the limit as $\eps \to 0^+$, 
we apply the dominated convergence theorem (as in the proof of Proposition \ref{prop:upper_bound}).
\end{proof}
We observe that Theorem \ref{teo:main_thm_curve} can be 
easily extended to the case of 
curves with one endpoint or both endpoints on $\partial\Om$.
Write:
\begin{itemize}
\item[(H4)] $\Omega$ is of class $C^1$, 
$\alpha:[a,b]\rightarrow \overline\Om$ is injective, arc-length parametrized,
of class $C^2$, $\alpha((a,b))\subset\Om$, and $\alpha$ hits $\partial \Om$ transversally at $\alpha(a),\alpha(b)$.
\end{itemize} 

\begin{Theorem}\label{teo_main_curve2}
Suppose (H3) and (H4).
Then \eqref{eq:relaxed_area_of_piecewise_Lipschitz_maps:curved_jump} holds
with $\affinesurface$ in \eqref{eq:affine_curve}.
\end{Theorem}

\begin{proof}
Lower bound: let $(v_k)\subset C^1(\Omega;\R^2)$ be a sequence converging to $u$ 
	strictly $BV(\Omega;\R^2)$. Fix $0<\rho<\frac{b-a}{2}$ and notice that $ \Lambda([a+\rho,b-\rho]\times [-\eps,\eps])\subset\Om$, for $\eps>0$ small enough. 
Then it is sufficient to show that 
		\begin{align}\label{eq:main_estimate_curve2}
		\lim_{\eps\rightarrow 0^+}
\liminf_{k\rightarrow +\infty}\mathcal{A}\Big(v_k, \Lambda(
		[a,b]\times [-\eps,\eps])\cap \Om\Big)
\geq\int_{[a+\rho,b-\rho]\times
			\intervallounitario
		}|\partial_t
\mappaaffine
\wedge\partial_s\mappaaffine
		|~dtds;
	\end{align}
since the lower bound will follow by the arbitrariness of 
$\rho>0$. After writing $\mathcal{A}(v_k, \Lambda(
	\segmentoingrassato)\cap \Om)\geq \mathcal{A}(v_k, \Lambda(
	[a+\rho,b-\rho]\times[-\eps,\eps]))$, the proof of \eqref{eq:main_estimate_curve2} is identical to that of \eqref{eq:main_estimate_curve}. 
	
Upper bound:  let us fix $\eta>0$ small enough
so that 
$B_{2\eta}(\alpha(a))$ and $B_{2\eta}(\alpha(b))$ are disjoint,
and consider 
$\Om^\eta:=\Om\cup B_{2\eta}(\alpha(a))\cup B_{2\eta}(\alpha(b))$. 
We extend the curve $\alpha$ 
(still calling $\alpha$ the extension)
in $\Om^\eta\setminus \Om$ in such a way that it satisfies
 (H4) in $\Om^\eta$, and so that
 it reaches the boundary of 
$B_{2\eta}(\alpha(a))\setminus \overline\Om$ 
and of $B_{2\eta}(\alpha(b))\setminus \overline\Om$
splitting both $B_{2\eta}(\alpha(a))\setminus \overline\Om$ and $B_{2\eta}(\alpha(b))\setminus \overline\Om$ in two connected components.
 If $\alpha$ 
is now defined on an interval of the form
$[a-\delta,b+\delta]$ with $\delta = \delta(\eta) >\eta$,
 and if we set
 $\Sigma^\delta
=\alpha([a-\delta,b+\delta])$,  we prescribe the traces $u^+$ and $u^-$ 
on $\Sigma^\delta$ in such a way that they are Lipschitz 
continuous and $u^+\circ\alpha=u^-\circ\alpha$ 
on $[a-\delta,a-\eta]\cup[b+\eta,b+\delta]$. 
 Finally we take a Lipschitz extension $u^\eta$ of $u$ 
on the four connected components of 
$B_{2\eta}(\alpha(a))\setminus \overline\Om\setminus \Sigma^\delta$ and of $B_{2\eta}(\alpha(b))\setminus \overline\Om\setminus \Sigma^\delta$. 
It turns out that 
$u^\eta\in W^{1,\infty}\big((B_{2\eta}(\alpha(a))
\cup B_{2\eta}(\alpha(b)))\setminus \jumpset^\eta; \R^2\big)$, where $\Sigma^\eta
=\alpha([a-\eta,b+\eta])\subset\Om^\eta$. 
Since the definition of $(u^\eta)^\pm$ 
is arbitrary, we can assume 
that  
\begin{align*}
	(u^\eta)^\pm(\alpha(t))=u^\pm(\alpha(a))\Big(
1-\frac{a-t}{\eta}\Big)\qquad\text{ for }t\in[a-\eta,a],
\\
	(u^\eta)^\pm(\alpha(t))=u^\pm(\alpha(b))\Big(1-\frac{t-b}{\eta}\Big)
\qquad\text{ for }t\in[b,b+\eta].
\end{align*}
For $\eps>0$ small enough, we see that $\Lambda_\eps:=\Lambda([a-\eta,b+\eta]\times[-\eps,\eps])\subset\Om^\eta$. Hence we define $v_k$ 
as in the proof of Proposition \ref{upper bound curve} 
with $\Om$ replaced by $\Om^\eta$ and $u$ replaced by $u^\eta$ 
(in particular, $v_\eps=u$ on 
$\Om\setminus\Lambda_\eps$). Finally, let us fix 
$\rho\in (0,\eta)$. We can write
\begin{equation*}
\begin{aligned}
\overline{\mathcal A}_{BV}(u, \Om)&
\leq \liminf_{\eps\rightarrow 0^+}
\mathcal A(v_\eps, \Om)
\\
&\leq \lim_{\eps\rightarrow 0^+} \int_{\Om\setminus  \Lambda_\eps}
|\mathcal M(\nabla u)|~dx+\liminf_{\eps\rightarrow 0^+}\int_{\Lambda([a-\rho,b+\rho]\times [-\eps,\eps])}|\mathcal M(\nabla v_\eps)|~dx
\\
&=\int_{\Omega}|\mathcal{M}(\nabla u)|~dx+\int_{[a-\rho,b+\rho]\times
	\intervallounitario
}|\partial_t
\mappaaffine
\wedge\partial_s\mappaaffine|~dtds,
\end{aligned}
\end{equation*}
where we use that
$\Om\subset \big(
(\Om\setminus  \Lambda_\eps)\cup \Lambda([a-\rho,b+\rho]\times [-\eps,\eps])\big)$
for $\eps>0$ small enough. The upper bound then follows by 
the arbitrariness of $\rho$.
\end{proof}

With straightforward modifications of the previous arguments one can show the following:

\begin{Corollary}\label{cor_multicurves}
Let $\Omega$ have $C^1$-boundary, 
let $n \in \mathbb N$ and 
$\alpha_i:[a_i,b_i]\rightarrow \overline\Om$, $i=1,\dots,n$, 
be curves satisfying 
either (H1)-(H2), 
or (H4).
Assume that $\jumpset_i:=\alpha_i([a_i,b_i])\subset\overline\Om$ are mutually disjoint, and 
let $u\in  W^{1,\infty}(\Omega\setminus 
\jumpset;\R^2)$ satisfy (H3), where $\jumpset:=\cup_{i=1}^n \jumpset_i$. Then 
 $$
 \overline{\mathcal{A}}_{BV}(u, \Omega)=\int_{\Omega}|\mathcal{M}(\nabla u)|~dx+\sum_{i=1}^n\int_{[a_i,b_i]\times
 	\intervallounitario
 }|\partial_t
\mappaaffine_{(i)}
\wedge\partial_s\mappaaffine_{(i)}~dtds.
 $$
 where $\mappaaffine_{(i)}:[a_i,b_i]\times \intervallounitario\rightarrow\R^3$ is the map
 $
 \mappaaffine_{(i)}(t,s)=(t,su^+(\alpha_i(t))+(1-s)u^-(\alpha_i(t))).
 $
\end{Corollary}

\section{Piecewise constant maps}\label{sec:piecewise_constant_maps}
In this section 
 we study the relaxed area \eqref{area_functional_rel2}
and the relaxed 
total variation \eqref{TVJ_functional_rel},
on certain piecewise constant maps.
 We start by exhibiting a $BV$ map
taking three values 
having infinite relaxed total variation of the Jacobian (and hence infinite $BV$-relaxed area), but finite
$L^1$-relaxed area.

\begin{figure}[t]
    \centering
    \includegraphics[scale=0.45]{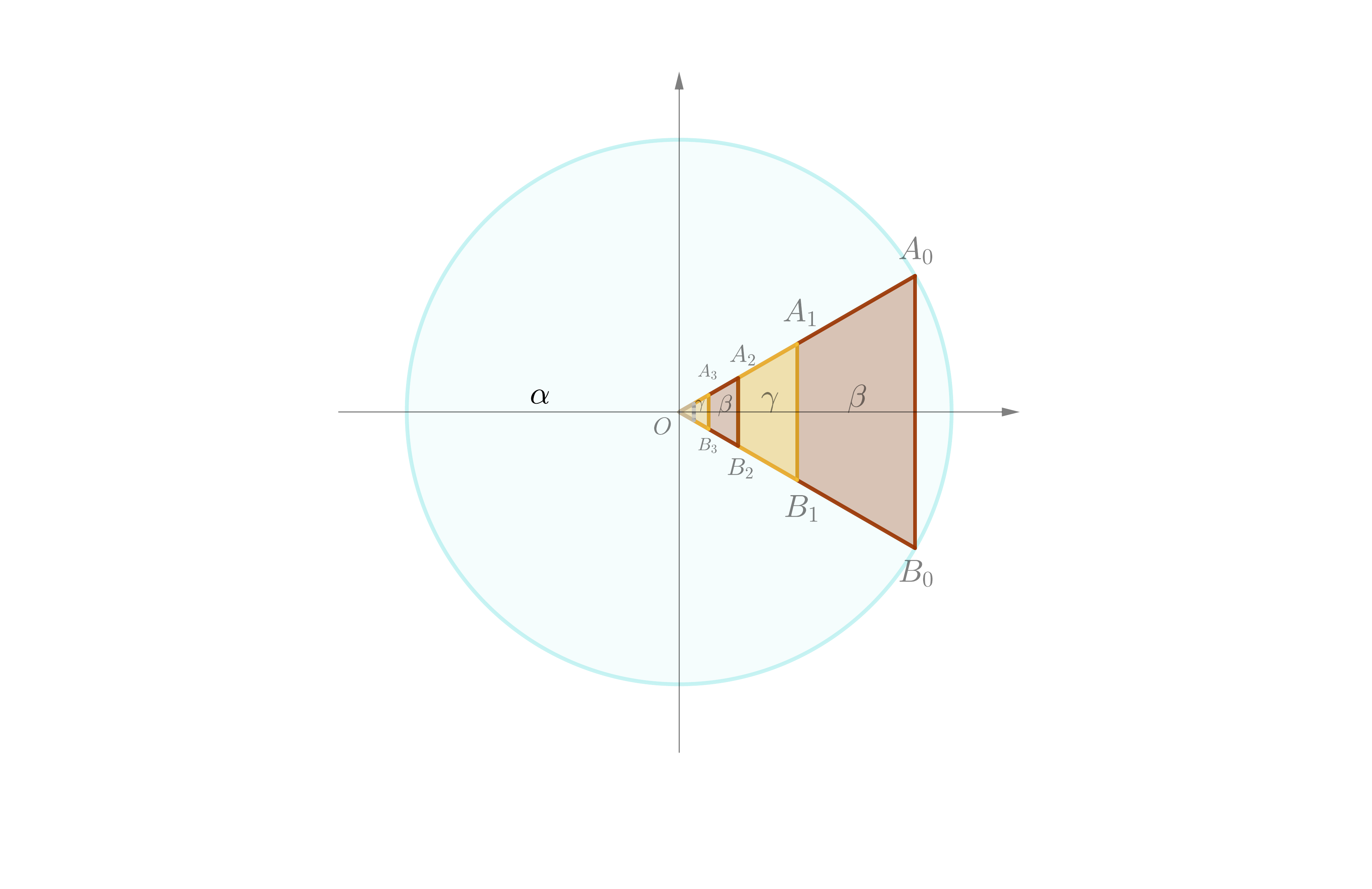}
    \caption{The source disc $B_1(0)$ and the values
$\{\alpha,\beta,\gamma\}$ of $u$, 
with infinitely many triple points.}
    \label{infinite triple}
\end{figure}

\begin{Example}({\bf $BV$-relaxed area and $L^1$-relaxed area})\label{tripunto infinito}
Let $\alpha,\beta,\gamma\in\R^2$ be three non-collinear vectors. 
Consider the map $u:B_1(0)
\subset \R^2\rightarrow\{\alpha,\beta,\gamma\}$ in Fig. \ref{infinite triple}, obtained by the following procedure: divide the 
source equilateral triangle $T_{A_0OB_0}$ in two regions with a vertical 
segment connecting $A_1$ and $B_1$, 
the middle points of the oblique sides of the triangle; assign the value $\beta$ and $\gamma$ on 
the right and on the left as in the figure,
 and repeat this construction on the 
equilateral triangle $T_{A_1OB_1}$, and then repeat the argument 
iteratively
on all smaller triangles; finally set $u=\alpha$ in $B_1(0)
\setminus  T_{A_0OB_0}$. 
In this way we get an infinite collection of triple points 
located at $\{A_i,B_i\}_{i\geq1}$. Then, $u\in BV(B_1(0);\{\alpha, \beta, \gamma\})$, 
since 
\begin{align*}
|Du|(B_1(0))
&=\left(1+2 (1- \sum_{i=1}^{+\infty} 2^{-2i})\right)
|\beta-\alpha|
+
2\sum_{i=1}^{+\infty}2^{-2i}|\alpha-\gamma|
+
\sum_{i=1}^{+\infty} 
2^{-i}|\beta-\gamma|
\\
&=\frac{7}{3}|\beta-\alpha|+\frac{2}{3}|\alpha-\gamma|+|\beta-\gamma|.
\end{align*}
On the other hand, consider an 
infinitesimal sequence $(r_i)_{i \geq 1}$ 
of radii 
with $0<r_i<2^{-(i+1)}$. With an argument similar to 
\cite[Theorem 1.3]{BCS}, we have
$$
\overline{TVJ}_{BV}(u, B_{r_i}(A_i))=
|T_{\alpha\beta\gamma}|,
$$
$|T_{\alpha\beta\gamma}|$ denoting the Lebesgue measure of the target triangle
with vertices $\alpha, \beta, \gamma$,
and thus,
for every $N\in\N$, 
$$
\overline{TVJ}_{BV}(u, B_1(0))\geq\overline{TVJ}_{BV}(u,
\cup_{i=1}^N B_{r_i}(A_i))\geq\sum_{i=1}^N|T_{\alpha\beta\gamma}|=N|T_{\alpha\beta\gamma}|.
$$
 Whence 
\begin{align}\label{TV infinity}
\overline{\mathcal{A}}_{BV}(u, B_1(0))\geq\overline{TVJ}_{BV}(u, B_1(0))=+\infty.
\end{align}
On the other hand, we claim that
\begin{equation}\label{eq:L1_ril_fin}
\overline{\mathcal{A}}_{L^1}(u, B_1(0))<+\infty.
\end{equation}
 Indeed, we can 
construct a sequence $(v_\eps)$ of 
piecewise constant maps on $B_1(0)$, taking values in $\{\alpha,\beta,\gamma\}$, with uniformly bounded $L^1$-relaxed area 
and converging to $u$ in $L^1(B_1(0); \R^2)$:
 Let $\eps\in(0,1)$ and consider the intersection with $T_{A_0 0 B_0}$ of a tubular neighbourhood 
of the segment $\overline{A_iB_i}$ of diameter $\eps{2^{-(i+1)}}$, for every $i\in\N$. Then, the map $v_\eps$ is obtained by modifying $u$ on these strips 
in the triangle, by assigning the value $\alpha$. 
Now, $v_\eps$ is a piecewise constant map valued in $\{\alpha,\beta,\gamma\}$ without triple points, hence, by \cite[Theorem 3.14]{AD}, 
\begin{align*}
& \overline{\mathcal{A}}_{L^1}(v_\eps, B_1(0))
=|B_1(0)|+|Dv_\eps|(B_1(0)) 
\\
\leq&
\pi+\frac{7}{3}|\beta-\alpha|+\frac{2}{3}|\alpha-\gamma|
+\left(1+\frac{\eps}{2}\right)\sum_{i=1}^{+\infty}2^{-i}(|\beta-\alpha|+|\alpha-\gamma|)\\
\leq&\pi+\frac{23}{6}|\beta-\alpha|+\frac{13}{6}|\alpha-\gamma|.
\end{align*}
 Clearly, $v_\eps\rightarrow u$ in $L^1(B_1(0);\R^2)$ as $\eps\rightarrow 0^+$, so by lower semicontinuity 
$$
\overline{\mathcal{A}}_{L^1}(u, B_1(0))\leq\pi+\frac{23}{6}|\beta-\alpha|+\frac{13}{6}|\alpha-\gamma|<+\infty.
$$
In particular
$$
\dom\Big(\overline{\mathcal{A}}^{BV}(\cdot, B_1(0))\Big)\subsetneq
\dom\Big(\overline{\mathcal{A}}^{L^1}(\cdot, B_1(0))\Big).
$$
\end{Example}

\begin{Remark}
Following the notation of \cite{Mc2}, one can show \eqref{TV infinity} 
also by considering the measure $\mu^J_w$ defined for every $w\in BV(B_1(0);\R^2)$ as
$$
\langle\mu^J_w,g\rangle=\frac{1}{2}\int_{\jump_w}({w^1}^-{w^2}^+-{w^1}^+{w^2}^-)\partial_\tau g d\mathcal{H}^1\quad\forall g\in C^\infty_c(B_1(0)),
$$
where $\tau=\nu^\perp$ and $\nu$ is the unit normal to 
$\jump_w$, so that $Dw\,\mres \jump_w=(w^+-w^-)\otimes\nu\mathcal{H}^1\,\mres \jump_w$.

If $\overline{\mathcal A}_{BV}(w,B_1(0))$ is finite, we can consider the unique cartesian current $T_w\in$ cart$(B_1(0);\R^2)$ associated to $w$ defined in \cite[Theorem 3.5]{Mc2}, whose vertical part is by definition equal to the minimal completely vertical lifting $\mu_v[w]$ associated to $w$, according to \cite[Definition 3.1]{Mc2}. Then, since $|\mu_v[w]|$ is lower semicontinuous with respect to the weak convergence of measures and $\mu_v[v]=TVJ(v, \cdot)$ for $v$ smooth (see \cite[(3.6)]{Mc2}), we get
$$
|\mu_v[w]|(B_1(0)\times\R^2)\leq\overline{TVJ}_{BV}(w, B_1(0)).
$$
 In particular, if $w\in BV(B_1(0);\R^2)$ is piecewise constant, we have
\begin{equation}\label{eq:TVJbar_larger}
|\mu^J_w|(B_1(0))\leq|\mu_v[w]|(B_1(0)\times\R^2)\leq
\overline{TVJ}_{BV}(w, B_1(0)),
\end{equation}
where the first inequality is a consequence of \cite[Corollary 4.3]{Mc2}.
 
 Now, if
by contradiction $\overline{\mathcal A}_{BV}(u, B_1(0))$ 
is finite for  the map $u$ 
in Example \ref{tripunto infinito} we have 
$$
\mu^J_u=\sum_{i=1}^{+\infty}
|T_{\alpha\beta\gamma}|(\delta_{A_i}-\delta_{B_i}).$$
In particular $|\mu^J_u|(B_1(0))=+\infty$, and \eqref{TV infinity} follows
from \eqref{eq:TVJbar_larger}.
In Example \ref{esempio alla Mucci}, 
we construct 
a piecewise constant map $u\in BV(B_1(0);\R^2)$ taking only five values in $\R^2$ with $\overline{TVJ}_{BV}(u, B_1(0))=+\infty$ and $\mu^J_u=0$. In that case, one can see even that $\mu_v[u]=0$, whence a maximal gap phenomenon occurs between the mass of the current $T_u$ (which is finite and without a vertical contribution) and $\overline{\mathcal{A}}_{BV}(u,B_1(0)) $ (which is infinite as well).
\end{Remark}

\subsection{Piecewise constant homogeneous maps}
\label{subsec:piecewise_constant_homogeneous_maps}
We need some tools that allow us to characterize (and compute in some cases) the relaxed functionals for $n-$uple point maps with $n\geq3$. Thus, for $r>0$, we consider maps $u:B_r:=B_r(0)\rightarrow\R^2$ of the form 
\begin{align}\label{def of u}
	u(x)=\gamma\left(\frac{x}{|x|}\right)\quad\mbox{for a.e. }x\in B_r,
\end{align}
where 
$\gamma:\Suno\rightarrow\{\alpha_1,\ldots,\alpha_n\}$ is piecewise 
constant and takes
the (not necessarily distinct) values $\alpha_1,\ldots,\alpha_n \in \R^2$ 
on the arcs $C_1,\ldots,C_n$ in 
the order (see Fig. \ref{curve_gamma} for $n=5$). So, $u$ is an $n-$uple point map
 with one $n-$uple junction at the origin. 
Now, we can consider the broken line curve $\widetilde\gamma\subset\R^2$ 
(an example of which is in Fig. \ref{curve_gamma}) 
made of the segments connecting $\alpha_1$ to $\alpha_2$, $\alpha_2$ 
to $\alpha_3$ and so on, closing up by connecting $\alpha_n$ to $\alpha_1$. 
The curve $\widetilde\gamma$ can be parametrized as in \eqref{gamma_tilde},
and 
the curves $\widetilde\gamma_i$ 
are constant.
Denoting by $L(\gamma)$ the length of $\widetilde\gamma$, we have
\begin{equation}\label{eq:def_L}
L(\gamma)=\sum_{i=1}^{n}|\alpha_{i+1}-\alpha_i|
=|\dot\gamma|(\Suno)=
\sup\left\{\sum_{i=1}^{m-1}|\gamma(\nu_{i+1})-\gamma(\nu_i)|: m \in \mathbb N, \{\nu_1,\ldots,\nu_m\}\subset\Suno\right\},
\end{equation}
with the convention $\alpha_{n+1}:=\alpha_1$,
\begin{figure}[t]
    \centering
    \includegraphics[scale=0.49]{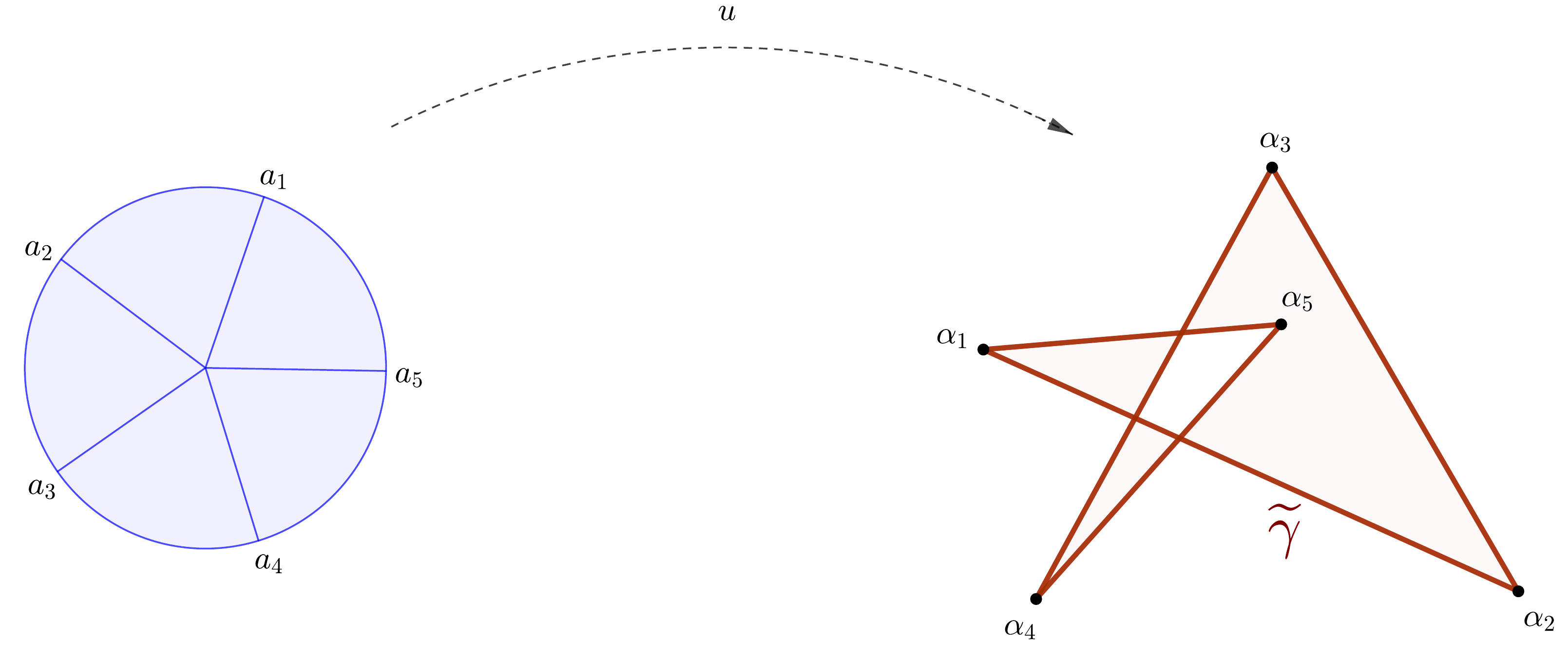}
    \caption{An $n$-uple point map and the corresponding curve $\gamma$, for $n=5$. 
}
    \label{curve_gamma}
\end{figure}
Clearly, by definition of $u$, we have
$$
|Du|(B_r)=r|\dot\gamma|(\Suno)=r L(\gamma).
$$
Thanks to Lemma \ref{lem_plateaurel}, for
 $\rilP(\gamma)$ as in \eqref{Plateau_rel}
we know that 
\begin{align}
	\rilP(\gamma)=P(\widetilde\gamma).
\end{align}
For a general $\gamma$ 
the computation of $\rilP(\gamma)$ seems not immediate.
For the configuration in Fig. \ref{curve_gamma}, we expect it to be the area of the 
region enclosed by $\widetilde\gamma$, with the small 
internal quadrilateral counted twice.

\begin{Theorem}[\textbf{Relaxation of TVJ on piecewise constant maps}]
\label{teo:relaxation_of_TV_on_piecewise_constant_maps}
Let $\{\alpha_1,\ldots,\alpha_n\}\subset\R^2$, 
$\gamma\in BV(\Suno;\{\alpha_1,\ldots,\alpha_n\})$ be a function 
with a finite number of jump points,
and let $u$ be as in \eqref{def of u}.
Then
$$
\overline{TVJ}_{BV}(u, B_r)=\rilP(\gamma).
$$
\end{Theorem}

\begin{proof} 
Lower bound: 
Assume that $(v_k)\subset C^1(B_r;\R^2)$ 
converges to $u$ strictly $BV(B_r;\R^2)$ and
$$
\lim_{k\rightarrow+\infty}\int_{B_r}|Jv_k|~dx=\overline{TVJ}_{BV}(u, B_r).
$$
 By Lemma \ref{lem:inheritance circ}, 
we can fix $\eps\in (0,r)$ and a not-relabeled 
subsequence depending on $\eps$,
such that $v_k\mres\partial B_\eps\rightarrow u\mres\partial B_\eps$ strictly $BV(\partial B_\eps;\R^2)$.
 Thus, using Corollary \ref{cor_continuity_P} and the rescaling invariance of \eqref{Plateau_rel}, we 
can estimate
\begin{align}\label{5.6}
\overline{TVJ}_{BV}(u, B_r)& \geq\liminf_{k\rightarrow +\infty}\int_{B_\eps}|Jv_{k}|~dx \geq\liminf_{k\rightarrow +\infty}P(v_k\mres\partial B_\eps)=\rilP(u\mres\partial B_\eps)=
\rilP(\gamma).
\end{align}

Upper bound: By an 
argument similar to the one at the beginning of the proof of Proposition 
\ref{prop:upper_bound}, it will be enough to construct a recovery 
sequence $(u_k)
\subset$ Lip$(B_r;\R^2)$. 
Let $\widetilde\gamma$ be as above.
We start by building a sequence $(\gamma_k)_k$ of 
Lipschitz reparameterizations of $\widetilde\gamma$ which converges 
strictly $BV(\Suno;\R^2)$ to $\gamma$. 
Let us denote by $a_1,\ldots,a_n\in[0,2\pi)$ the angular coordinates of the extremal 
points of $C_1,\ldots,C_n$, 
and assume without loss of generality $0=a_1<a_2<\dots<a_n$. 
Then 
 $$
\bigcup_{i=1}^n[a_i,a_{i+1}]=[0,2\pi],
$$
with the convention $a_{n+1}=2\pi$. 
Let $(\delta_k)_k$ be an infinitesimal sequence with $0<\delta_k<\max\{|a_{i+1}-a_{i}|,i=1,\ldots,n\}$, for instance $\delta_k=\frac{2}{k}$, $k$ large enough. We define the piecewise affine map $\gamma_k:[0,2\pi]\rightarrow\R^2$ as
\begin{equation}\label{gamma_k}
\gamma_k(t)=
\begin{cases}
\alpha_{i} &\text{if }t \in [a_i+\delta_k/2,a_{i+1}-\delta_{k}/2],
\\
\displaystyle \frac{a_{i+1}+\delta_{k}/2-t}{\delta_{k}}\alpha_{i}+\frac{t-a_{i+1}+\delta_{k}/2}{\delta_{k}}\alpha_{i+1}&\text{if }t\in[a_{i+1}-\delta_{k}/2,a_{i+1}+\delta_k/2],
\end{cases}
\quad i=1,\ldots,n.
\end{equation}
Then $\gamma_k\rightarrow\gamma$ strictly $BV(\Suno;\R^2)$
(a direct computation shows that $|\dot\gamma_k|(\Suno)=|\dot\gamma|(\Suno)$), 
$\gamma_k$ are uniformly bounded in $L^\infty$, and converge 
almost everywhere to $\gamma$.
As a consequence, from Corollary \ref{cor_continuity_P},
\begin{align}\label{conv_P}
P(\gamma_k)\rightarrow \rilP(\gamma)\qquad\text{ as }k\rightarrow +\infty.
\end{align}
Therefore, by \eqref{Plateau} we choose, for all $k>1$ large enough, a map $v_k\in {\rm Lip}(B_{1};\R^2)$ such that 
\begin{align}\label{conv_P2}
v_k\mres\Suno=\gamma_k, 
\qquad 
\qquad\left|P(\gamma_k)-\int_{B_{1}}|Jv_k|~dx\right|\leq \frac1k.
\end{align}
Let $c_k>0$ be the Lipschitz constant of $v_k$. 
Defining $v_{k,\rho}
\in \text{Lip}(B_\rho;\R^2)$ as  $v_{k,_\rho}(y):=
v_k(\frac{y}{\rho})$ for any $\rho>0$, 
it is straightforward that the Lipschitz constant of $v_{k,\rho}$ 
is $c_k/\rho$.

We now choose an infinitesimal sequence $(\rho_k)\subset (0,r)$ 
in such a way that 
$\lim_{k \to +\infty} c_k\rho_k =0$.
As a consequence we get
\begin{align}\label{int_0}
\int_{B_{\rho_k}}|\nabla v_{k,\rho_k}
|~dx\leq \pi c_k\rho_k\rightarrow 0\qquad \text{ as }k\rightarrow +\infty.
\end{align}
We are now in a position to introduce our recovery sequence: 
We define $u_k\in \text{Lip}(B_r;\R^2)$ as
\begin{align}\label{def_rec_npoint}
u_k(x):=\begin{cases}
	\gamma_k\left(\frac{x}{|x|}\right) &\forall x\in B_r\setminus B_{\rho_k},\\
	v_{k,{\rho_k}}(x)&\forall x\in B_{\rho_k}.
\end{cases}
	\end{align}
Using that 
$\gamma_k\rightarrow \gamma$ strictly $BV(\Suno;\R^2)$ 
and \eqref{int_0} we 
see that $u_k\rightarrow u$ strictly $BV(B_r;\R^2)$. 
Finally, since in $B_r\setminus B_{\rho_k}$ the map
$u_k$ depends only on the angular coordinate, 
its Jacobian determinant vanishes in  $B_r\setminus B_{\rho_k}$. Hence 
\begin{align}
\liminf_{k\rightarrow +\infty}\int_{B_r}|Ju_k|~dx=\liminf_{k\rightarrow +\infty}\int_{B_{\rho_k}}|Jv_{k,{\rho_k}}|~dx
= \rilP(\gamma),
\end{align}
the convergence being a consequence of \eqref{eq:int_Jac_radius}, \eqref{conv_P2}, and \eqref{conv_P}. 
\end{proof}

As a consequence of Theorem 
\ref{teo:relaxation_of_TV_on_piecewise_constant_maps}
 we deduce:
\begin{Theorem}[\textbf{Relaxation of $\mathcal A$ on 
piecewise constant maps}]
\label{area n-ple point}
Let $\gamma$ and $u$ as in Theorem \ref{teo:relaxation_of_TV_on_piecewise_constant_maps}. 
Then, for any $r>0$, we have 
\begin{align}\label{eq:area_npoint}
\overline{\mathcal{A}}_{BV}(u, B_r)
=\pi r^2+r L(\gamma) +\rilP(\gamma).
\end{align}
\end{Theorem}
\begin{proof}
Lower bound: Suppose that $v_k\in C^1(B_r;\R^2)$ is such that
 $$v_k\rightarrow u\quad\mbox{strictly } BV(B_r;\R^2)\quad \mbox{and }\quad \lim_{k\rightarrow+\infty}
{\mathcal{A}(v_k, B_r)}=\liminf_{k\rightarrow+\infty}
{\mathcal{A}(v_k, B_r)}.$$
 Now, let $\eps\in(0,r)$  and  write
${\mathcal A}(v_k, B_r) = {\mathcal A}(v_k, B_r \setminus B_{\eps})
+
{\mathcal A}(v_k, B_{\eps}) \geq 
{\mathcal A}(v_k, B_r \setminus B_{\eps})
+ \int_{B_{\eps}} \vert J v_{k}\vert ~dx$,
so that, by \cite[Theorem 3.7]{AD},
\begin{equation}\nonumber
\begin{aligned}
\lim_{k\rightarrow+\infty}
{\mathcal{A}(v_k, B_r)}&\geq\liminf_{k\rightarrow+\infty}\mathcal{A}
(v_k, B_r\setminus B_{\varepsilon})+\liminf_{k\rightarrow+\infty}\int_{B_{\epsilon}}|Jv_{k}|~dx
\\
&\geq|B_r\setminus B_\eps|+r(1-\eps) L(\gamma) +\liminf_{k\rightarrow+\infty}\int_{B_{\epsilon}}|Jv_{k}|~dx\\
&\geq |B_r\setminus B_\eps|+r(1-\eps) L(\gamma) +\rilP(\gamma),
\end{aligned}
\end{equation}
where in the last line we have applied Theorem \ref{teo:relaxation_of_TV_on_piecewise_constant_maps}
 with $r$ replaced by $\eps$.
We now pass to the limit as 
$\eps\rightarrow 0^+$ to get the lower bound
$\overline{\mathcal{A}}_{BV}(u, B_r)
\geq \pi r^2+r L(\gamma) +\rilP(\gamma)$
 in \eqref{eq:area_npoint}.

Upper bound: It is sufficient to consider the sequence $(u_k)_k$ 
defined in \eqref{def_rec_npoint}, for which
\begin{equation}\nonumber
\begin{aligned}
\overline{\mathcal{A}}_{BV}(u, B_r)&\leq\limsup_{k\rightarrow + \infty}\mathcal{A}(u_k, B_1)\leq|B_r|+\lim_{k\rightarrow + \infty}\int_{B_r}|\grad u_k|~dx+\lim_{k\rightarrow
+\infty}\int_{B_r}|Ju_k|~dx\\
&=\pi r^2+ r L(\gamma) +\rilP(\gamma).
\end{aligned}
\end{equation}
\end{proof}

Now, we are in the position to show
 an example of a piecewise constant map $u\in BV(B_1;\R^2)$ with infinite relaxed Jacobian total variation but vanishing associated minimal vertical lifting measure $\mu_v[u]$. This map 
is constructed in Example \ref{esempio alla Mucci}, while the Example \ref{double butterfly} 
is preparatory.  
\begin{Example}\label{double butterfly}
We want to show here how singular topological phenomena related to the double-eight map 
\cite{Ma}, \cite{GMS1}, \cite{Mc1},\cite{Pa},\cite{DP} arise also among piecewise constant maps. In particular, as pointed out in \cite{Mc2}, for the homogeneous extension of the double-eight map, a gap phenomenon occurs between the minimal vertical lifting measure and the relaxed Jacobian total variation.
 We show now that we find such a gap also among piecewise constant maps, by exhibiting a piecewise constant map with vanishing minimal vertical lifting measure but with finite non-zero $\overline{TVJ}$. Namely, we are going to define a map $u:B_1\rightarrow \R^2$ assuming five distinct values, for which the resulting closed curve $\widetilde\gamma$ has 
zero degree, but is homotopically trivial, since it is, in fact, homeomorphic to the double-eight curve. Let $\{\alpha_1,\alpha_2,\alpha_3,\alpha_4,\alpha_5\}\subset\R^2$ be the vertices of two equilateral triangles with a common vertex, 
say $\alpha_1$ (see Figure \ref{butterfly}).  Fix 
a partition of $\Suno$ in twelve disjoint non-empty arcs $C_1,\ldots,C_{12}$ (not necessarily of the same length), with extremal points $a_1,\ldots,a_{12}$ in counter-clockwise order. Then, define $\gamma:\Suno
\rightarrow\{\alpha_1,\alpha_2,\alpha_3,\alpha_4,\alpha_5\}$ to be constant on the arcs $C_1,\ldots,C_{12}$, precisely equal to, in the order,  $\alpha_1,\alpha_2,\alpha_3,\alpha_1,\alpha_4,\alpha_5,\alpha_1,\alpha_3,\alpha_2,\alpha_1,\alpha_5,\alpha_4$. Then, 
the broken line curve $\widetilde\gamma$ runs consecutively the triangles $T_{123}:=T_{\alpha_1\alpha_2\alpha_3}$ and $T_{145}:=T_{\alpha_1\alpha_4\alpha_5}$ twice, and every time with different orientation.
Define $u$ as in \eqref{def of u}, obtaining a $12$-point map. Now, by
applying Theorem \ref{teo:relaxation_of_TV_on_piecewise_constant_maps} and computing the minimum of the Plateau problem \eqref{Plateau} for $\widetilde\gamma$ as in \cite[Theorem 5]{Pa}, 
we obtain

\begin{align}\label{butterfly TV}
\overline{TVJ}_{BV}(u, B_1)=\rilP(\gamma)=P(\widetilde\gamma)=2\min\{|T_{123}|,|T_{145}|\}.
\end{align}
 Moreover, it is not difficult to see that
$$
\mu^J_u=(|T_{123}|+|T_{145}|-|T_{123}|-|T_{145}|)\delta_{0}=0.
$$
 In this case, we 
have also $\mu_v[u]=0$, indeed we 
can prove that the unique current $T_u$ with minimal completely vertical lifting associated to $u$ is given by
\begin{align}\label{Tu}
T_u=G_u+S=\sum_{l=1}^{12}[\![\widehat C_l]\!]\times[\![c_l]\!]+\sum_{l=1}^{12}[\![0,a_{l}]\!]\times[\![c_{l-1},c_{l}]\!],
\end{align}
where $\widehat C_l$ is the circular sector corresponding to $C_l$ and $c_l$ is the assigned value of $\gamma$ on $C_l$ for $l=1,\ldots,12$  (we used the convention $c_{0}=c_{12}$).
Let us show \eqref{Tu}. One checks 
that $\mu^j_i[T_u]=\mu_i^j[u]$ for $ i,j=1,2$ by proceeding as in Remark \ref{rem:link_Mucci}. So, it remains to prove that $T_u\in\mathrm{cart}(B_1;\R^2)$: it is enough to 
check that $(\partial T_u) \mres B_1\times\R^2=0$. Compute
$$
\partial S=\sum_{l=1}^{12}\partial\left([\![0,a_{l}]\!]\times[\![c_{l-1},c_{l}]\!]\right)=\sum_{l=1}^{12}\left(-[\![0]\!]\times[\![c_{l-1},c_{l}]\!]+[\![0,a_{l}]\!]\times[\![c_{l}]\!]-[\![0,a_{l}]\!]\times[\![c_{l-1}]\!]\right).
$$
Now, since by convention $a_{13}=a_1$,
$$
\partial G_u=\sum_{l=1}^{12}\left([\![0,a_{l+1}]\!]\times[\![c_{l}]\!]-[\![0,a_{l}]\!]\times[\![c_{l}]\!]\right)=-\sum_{l=1}^{12}\left([\![0,a_{l}]\!]\times[\![c_{l}]\!]-[\![0,a_{l}]\!]\times[\![c_{l-1}]\!]\right).
$$
Moreover, by the choice of $\{c_l\}$,
$$
\sum_{l=1}^{12}[\![0]\!]\times[\![c_{l-1},c_{l}]\!]=[\![0]\!]\times[\![\alpha_1,\alpha_2]\!]+[\![0]\!]\times[\![\alpha_2,\alpha_3]\!]+\ldots+[\![0]\!]\times[\![\alpha_4,\alpha_1]\!]=0.
$$
Therefore, $\partial G_u=-\partial S$.\\
Notice that the action of $T_u$ against 2-forms with only vertical differentials is 0, which means that $T_u$ does not have completely vertical part and so $\mu_v[u]=0$. Roughly, 
due to cancellations in the part of the boundary of $T_u$ in corrispondence to the origin, the current $T_u$ is not able to detect the hole upon the origin in the graph of $u$, generated by the presence of the 
multiple junction.
\end{Example}
\begin{figure}[t]
    \centering
    \includegraphics[scale=0.44]{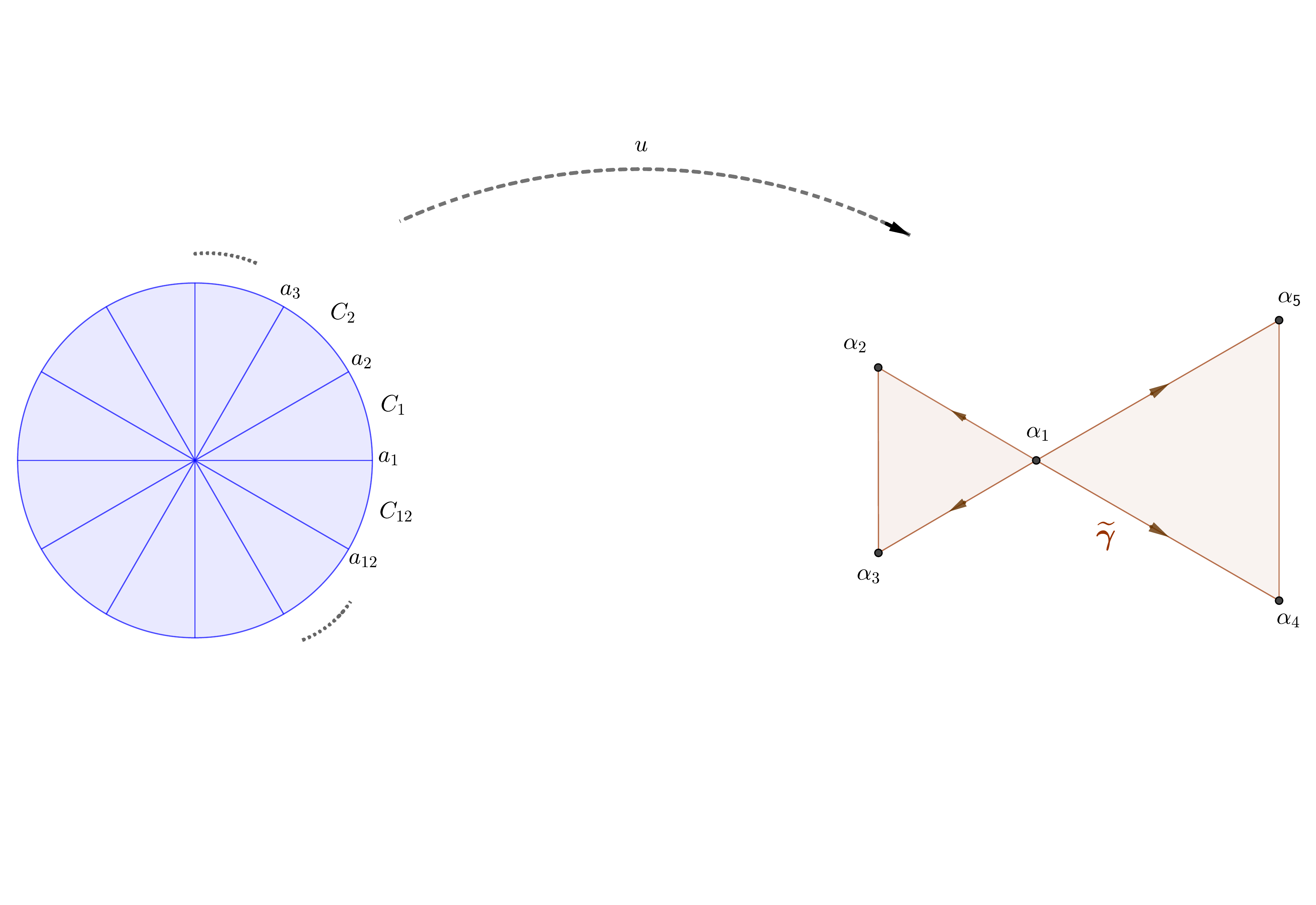}
    \caption{The map $u$ and the broken line curve $\widetilde\gamma$ of Example \ref{double butterfly}. 
}
    \label{butterfly}
\end{figure}

\begin{figure}[t]
    \centering
    \includegraphics[scale=0.49]{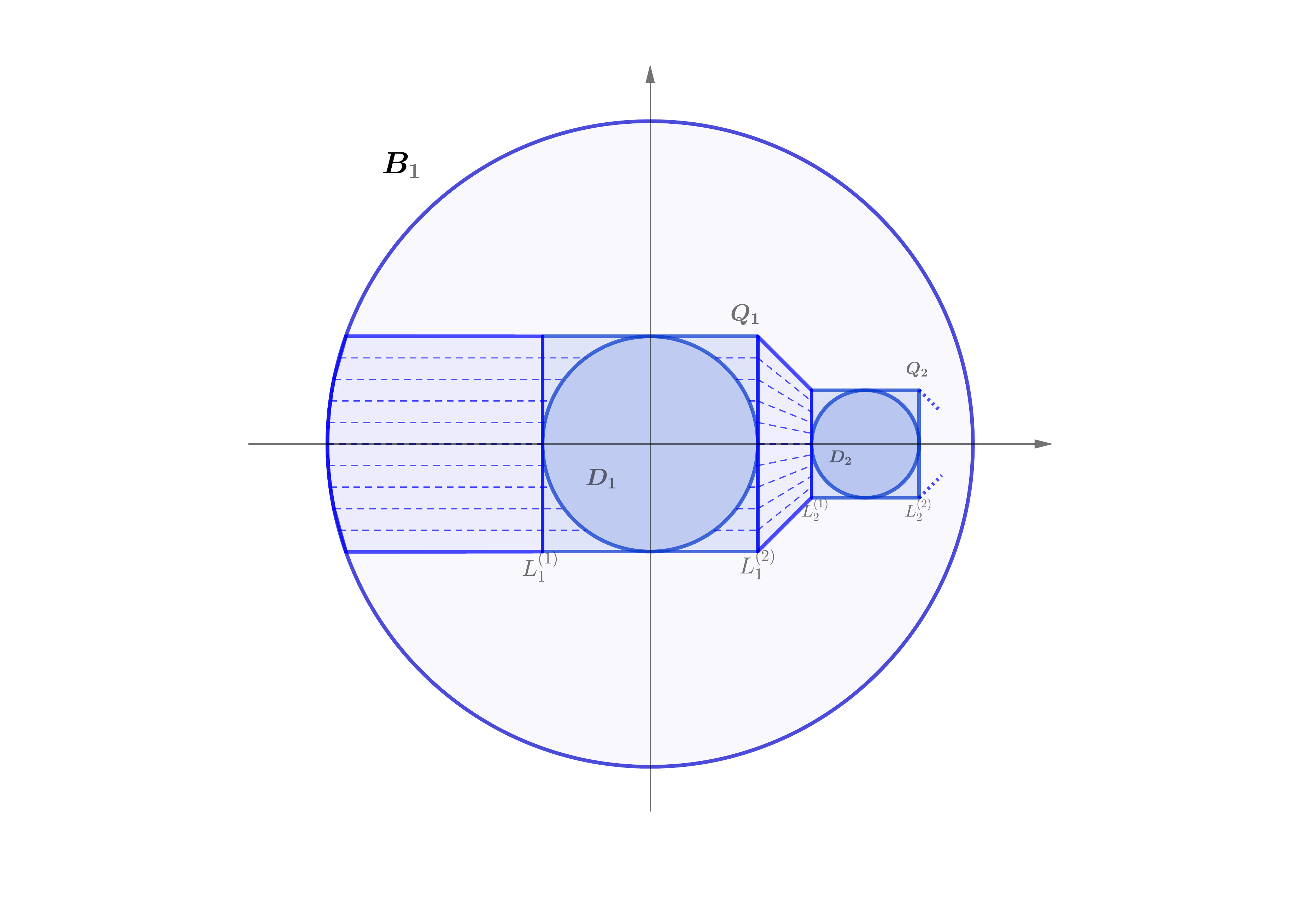}
    \caption{The sequence $\{D_i\}\subset B_1$ of disks 
of Example \ref{esempio alla Mucci}. }
    \label{fig:mucci_domain}
\end{figure}
\begin{Example}\label{esempio alla Mucci}
This example is an adaptation of 
\cite[Theorem 1.3]{Mc1} to the case of piecewise constant maps. Indeed, we construct a piecewise constant map $u$, 
taking only five values of $\R^2$, such that 
$$\mu_v[u]=0\qquad\mbox{ and }\qquad\overline{TVJ}_{BV}(u, B_1)=+\infty.$$ The idea is to replicate the map of Example \ref{double butterfly} infinitely many times on a sequence $\{D_i\}_{i\in\N}\subset B_1$
of disjoint 
balls, 
whose measures form  an infinitesimal sequence (see Figure \ref{fig:mucci_domain}). So, for $i\in\N$, set $$ D_i:=B_{r_i}(x_i), \quad \mbox{with }x_i:=\left(-1+\sum_{j=0}^{i-1}2^{-j},0\right),\quad r_i:=2^{-i-1}.$$
Let $\{\alpha_1,\alpha_2,\alpha_3,\alpha_4,\alpha_5\}\subset\R^2$ and $\gamma:\Suno\rightarrow\{\alpha_1,\alpha_2,\alpha_3,\alpha_4,\alpha_5\}$ 
be as in Example \ref{double butterfly}. 
Now, define the map $\widehat \gamma:\Suno\rightarrow\{\alpha_1,\alpha_2,\alpha_3,\alpha_4,\alpha_5\}$ 
in the 
same way as $\gamma$, 
but with different order of the values, in a symmetric way with respect to the vertical axis through $\alpha_1$, namely, in the same arcs $C_1,\ldots,C_{12}$, $\widehat \gamma$ is equal to $\alpha_1,\alpha_5,\alpha_4,\alpha_1,\alpha_3,\alpha_2,\alpha_1,\alpha_4,\alpha_5,\alpha_1,\alpha_2,\alpha_3$. Then, for $i\in\N$, define $u_{|D_i}:=u^{(i)}$ as
 \begin{equation*}
u^{(i)}(x)=\left\{
\begin{aligned}
&\gamma\left(\frac{x-x_i}{|x-x_i|}\right)\quad\mbox{if $i$ is odd},\\
&\widehat\gamma\left(\frac{x-x_i}{|x-x_i|}\right)\quad\mbox{if $i$ is even}.
\end{aligned}
\right.
\end{equation*}
It remains to define $u$ in $B_1\setminus\cup_{i\in\N}D_i$.  Start by considering, for every $i\in\N$, the square $Q_i$ that circumscribes $D_i$ and extend $u^{(i)}$ to $Q_i$ to be constant along horizontal lines. Now, denote by $L_i^{(1)}$ and $L_i^{(2)}$ the vertical left and right sides of $\partial Q_i$, then extend $u$ to the 
convex hull of $L_i^{(2)}$ and $L_{i+1}^{(1)}$ to be constant along straight lines which interpolate pointwise the two sides. Finally, extend $u$ in the strip that connects $L_1^{(1)}$ to $\partial B_1$ to be constant along horizontal lines and set $u=\alpha_1$ in the rest of $B_1$. 
(see Figure \ref{fig:mucci_domain}).
It is 
not difficult to see that 
$u \in BV(B_1;\R^2)$, by the choice of the infinitesimal sequence $(r_i)$.
Thus, assuming by contradiction that $\overline{\mathcal A}_{BV}(u, B_1)$ be finite, one can define the current $T_u=G_u+S$ in a 
similar way as in  Example \ref{double butterfly}, that is to say, by setting $S$ to be the trivial affine interpolation surface on the 
jump segments of $u$. One can prove in the same way that $T_u$ is the current with minimal completely vertical lifting associated to $u$ and $\mu_v[u]=0$. In particular, $T_u\in$ cart$(B_1;\R^2)$ and has finite mass.
On the other hand,
$$
\overline{TVJ}_{BV}(u, B_1)\geq\sum_{i=1}^{+\infty}\overline{TVJ}_{BV}(u, D_i)=\sum_{i=1}^{+\infty}2\min\{|T_{\alpha_1\alpha_2\alpha_3}|,|T_{\alpha_1\alpha_4\alpha_5}|\}=+\infty.
$$
In particular $\overline{\mathcal A}_{BV}(u, B_1)=+\infty$ as well.
\end{Example}

\section{Piecewise Lipschitz maps}\label{sec:Lip_npoint}
In this last section we combine the results 
of the previous sections and compute the $BV$-relaxed area for an interesting class of maps 
that we call \textit{piecewise Lipschitz maps}, quickly mentioned in the Introduction. 
As stated in our main result
(Theorem \ref{teo:Lip_npoint}), 
 the relaxed area turns out to be composed by a regular term and a singular one, that 
interestingly 
further splits into two non-trivial pieces, respectively related to the $1$-dimensional and $0$-dimensional singularities. \\
Let $\Omega \subset \R^2$ be a connected bounded open set 
with boundary of class $C^1$.
We say that a collection $\{\Om_1,\ldots,\Om_N\}$ 
of disjoint nonempty open sets
is a Lipschitz partition of $\Om$ if $\overline \Om=\cup_{k=1}^N\overline\Om_k$ and 
for each $k=1,\ldots,N$, $\Om_k$ is connected and Lipschitz.

For a given Lipschitz partition of $\Om$
we can consider its interface $\jumpset:=
\cup_{k=1}^N\partial\Om_k$. Also, we can define the (possibly empty) set of interior junction points $\{p_i\}_{i=1}^m$, i.e. points $p_i\in\Om$ such that 
there exist $r>0$ and an integer $N_i$ with $3\leq N_i\leq N$, 
such that $B_r(p_i)\subset\Om$ and $B_s(p_i)$ has nonempty intersection 
with exactly $N_i$  connected components of $\Om$, for every $s \in (0,r]$.

We shall consider Lipschitz partitions 
whose interface is a \textit{network} in the following sense:

\begin{Definition}[\textbf{Network}]
\label{def:simple_network}
The interface $\jumpset$ 
of a Lipschitz partition of $\Om$
is a network if
\begin{equation}\label{J}
\jumpset:=\bigcup_{\indiceaffine=1}^n \overline J_\indiceaffine,
\qquad
J_\indiceaffine = \alpha_\indiceaffine(I_\indiceaffine),
\ \ \intervalloparametri_\indiceaffine =(a_\indiceaffine, b_\indiceaffine),
\end{equation} where the curves $\alpha_\indiceaffine
:\overline \intervalloparametri_\indiceaffine
:=[a_\indiceaffine,b_\indiceaffine]\rightarrow\overline\Om$, 
$\indiceaffine=1,\dots,n$, 
 satisfy the following properties:
\begin{itemize}
\item[-]
$\alpha_\indiceaffine$ is of class $C^2$,
injective with $|\dot\alpha_\indiceaffine|\equiv1$ 
on $\intervalloparametri_\indiceaffine$, 
and 
 $J_\indiceaffine\subset\Om$;
\item[-] 
$\indiceaffine_1\neq \indiceaffine_2 \Rightarrow
J_{\indiceaffine_1}
\cap J_{\indiceaffine_2}
=\varnothing$; 
\item[-]
$\alpha_\indiceaffine
(\{a_\indiceaffine,b_\indiceaffine\})\subset\{p_1,\dots,p_m\}
\cup\partial\Om$ 
for all $\indiceaffine=1,\dots,n$ such 
that $\alpha_\indiceaffine(a_\indiceaffine)\neq\alpha_\indiceaffine(b_\indiceaffine)$;
\item[-]
if $x\in\overline J_\indiceaffine\cap\partial\Om$, $\alpha_\indiceaffine$ is transversal to $\partial\Om$ at $x$;
\item[-]
$\indiceaffine_1\neq \indiceaffine_2 \Rightarrow
\overline J_{\indiceaffine_1}
\cap\overline  J_{\indiceaffine_2} \subset\{p_1,\dots,p_m\}$.
\end{itemize}
\end{Definition}
From the last condition it follows
that 
if two curves have endpoints on 
$\partial\Om$, then these points are distinct.

\begin{Definition}[\textbf{Piecewise Lipschitz map}]\label{def_MLM}
Let $\{\Om_k\}_{k=1}^N$ be a Lipschitz partition of $\Om$ whose 
interface $\jumpset$ is a network. We say that
$u\in BV(\Om;\R^2$) is a \textit{piecewise Lipschitz  
map} if its jump set $\jump_u$
coincides with $\jumpset$
and $u\mres \concom_k\in\mathrm{Lip}(\concom_k;\R^2)$ for any 
$k = 1,\dots, N$.
\end{Definition}

Since $u\mres \concom_k\in\mathrm{Lip}(\concom_k;\R^2)$,  the trace of $u$ 
on $\partial \concom_k$ is also Lipschitz. 
In particular, for any $i\in\{1,\dots,m\}$  such that $p_i\in \partial \concom_k$, 
$$
\exists 
\lim_{\substack{x\rightarrow p_i\\x\in \concom_k}}u(x)=:\beta_{i}^k \in \R^2.
$$
Let $\rho>0$ be sufficiently small so that
$B_\rho(p_i)\subset\Om$
for $i\in\{1,\dots,m\}$. 
Let $\indiceaffine
\in\{1,\dots,n\}$ be such that $p_i$ is an endpoint of 
$\overline J_\indiceaffine$;
since  $\alpha_\indiceaffine$ is of class $C^2$,  for $\rho$ small enough the 
intersection $\overline J_\indiceaffine
\cap \partial B_\rho(p_i)$ consists either of a single point,
or of two points if $\alpha_\indiceaffine(a_\indiceaffine)
=\alpha_\indiceaffine(b_\indiceaffine)=p_i$. 
Hence, the map $u\mres\partial B_\rho(p_i)$ is piecewise Lipschitz 
and jumps at any point of $\jumpset \cap \partial B_\rho(p_i)$. In particular, the number of these jump points is, by definition of junction point,
$$
N_i=\sharp\big(\jumpset\cap \partial B_\rho(p_i)\big) \geq 3,
\qquad i=1,\dots,m.
$$
For $i=1,\dots,m$, we denote by 
$\concom^i_{1},\dots, \concom^i_{{N_i}}$ the connected components of 
$\Om\setminus \jumpset$ 
whose closure contains $p_i$, chosen in counterclockwise order around $p_i$.
Since $\Omega_k$ is Lipschitz for every $k=1,\ldots,N$,
any $\concom^i_k$ has a corner at $p_i$ 
whose aperture is a positive angle $\theta^k_i\in (0,2\pi)$. 

\begin{Lemma}[\textbf{Circular slices}]\label{lemma_palline}
Let $i\in\{1,\dots,m\}$ be fixed and let $\rho>0$ be as above. Then the maps $\gamma^i_\rho\in BV(\Suno;\R^2)$ defined by
$\gamma^i_\rho(\nu):=u(p_i+\rho \nu)$  
converge strictly $BV(\Suno;\R^2)$, as $\rho\rightarrow 0^+$, 
to a piecewise constant map $\gamma^i: \Suno \to \R^2$
taking, in counterclockwise order, the values $\beta_i^1,\beta_i^2,\dots,\beta_i^{N_i}$ on arcs of size $\theta_i^1,\theta_i^2,\dots,\theta_i^{N_i}$, respectively.
\end{Lemma}
The map $\gamma^i$ 
has $N_i$ jumps on $\Suno$ whose angular coordinates are denoted by  $a_i^1,a_i^2,\dots,a_i^{N_i}$  (where\footnote{With the convention $N_i+1=1$.} $a_i^j-a_i^{j-1}=\theta_i^j$, for $j=1,\dots,N_i+1$).

\begin{proof}
It is easy to see that $(\gamma^i_\rho)$ 
converges to $\gamma^i$ almost everywhere on $\Suno$ as $\rho 
\to 0^+$. Moreover, $\gamma^i_\rho$, for $\rho$ small enough, has exactly $N_i$ jumps at points $a_{i,\rho}^j$ of amplitude $|u^+(p_i+\rho a_{i,\rho}^j)-u^-(p_i+\rho a_{i,\rho}^j)|$ 
which tend, by 
continuity of $u$ in $B_\rho(p_i)\setminus \jumpset$, to $|\beta_i^j-\beta_i^{j+1}|$. Also, on the arcs between $a^j_{i,\rho}$ and $a^{j+1}_{i,\rho}$,
 $|\dot\gamma^i_\rho|\leq L\rho$, where $L$ 
is the maximum of the Lipschitz constants of $u$ on the sectors $\concom^i_k$. Hence $|\dot\gamma_\rho^i|(\Suno)\rightarrow|\dot\gamma^i|(\Suno)$ and the thesis follows straightforwardly.
\end{proof}
For $\indiceaffine=1,\ldots,n$, we denote by $u_{(\indiceaffine)}^\pm$ 
the two traces  of $u$ on $J_\indiceaffine$, 
and consider the affine interpolation surface $\affinesurfaceindice:
[a_\indiceaffine,b_\indiceaffine]\times \intervallounitario\rightarrow\R^3$ 
spanning the graphs of $u_{(\indiceaffine)}^-$ and $u_{(\indiceaffine^+)}$, given by
\eqref{eq:affine_surface_l}:
\begin{equation}\label{eq:affine_interpolation_surface_i}
\affinesurfaceindice
(t,s)=(t,su_{(\indiceaffine)}^+(t)+(1-s)u_{(\indiceaffine)}^-(t)),\qquad (t,s)\in[a_\indiceaffine,
b_\indiceaffine]\times \intervallounitario.
\end{equation}

We are now ready to prove 
Theorem \ref{teo:Lip_npoint}.
\begin{proof}[Proof of Theorem \ref{teo:Lip_npoint}]
Lower bound: 
Consider a sequence  $(v_k)\subset C^1(\Om;\R^2)$ converging to $u$ strictly $BV(\Om;\R^2)$. 
For any $\rho>0$ small enough, we take a family of 
mutually disjoint balls $B_\rho(p_i) \subset \Om$, $i=1,\dots,m$. 
By Lemma \ref{lem:inheritance circ}, there exists a subsequence $(v_{k_h})\subset(v_k)$ depending on $\rho$ such that ${\rm for}~ i=1,\ldots,m$
\begin{align}\label{converg on circ}
v_{k_h}\mres\partial B_\rho(p_i)\rightarrow u\mres\partial B_\rho(p_i)\quad\mbox{strictly }BV(\partial B_\rho(p_i);\R^2).
\end{align}
We may also assume that ${\rm for}~ i=1,\ldots,m$
 $$
\liminf_{k\rightarrow+\infty}\int_{B_{\rho}(p_i)} \vert J v_k\vert ~dx=\lim_{h\rightarrow+\infty}\int_{B_{\rho}(p_i)} \vert J v_{k_h}\vert ~dx.
$$
Then 
$$
{\mathcal A}(v_{k_h}, \Om) = 
{\mathcal A}(v_{k_h}, \Om \setminus \cup_{i=1}^m 
\overline B_\rho(p_i))
+\sum_{i=1}^m
{\mathcal A}(v_{k_h};\overline B_{\rho}(p_i)) \geq 
{\mathcal A}(v_{k_h}, \Om \setminus \cup_{i=1}^m \overline 
B_\rho(p_i))
+ \sum_{i=1}^m\int_{B_{\rho}(p_i)} \vert J v_{k_h}\vert dx.
$$
By Corollary \ref{cor_multicurves}, we get
\begin{align*}
\liminf_{h\rightarrow+\infty}{\mathcal A}(v_{k_h}, \Om \setminus \cup_{i=1}^m 
\overline B_\rho(p_i))
\geq&
\overline{\mathcal{A}}_{BV}(u, \Om \setminus \cup_{i=1}^m \overline B_\rho(p_i))\\
=&\int_{\Om \setminus \cup_{i=1}^m B_\rho(p_i)}|\mathcal{M}(\nabla u)|dx+
\sum_{\indiceaffine=1}^n\int_{[a^\rho_\indiceaffine,b^\rho_\indiceaffine]
\times \intervallounitario}|\partial_t \affinesurfaceindice
\wedge\partial_s \affinesurfaceindice|~dtds
\\
&{\longrightarrow}\int_{\Om}|\mathcal{M}(\nabla u)|dx+\sum_{\indiceaffine=1}^n\int_{[a_\indiceaffine,b_\indiceaffine]\times \intervallounitario}|\partial_t \affinesurfaceindice
\wedge\partial_s \affinesurfaceindice|dtds\quad\mbox{as }\rho\rightarrow 0^+,
\end {align*}
where 
$(a^\rho_\indiceaffine),(b^\rho_\indiceaffine)\subset[a_\indiceaffine,b_\indiceaffine]$ are 
respectively a decreasing and increasing sequence of numbers
satisfying $a^\rho_\indiceaffine\rightarrow a_\indiceaffine$ and 
$b^\rho_\indiceaffine\rightarrow b_\indiceaffine$ as $\rho\rightarrow 0^+$ and $\alpha_\indiceaffine([a^\rho_\indiceaffine,b^\rho_\indiceaffine])=\alpha_\indiceaffine([a_\indiceaffine,b_\indiceaffine])\setminus\cup_{i=1}^m B_\rho(p_i)$.

Let us recall that, by Lemma \ref{lem_plateaurel}, $P(\tilde\gamma^i)=\rilP(\gamma^i)$, 
with $\gamma^i$ as in Lemma \ref{lemma_palline}. So, it 
remains to show that
\begin{equation}\label{eq:liminf_liminf}
\liminf_{\rho\rightarrow 0^+}\lim_{h\rightarrow+\infty}\int_{B_{\rho}(p_i)} \vert 
J v_{k_h}\vert ~dx\geq \rilP(\gamma^i)
\qquad \forall i=1,\dots,m. 
\end{equation}
By definition \eqref{Plateau_rel}, using \eqref{Plateaurescaled} and \eqref{converg on circ}, we readily conclude that
$$\lim_{h\rightarrow+\infty}\int_{B_{\rho}(p_i)} 
\vert J v_{k_h}\vert ~dx\geq \rilP(\gamma_\rho^i),$$
where $\gamma_\rho^i$ is defined 
in Lemma \ref{lemma_palline}. Then, 
since $\gamma_\rho^i$ converge to $\gamma^i$ strictly $BV(\Suno;\R^2)$ as 
$\rho \to 0^+$,
\eqref{eq:liminf_liminf} follows, thanks to Lemma \ref{lemma_palline} and Corollary \ref{cor_continuity_P2}.

\medskip

Upper bound: Fix $r>0$ small enough and consider 
mutually disjoint balls $B_{r}(p_i) \subset \Om$, $i=1,\dots,m$, such that, for every $\indiceaffine\in\{1,\ldots,n\}$, $J_\indiceaffine\cap\partial B_s(p_i)$, if nonempty, consists either of a single point, or of two points if $\alpha_\indiceaffine(a_\indiceaffine)=\alpha_\indiceaffine(b_\indiceaffine)=p_i$, for every $s \in (0,r]$.
 
Clearly, the difficulty of the proof is concentrated around the 
junction points $p_i$. The idea is to modify $u$ on $\cup_{i=1}^mB_{r}(p_i)$ by constructing a new map $u_r$ (see \eqref{def ur 1} and \eqref{def ur 2}), which coincides with $u$ out of $\cup_{i=1}^mB_{r}(p_i)$ and converges to $u$ strictly $BV(\Om;\R^2)$ as $r$ 
tends to $0^+$.
 The map $u_r$ will be again a piecewise Lipschitz map 
with the same set $\{p_i\}$ 
of junction points,
but different jump set $\jumpset_r$, 
with $\jumpset_r\cap B_{r/2}(p_i)$ made of segments, i.e. $u_r$ is 
of the form  \eqref{def of u} in $B_{r/2}(p_i)$. 
The difficult point will be to provide that $\jumpset_r$ is still a union of (pairwise disjoint up to the endpoints) $C^2$-curves $\widehat\alpha_\indiceaffine$, in particular that each on e
hits $\partial B_{r/2}(p_i)$ with vanishing second derivative. 
At the end, we will apply Theorem \ref{area n-ple point} to $u_r$ in $\cup_{i=1}^mB_{r/2}(p_i)$ and Corollary \ref{cor_multicurves} to $u_r$ in $\Om\setminus(\cup_{i=1}^mB_{r/2}(p_i))$, and
conclude by lower semicontinuity of $\overline{\mathcal{A}}_{BV}(\cdot, \Omega)$. \\
We start by considering a smooth strictly increasing 
 surjective function $\psi_r:[\frac{r}{2},+\infty)\rightarrow [0,+\infty)$ with \footnote{The exponent 
must be chosen greater than $2$ in order to ensure \eqref{vanishing acceleration}.}
\begin{align}\label{psi}
\psi_r(\rho)=\rho\quad \forall\rho\geq r,\quad \psi_r(\rho)=\left(\rho-\frac r2\right)^3 \mbox{in a right neighborhood of } \frac r2,\quad|\psi_r'|
\leq C\quad {\rm in}~ \left(\frac r2,r\right)
\end{align}
 with $C>0$ independent of $r$. 
We define the radial map $\Phi_r: \R^2\setminus B_{\frac r2}(0)\rightarrow\R^2\setminus\{0\}$ as $$\Phi_r(x)=\psi_r(|x|)
\frac{x}{|x|},$$ 
whose inverse is $\Phi_r^{-1}(y)=f_r(|y|)\frac{y}{|y|}$, where
 $f_r:=\psi_r^{-1}$,
and set
\begin{align}
\widehat u_r(x):=u(p_i+\Phi_r(x-p_i))
\qquad\text{ for }x\in B_r(p_i)\setminus \overline B_{\frac r2}(p_i),\;i=1,\dots,m.
\end{align}
The jump set of $\widehat u_r$ in $B_{r}(p_i)\setminus B_{r/2}(p_i)$ is 
parametrized by the curves 
\begin{align}\label{def alfa capp}
\widehat\alpha_\indiceaffine
:=p_i+\Phi_r^{-1}(\alpha_\indiceaffine-p_i)\qquad\forall\indiceaffine=1,\dots,n.
\end{align}
 Notice carefully that $\widehat\alpha_\indiceaffine$ is parametrized on the same parameter interval of $\alpha_\indiceaffine$, but this is 
not an arc length parametrization
for $\widehat\alpha_\indiceaffine$.
Moreover, thanks to the regularity of $\Phi_r$, 
the map
\begin{equation}\label{def ur 1}
u_r:=\begin{cases}
		u&\text{ in }\Om\setminus (\cup_{i=1}^m B_r(p_i))\\
           \widehat u_r&\text{ in }B_r(p_i)\setminus B_{\frac r2}(p_i)

\end{cases}
\end{equation}
has jump set $\jumpset_r$ which is parametrized by the curves $\widehat \alpha_\indiceaffine$, 
whose supports $\widehat J_\indiceaffine$ are pairwise disjoint and in turn coincide with the ones of $\alpha_\indiceaffine$ in $\Om\setminus (\cup_{i=1}^m B_r(p_i))$. \\

\smallskip
\textit{Step 1}: Let us first check that the length of $\widehat \alpha_\indiceaffine$ in $\cup_{i=1}^m( B_r(p_i)\setminus B_{r/2}(p_i))$ is controlled, more precisely, we will show that for each $i$
and $\indiceaffine$, the length of $\widehat\alpha_\indiceaffine$ in $B_r(p_i)\setminus B_{r/2}(p_i)$ 
goes to $0$ as $r \to 0^+$. We 
suppose that $J_\indiceaffine\cap\partial B_s(p_i)$, for every $s\leq r$, consists  of a single point, 
because the 
argument adapts also if $\alpha_\indiceaffine$ has two arcs exiting from $p_i$, simply by considering them separately.
To this aim, fix $i$ and $\indiceaffine$ and denote $ \alpha_\indiceaffine=\alpha$, $J_\indiceaffine=J$.
 Without loss of generality, assume $p_i=0$, $B_r(0)=B_r$, and suppose that $J\cap B_r$ is parametrized by arc length on $[0,R]$, with $\alpha(0)=0$ and $\alpha(R)\in\partial B_r$, 
where $R(r) = R =\mathcal{H}^1(J\cap B_r)$.
We can express the gradient of $\Phi_r^{-1}$ 
as follows:
\begin{align}\label{grad Phi}
\nabla \Phi_r^{-1}(y)=f'_r(|y|)\frac{y}{|y|}\otimes\frac{y}{|y|}+f_r(|y|)\nabla\left(\frac{y}{|y|}\right)=f'_r(|y|)\frac{y}{|y|}\otimes\frac{y}{|y|}+\frac{f_r(|y|)}{|y|}\Pi(y),
\end{align}
where 
$$
\Pi(y):=\mathrm{Id}-\frac{y\otimes y}{|y|^2},
$$
and we used that 
\begin{align}\label{nabla vortex}
\nabla\left(\frac{y}{|y|}\right)=\frac{1}{|y|}\Pi(y).
\end{align}
From \eqref{def alfa capp}, we have $\dot{\widehat \alpha}=\nabla \Phi_r^{-1}(\alpha)\dot\alpha$, and using \eqref{grad Phi} and 
$|\dot\alpha|=1$,
\begin{align}\label{derivata alfa cap}
|\dot{\widehat \alpha}|\leq f'_r(|\alpha|)+\frac{f_r(|\alpha|)}{|\alpha|}\left|\Pi(\alpha)\dot\alpha\right|.
\end{align}
Notice that if $r$ is small, the function $t\mapsto|\alpha(t)|=:\sigma(t)$ is $C^1$ and invertible from $[0,R]$ to $[0,r]$. 
Moreover, $\sigma'(t)=\frac{\alpha(t)}{|\alpha(t)|}\cdot\dot\alpha(t)\rightarrow\frac{\dot\alpha(0)}{|\dot\alpha(0)|}\cdot\dot\alpha(0)=|\dot\alpha(0)|=1$ as $t \to 0^+$. 
Let us integrate on $[0,R]$ the term $f'_r(|\alpha|)$: performing the change of variable $\sigma(t)=\rho$, we get
$$
\int_0^R f'_r(|\alpha(t)|)dt=\int_0^R f_r'\left(\sigma(t)\right)dt=\int_0^rf_r'(\rho)\frac{d\rho}{\sigma'(\sigma^{-1}(\rho))}\leq2\int_0^rf'_r(\rho)d\rho,
$$
where in the last inequality we used that, for small $r$, $\sigma'(\sigma^{-1}(\rho))\geq\frac12$ for every $\rho\in[0,r]$. 
Sending $r$ to $0^+$, 
we have that $\int_0^R f'_r(|\alpha(t)|)dt\rightarrow 0$ by integrability of $f'$ near to the origin.

In order to estimate the second term on the right hand side of \eqref{derivata alfa cap}, we can use a Taylor expansion of $\alpha$ around $0$, writing $\alpha(t)=vt+w{t^2}+o(t^2)$, with $v=\dot\alpha(0)$, $w=\frac{\ddot\alpha(0)}{2}$, and $\lim_{t\rightarrow 0^+}o(t^p)/t^p=0$. We have 
\begin{align*}
\Pi(\alpha)\dot\alpha=\Pi(vt+w{t^2}+o(t^2))(v+2wt+o_2(t))=\Pi(v+w{t}+o_1(t))(v+2wt+o_2(t)),
\end{align*}
where $o_1(t)=o(t^2)/t$ and $o_2(t)=o(t)$. Writing
$v+2wt+o_2(t)=v+wt+o_1(t)+wt+o_2(t)-o_1(t)$,
we get
\begin{align*}
\Pi(\alpha)\dot\alpha&=\Pi(v+w{t}+o_1(t))(v+wt+o_1(t))+\Pi(v+w{t}+o_1(t))(w{t}+o_2(t)-o_1(t)).
\end{align*}
The first term on the right hand side is $0$ and the norm of the second term can be estimated from above by $|w|t+o(t)$. Now, by definition of arc length parameter, $R=\mathcal{H}^1(\mbox{spt}\alpha\cap B_r(0))
\rightarrow 0$ as $r\rightarrow 0^+$.
Moreover, by Taylor expansion,
$
|\alpha(t)|>\frac{t}{2}\mbox{ for }t \mbox{ small enough}.
$
Therefore, since $f_r(0)=\frac{r}{2}$,  for $r$ small enough we have $\frac{f_r(|\alpha(t)|)}{|\alpha(t)|}\leq\frac{2r}{t}$ on $[0,R].$ So, integrating on $[0,R]$ the second term on
 the right hand side of \eqref{derivata alfa cap}, 
$$
\int_0^R\frac{f_r(|\alpha(t)|)}{|\alpha(t)|}\left|\Pi(\alpha(t))\dot\alpha(t)\right|dt\leq\int_0^R\frac{2r}{t}(|w|t+o(t))dt\rightarrow0 \qquad{\mbox{as }}r\rightarrow0^+.
$$
\smallskip
\textit{Step 2}: Let $\widehat J=\widehat J_l$ be the support of $\widehat \alpha$; let us show that there is a parametrization of $\widehat J\cap (B_r\setminus B_{r/2})$ on an interval $[0,L]$, which is of class $C^2$ up to $0$ and with vanishing second derivative at $0$. Indeed, set $L:=\mathcal{H}^1(\widehat J \cap (B_r\setminus B_{r/2}))$ and consider the arc-length parameter $s\in [0,L]$ given by 
$$s(t)=\int_0^t|V_r(\alpha(\tau))|d\tau,$$
where
$$
V_r(\alpha):=\nabla \Phi_r^{-1}(\alpha)\dot\alpha.
$$
We compute 
\begin{align}\label{curvatura}
	&\frac{d^2}{ds^2}\widehat \alpha(t)=\frac{d}{ds}\left(\frac{V_r(\alpha)}{|V_r(\alpha)|}\right)=\Pi(V_r(\alpha))
\left(\frac{\nabla^2 \Phi_r^{-1}(\alpha):(\dot\alpha\otimes\dot\alpha)+\nabla\Phi_r^{-1}(\alpha)\ddot\alpha}{|V_r(\alpha)|^{2}}\right).
\end{align}
Here and in what follows, $\alpha$ is evaluated at $t=t(s)$ and $\dot\alpha$ and $\ddot\alpha$ denote the first and second derivative of $\alpha$ with respect to $t$. The operation $:$ between a tensor $T=(T_{ijk})\in\R^{2\times2\times2}$ and a matrix $M=(M_{ij})\in\R^{2\times2}$ is defined as the vector $T:M\in\R^2$ with components $(T:M)_k=T_{ijk}M_{ij}$
for $k=1,2$.\\
We get
\begin{align}\label{second derivative}
	\left|\frac{d^2}{ds^2}\widehat \alpha(t)\right|
&\leq \left|\Pi(V_r(\alpha))\left(
\frac{\nabla^2 \Phi_r^{-1}(\alpha):(\dot\alpha\otimes\dot\alpha)}{|V_r(\alpha)|^{2}}\right)\right|
+\frac{|\nabla\Phi_r^{-1}(\alpha)\ddot\alpha|}{|V_r(\alpha)|^2}\nonumber\\
&\leq \left|\Pi(V_r(\alpha))\left(
\frac{\nabla^2 \Phi_r^{-1}(\alpha):(\dot\alpha\otimes\dot\alpha)}{|V_r(\alpha)|^{2}}\right)\right|
+C\frac{{f'_r(|\alpha|)}+\frac{f_r(|\alpha|)}{|\alpha|}}{|V_r(\alpha)|^2}.
\end{align}
where we have used \eqref{grad Phi} and that $\ddot\alpha$ is bounded.\\
The Hessian of $\Phi_r^{-1}$ can be computed as
\begin{align*}
\nabla^2\Phi_r^{-1}(y)=&f''_r(|y|)\frac{y}{|y|}\otimes\frac{y}{|y|}\otimes\frac{y}{|y|}+f'_r(|y|)\nabla\left(\frac{y}{|y|}\otimes\frac{y}{|y|}\right)+\\
&+f'_r(|y|)\frac{y}{|y|}\otimes\nabla\left(\frac{y}{|y|}\right)+f_r(|y|)\nabla^2\left(\frac{y}{|y|}\right)\\
=&f''_r(|y|)\frac{y}{|y|}\otimes\frac{y}{|y|}\otimes\frac{y}{|y|}+f'_r(|y|)\nabla\left(\frac{y}{|y|}\right)\otimes\frac{y}{|y|}+\\
&+2f'_r(|y|)\frac{y}{|y|}\otimes\nabla\left(\frac{y}{|y|}\right)+f_r(|y|)\nabla\left(\nabla\left(\frac{y}{|y|}\right)\right).
\end{align*}
Then, by \eqref{nabla vortex}, we have
\begin{align*}
\nabla^2\Phi_r^{-1}(\alpha)=&f''_r(|\alpha|)\frac{\alpha}{|\alpha|}\!\otimes\frac{\alpha}{|\alpha|}\otimes\!\frac{\alpha}{|\alpha|}
+\left(\frac{f'_r(|\alpha|)}{|\alpha|}-2\frac{f_r(|\alpha|)}{|\alpha|^2}\right)\Pi(\alpha)\otimes\frac{\alpha}{|\alpha|}\\
&+\left(2\frac{f'_r(|\alpha|)}{|\alpha|}-\frac{f_r(|\alpha|)}{|\alpha|^2}\right)\frac{\alpha}{|\alpha|}\otimes \Pi(\alpha).
\end{align*}
So, for $k=1,2$, we have
\begin{align}
&\left(\nabla^2\Phi_r^{-1}(\alpha):(\dot\alpha\otimes\dot\alpha)\right)_k\nonumber\\
=& f''_r(|\alpha|)\left(\left(\frac{\alpha}{|\alpha|}\otimes\frac{\alpha}{|\alpha|}\otimes\frac{\alpha}{|\alpha|}\right):(\dot\alpha\otimes\dot\alpha)\right)_k\nonumber\\
&+\left(\frac{f'_r(|\alpha|)}{|\alpha|}-2\frac{f_r(|\alpha|)}{|\alpha|^2}\right)\left(\left(\Pi(\alpha)\otimes\frac{\alpha}{|\alpha|}\right):(\dot\alpha\otimes\dot\alpha)\right)_k\label{term1}\\
&+\left(2\frac{f'_r(|\alpha|)}{|\alpha|}-\frac{f_r(|\alpha|)}{|\alpha|^2}\right)\left(\left(\frac{\alpha}{|\alpha|}\otimes \Pi(\alpha)\right):(\dot\alpha\otimes\dot\alpha)\right)_k.\label{term2}
\end{align}
Notice that, since $\Pi(\alpha)$ is symmetric,
\begin{align}\label{property of A}
\Pi(\alpha)_{ij}\alpha_j=0,\quad \Pi(\alpha)_{ij}\alpha_i=0,
\end{align}
where we sum on repeated indeces.
So, using \eqref{property of A} and that, from Taylor expansion, $\dot\alpha(t)=v+2wt+o(t)=\frac{\alpha(t)}{t}+wt+o(t)$, we have
\begin{align*}
\left(\left(\Pi(\alpha)\otimes\frac{\alpha}{|\alpha|}\right):(\dot\alpha\otimes\dot\alpha)\right)_k&=\Pi(\alpha)_{ij}\dot\alpha_i\dot\alpha_j\frac{\alpha_k}{|\alpha|}=\Pi(\alpha)_{ij}\left(\frac{\alpha_i}{t}+w_it+o(t)\right)\dot\alpha_j\frac{\alpha_k}{|\alpha|}=\\
&=\Pi(\alpha)_{ij}\left(w_it+o(t)\right)\dot\alpha_j\frac{\alpha_k}{|\alpha|};
\end{align*}
\begin{align*}
\left(\left(\frac{\alpha}{|\alpha|}\otimes \Pi(\alpha)\right):(\dot\alpha\otimes\dot\alpha)\right)_k&=\frac{\alpha_i}{|\alpha|}\Pi(\alpha)_{jk}\dot\alpha_i\dot\alpha_j=\frac{\alpha_i}{|\alpha|}\Pi(\alpha)_{jk}\left(\frac{\alpha_j}{t}+w_jt+o(t)\right)\dot\alpha_i\\
&=\frac{\alpha_i}{|\alpha|}\Pi(\alpha)_{jk}\left(w_jt+o(t)\right)\dot\alpha_i.
\end{align*}
So, the norm of the sum of \eqref{term1} and \eqref{term2} can be easily estimated by
$$
3\left(\frac{f'_r(|\alpha|)}{|\alpha|}+\frac{f_r(|\alpha|)}{|\alpha|^2}\right)(|w|t+o(t))\leq C\left({f'_r(|\alpha|)}+\frac{f_r(|\alpha|)}{|\alpha|}\right),
$$ 
where we used that, for $t$ small, $|\alpha(t)|\geq\frac{t}{2}$.\\
Therefore, \eqref{second derivative} becomes
\begin{align}\label{curvature 2}
	\left|\frac{d^2}{ds^2}\widehat \alpha(t)\right|
\leq& \left|f''_r(|\alpha|)\right|\left|\Pi(V_r(\alpha))\left(
\frac{\frac{\alpha}{|\alpha|}\otimes\frac{\alpha}{|\alpha|}\otimes\frac{\alpha}{|\alpha|}:(\dot\alpha\otimes\dot\alpha)}{|V_r(\alpha)|^{2}}\right)\right|+C\frac{{f'_r(|\alpha|)}+\frac{f_r(|\alpha|)}{|\alpha|}}{|V_r(\alpha)|^2}.
\end{align}
Now we treat the first term of the right hand side of \eqref{curvature 2}.
For $j=1,2$, by definition of $V_r(\alpha)$, using Taylor expansion and \eqref{property of A}, we have
\begin{equation}\label{V_r}
\begin{aligned}
(V_r)_j(\alpha)&=f'_r(|\alpha|)\frac{\alpha_i\alpha_j}{|\alpha|^2}\dot\alpha_i+f_r(|\alpha|)\Pi(\alpha)_{ij}\dot\alpha_i\\
&=f'_r(|\alpha|)\frac{\alpha_i\alpha_j}{|\alpha|^2}\left(\frac{\alpha_i}{t}+w_it+o(t)\right)+f_r(|\alpha|)\Pi(\alpha)_{ij}\left(\frac{\alpha_i}{t}+w_it+o(t)\right)\\
&=f'_r(|\alpha|)\left(\frac{\alpha_j}{t}+\frac{\alpha_i\alpha_j}{|\alpha|^2}w_it+o(t)\right)+f_r(|\alpha|)\Pi(\alpha)_{ij}\left(w_it+o(t)\right)\\
&=f'_r(|\alpha|)\left(\frac{\alpha_j}{t}+o(t)\right)+f_r(|\alpha|)O_j(t),
\end{aligned}
\end{equation}
where in the last equality we used that $\alpha_i w_i=o(t)$, since $v_iw_i=0$ because $|\dot\alpha|=1$, and we setted $O_j(t):=\Pi(\alpha)_{ij}(w_it+o(t))$, meaning that $\lim_{t\rightarrow0^+}|O_j(t)|/t<+\infty$.
Then, we get
$$
\alpha=t\left(\frac{V_r(\alpha)-O(t)}{f'_r(|\alpha|)}+o(t)\right).
$$
So, 
\begin{align*}
\Pi(V_r(\alpha))
\frac{\frac{\alpha}{|\alpha|}\otimes\frac{\alpha}{|\alpha|}\otimes\frac{\alpha}{|\alpha|}:(\dot\alpha\otimes\dot\alpha)}{|V_r(\alpha)|^{2}}=&\frac{\alpha_i\alpha_j}{|\alpha|^2}\dot\alpha_i\dot\alpha_j\Pi(V_r(\alpha))
\frac{\frac{\alpha}{|\alpha|}}{|V_r(\alpha)|^{2}}\\
=&\frac{\alpha_i\alpha_j}{|\alpha|^2}\dot\alpha_i\dot\alpha_j\frac{t}{|\alpha|}\Pi(V_r(\alpha))
\frac{\left(\frac{V_r(\alpha)-O(t)}{f'_r(|\alpha|)}+o(t)\right)}{|V_r(\alpha)|^{2}}\\
=&\frac{\alpha_i\alpha_j}{|\alpha|^2}\dot\alpha_i\dot\alpha_j\frac{t}{|\alpha|}\Pi(V_r(\alpha))
\frac{\left(\frac{O(t)}{f'_r(|\alpha|)}+o(t)\right)}{|V_r(\alpha)|^{2}},
\end{align*}
where we used that 
$\Pi(V_r(\alpha))V_r(\alpha)=0.$ 
For $t$ small, we get
$$
\left|\Pi(V_r(\alpha))
\frac{\frac{\alpha}{|\alpha|}\otimes\frac{\alpha}{|\alpha|}\otimes\frac{\alpha}{|\alpha|}:(\dot\alpha\otimes\dot\alpha)}{|V_r(\alpha)|^{2}}\right|\leq 2\frac{\frac{O(t)}{f'_r(|\alpha|)}+o(t)}{|V_r(\alpha)|^2}.
$$
Finally, from \eqref{curvature 2}, we obtain
$$
\left|\frac{d^2}{ds^2}\widehat \alpha(t)\right|\leq\left|f''_r(|\alpha|)\right|\frac{\frac{O(t)}{f'_r(|\alpha|)}+o(t)}{|V_r(\alpha)|^2}
+C\frac{{f'_r(|\alpha|)}+\frac{f_r(|\alpha|)}{|\alpha|}}{|V_r(\alpha)|^2}.
$$
From the definition of $f_r$, we have that $f_r(|\alpha(t)|)= \frac{r}{2}+t^\frac{1}{3}+o(t^\frac{1}{3})$ for $t$ near to $0$. So, by \eqref{V_r}, we have $|V_r(\alpha(t))|\geq Cf'_r(|\alpha(t)|)=Ct^{-\frac23}+o(t^{-\frac23})$. Then, since $|f_r''(|\alpha(t)|)|=Ct^{-\frac53}+o(t^{-\frac53})$, a straightforward check shows that 
\begin{align}\label{vanishing acceleration}
\frac{d^2}{ds^2}\widehat \alpha(t)\rightarrow 0\quad\mbox{as }t\rightarrow0^+.
\end{align}

We conclude that  the curve $\widehat \alpha$ is $C^2$ up to $0$ with vanishing second derivative, and hence can be extended on the  interval $(-\frac r2,0)$ to a (not relabeled) curve $\widehat \alpha$ whose support is a straight segment connecting $\widehat\alpha(0)$ to $0$ (namely a radius of $B_{r/2}(0)$).
Going back to the curves $\widehat\alpha_\indiceaffine$, we have just proved that we can extend them in $B_{ r/2}(p_i)$ with  $C^2$-regularity using a segment along a radius, reaching $p_i$. In particular, the new supports of $\widehat \alpha_\indiceaffine$'s form a $N^i$-junction point around $p_i$ in $ B_{ r/2}(p_i)$, whose 
circular sectors $\widehat C^i_j$ ($j=1,\dots, N_i$) have amplitudes $\theta_i^1,\dots, \theta_i^{N_i}$ (according to Lemma \ref{lemma_palline}). Up to a reparametrization by arc-length of $\widehat \alpha_\indiceaffine$, we will suppose that 
$\widehat\alpha_\indiceaffine:[\widehat a_\indiceaffine,\widehat b_\indiceaffine]\rightarrow\R^2$ have always derivative of modulus $1$.

\smallskip
\textit{Step 3}: We are ready to extend the map $u_r$ in $B_{r/2}{(p_i)}$.
We eventually observe that, from \eqref{def ur 1}, $u_r(x)=\gamma^i\left(\frac{2}{r}(x-p_i)\right)$ on $\partial B_{ r/2}(p_i)$ (see Lemma \ref{lemma_palline}), and hence it is constant on any arc with angular coordinate in $(a_i^{j-1},a_i^{j})$. Hence we define
\begin{align}\label{def ur 2}
u_r(x):=\gamma^i\left(\frac{x-p_i}{|x-p_i|}\right)\qquad\qquad x\in B_{\frac r2}(p_i).
\end{align}
Now, $u_r$ satisfies the hypotheses of Corollary \ref{cor_multicurves} 
in $\Om_r:=\Om\setminus (\cup_{i=1}^m \overline B_{r/4}(p_i))$, where all the curves $\widehat \alpha_j$ satisfy hypotheses (H3), and they run on a straight segment (along a radius of $B_{r/2}(p_i)$) inside $B_{r/2}(p_i)\setminus B_{r/4}(p_i)$.
Then we introduce a sequence of Lipschitz maps $\widetilde v_k:\Om_r\rightarrow\R^2$ which are defined as in \eqref{v_k}, where, we recall, $\eps=\frac1k$, with $u_r$ in place of $u$ and $\Lambda=$ id; in particular, for $k$ large enough, the trace of $\widetilde v_k$ on $\partial B_{r/3}(p_i)$ is a piecewise affine map coinciding with $\gamma_k$ in \eqref{gamma_k}, with $\beta_i$ in place of $\alpha_i$. Thus, if we introduce also the sequence of Lipschitz maps $\widehat v_k:B_{r/2}(p_i)\rightarrow\R^2$ as in \eqref{def_rec_npoint} (with $B_r$ replaced by $B_{r/2}(p_i)$) we see that 
$\widetilde v_k=\widehat v_k$ on $\partial B_{r/3}(p_i)$.
Therefore we define
\begin{align}
v^r_k:=\begin{cases}
	\widetilde v_k&\text{ in }\Om\setminus(\cup_{i=1}^m B_{r/3}(p_i)) \\
	\widehat v_k&\text{ in }\cup_{i=1}^m B_{r/3}(p_i),
\end{cases}
\end{align}
and we readily see that $v^r_k\rightarrow u_r$ strictly $BV(\Om;\R^2)$.\\
Since the supports of $\alpha_\indiceaffine$ and $\widehat\alpha_\indiceaffine$ coincide out of $\cup_iB_r(p_i)$, there exist $\widehat a_\indiceaffine^r,\widehat b_\indiceaffine^r\in[\widehat a_\indiceaffine,\widehat b_\indiceaffine]$ and $ a_\indiceaffine^r, b_\indiceaffine^r\in[ a_\indiceaffine, b_\indiceaffine]$ such that
$$\widehat\alpha_\indiceaffine([\widehat a_\indiceaffine^r,\widehat b_\indiceaffine^r])=\alpha_\indiceaffine([a_\indiceaffine^r,b_\indiceaffine^r]),\quad\widehat\alpha_\indiceaffine(\widehat a_\indiceaffine^r)=\alpha_\indiceaffine(a_\indiceaffine^r),\quad\widehat\alpha_\indiceaffine(\widehat b_\indiceaffine^r)=\alpha_\indiceaffine(b_\indiceaffine^r).$$
In particular, $\widehat b_\indiceaffine^r-\widehat a_\indiceaffine^r= b_\indiceaffine^r- a_\indiceaffine^r$, so up to a translation of the parameter interval of $[\widehat a_\indiceaffine, \widehat b_\indiceaffine]$, we can suppose $\widehat a_\indiceaffine^r= a_\indiceaffine^r$ and $\widehat b_\indiceaffine^r=b_\indiceaffine^r$. Clearly, $a_\indiceaffine^r\rightarrow a_\indiceaffine$ non increasingly and $ b_\indiceaffine^r\rightarrow b_\indiceaffine$ non decreasingly as $r\rightarrow0^+$.\\
In view of Corollary \ref{cor_multicurves} and Theorem \ref{area n-ple point} 
we conclude 
\begin{align}\label{lsc_bis}
	\overline{\mathcal A}_{BV}(u_r, \Om)&\leq \lim_{k\rightarrow +\infty}\mathcal A(v_k^r, \Om)=\int_{\Omega\setminus(\cup_{i=1}^mB_{r}(p_i)) }|\mathcal{M}(\nabla u)|~dx+\sum_{\indiceaffine=1}^n\int_{[ \widehat a_\indiceaffine, \widehat b_\indiceaffine]\times
		\intervallounitario
	}|\partial_t
	\affinesurface_{\indiceaffine,r}
	\wedge\partial_s\affinesurface_{\indiceaffine,r}|~dtds\nonumber\\
	&\;\;\;+\int_{\cup_{i=1}^m(B_{r}(p_i)\setminus B_{r/3}(p_i))}|\mathcal{M}(\nabla u_r)|~dx
	+m\frac{\pi r^2}{9}+\sum_{i=1}^m\rilP(\gamma^i)\nonumber
	\\
	&=\int_{\Omega\setminus(\cup_{i=1}^mB_{r}(p_i)) }|\mathcal{M}(\nabla u)|~dx+\sum_{\indiceaffine=1}^n\int_{[a^r_\indiceaffine, b^r_\indiceaffine]\times
		\intervallounitario
	}|\partial_t
	\affinesurface_{\indiceaffine}
	\wedge\partial_s\affinesurface_{\indiceaffine}|~dtds+\sum_{i=1}^m\rilP(\gamma^i)\nonumber\\
	&\;\;\;+\int_{\cup_{i=1}^m(B_{r}(p_i)\setminus B_{r/3}(p_i))}|\mathcal{M}(\nabla u_r)|~dx
	+\sum_{\indiceaffine=1}^n\int_{([\widehat a_\indiceaffine^{r/3}, a_\indiceaffine^r]\cup [ b_\indiceaffine^{r},\widehat b_\indiceaffine^{r/3}])\times
		\intervallounitario
	}|\partial_t
	\affinesurface_{\indiceaffine,r}
	\wedge\partial_s\affinesurface_{\indiceaffine,r}|~dtds\nonumber\\
	&\;\;\;+\frac{r}{3}\sum_{i=1}^m\sum_{j=1}^{N_i}|\beta_i^j-\beta_i^{j+1}|+m\frac{\pi r^2}{9},
\end{align} 
where for all $\indiceaffine=1,\dots,n$ we have $\widehat a_\indiceaffine\leq\widehat a_\indiceaffine^{r/3}\leq a_\indiceaffine^r<b_\indiceaffine^r\leq\widehat b_\indiceaffine^{r/3}\leq\widehat b_\indiceaffine$, where $ \widehat\alpha_\indiceaffine(\widehat a_\indiceaffine^{\frac r3})\in\partial B_{r/3}(p_i)$, $\widehat\alpha_\indiceaffine(\widehat b_\indiceaffine^{\frac r3})\in\partial B_{r/3}(p_j)$ for some $i,j\in\{1,\ldots,m\}$, unless one of them belongs to $\partial\Om$, and where $\affinesurface_{\indiceaffine,r}$ is defined as $\affinesurface_{\indiceaffine}$ with $u_r$ replacing $u$.

Now, since by \eqref{psi} $|\psi'_r|\leq C$, $u_r$ is still a piecewise Lipschitz 
map on $\Omega$, hence, by \textit{Step 1}, the last four terms in \eqref{lsc} are negligible as $r\rightarrow 0^+$. 
We then conclude, provided that $u_r\rightarrow u$ strictly $BV(\Om;\R^2)$, that
\begin{align*}
		\overline{\mathcal A}_{BV}(u, \Om)\leq\liminf_{r\rightarrow 0^+}\overline{\mathcal A}_{BV}(u_r, \Om)=\int_{\Omega}|\mathcal{M}(\nabla u)|~dx+\sum_{\indiceaffine=1}^n\int_{[a_\indiceaffine, b_\indiceaffine]\times
			\intervallounitario
		}|\partial_t
		\affinesurface_{\indiceaffine}
		\wedge\partial_s\affinesurface_{\indiceaffine}|~dtds+\sum_{i=1}^m\rilP(\gamma^i),
\end{align*} 
that is the thesis.
In order to check that $u_r\rightarrow u$ strictly $BV(\Om;\R^2)$ it is sufficient to observe that $u=u_r$ outside $\cup_{i=1}^m B_r(p_i)$ and that 
\begin{equation*}
\begin{aligned}
& \limsup_{r\rightarrow 0^+}|D u_r|({\cup_{i=1}^m 
B_r(p_i)})
\\
\leq&  
\limsup_{r\rightarrow 0^+}
\limsup_{k\rightarrow +\infty}\int_{\cup_{i=1}^m B_r(p_i)}\sqrt{1+|\nabla v_k^r|^2}~dx
\\
\leq&  
\limsup_{r\rightarrow 0^+}
\lim_{k\rightarrow +\infty}\mathcal{A}(v_k^r;{\cup_{i=1}^m B_r(p_i)})
\\
 =& \limsup_{r\rightarrow 0^+}\Big( \int_{\cup_{i=1}^m(B_{r}(p_i)\setminus B_{r/3}(p_i))}|\mathcal{M}(\nabla u_r)|~dx
	+m\frac{\pi r^2}{9}
\\
&+ \sum_{\indiceaffine=1}^n\int_{([\widehat a_\indiceaffine^{r/3},\widehat a_\indiceaffine^r]\cup [\widehat b_\indiceaffine^{r},\widehat b_\indiceaffine^{r/3}])\times
		\intervallounitario
	}|\partial_t
	\affinesurface_{\indiceaffine,r}
	\wedge\partial_s\affinesurface_{\indiceaffine,r}|~dtds+\frac{r}{3}\sum_{i=1}^m\sum_{j=1}^{N_i}|\alpha^i_j-\alpha^i_{j+1}|\Big)=0.
\end{aligned}
\end{equation*}
The proof is complete.
\end{proof}

\textsc{Acknowledgements:}
We thank Domenico Mucci for stimulating discussions
and suggestions.
The authors are members of the Gruppo Nazionale per l'Analisi Matematica, la Probabilit\`a e le loro Applicazioni (GNAMPA) of the INdAM of Italy. RS also acknowledges the partial financial support of the F-CUR project number $\textrm{2262-2022-SR-CONRICMIUR\_PC-FCUR202\_002}$  of the University of Siena.


\end{document}